\documentclass[reqno]{amsart}

\usepackage[margin=1.5in]{geometry}

\usepackage{amsmath,amsthm,amscd,amssymb}
\usepackage{latexsym}
\usepackage{color}

\usepackage{bbm}

\usepackage{todonotes}
\usepackage[all]{xy}

\usepackage{calrsfs}
\DeclareMathAlphabet{\pazocal}{OMS}{zplm}{m}{n}

\numberwithin{equation}{section}

\newtheorem{Thm}{Theorem}[section]

\newtheorem{Lem}[Thm]{Lemma}
\newtheorem{Prop}[Thm]{Proposition}
\newtheorem{Def}[Thm]{Definition}
\newtheorem{Rmk}[Thm]{Remark}

\newtheorem{Cond}[Thm]{Condition}
\newtheorem{Exam}[Thm]{Example}

\newcommand{\e}{\varepsilon}

\newcommand{\eps}{\theta}
\newcommand{\epp}{\epsilon}

\renewcommand{\L}{\Lambda}

\newcommand{\N}{\mathbb{N}}
\newcommand{\Z}{\mathbb{Z}}
\renewcommand{\P}{\mathbb{P}}
\newcommand{\R}{\mathbb{R}}
\newcommand{\E}{\mathbb{E}}
\newcommand{\T}{\mathbb{T}}

\newcommand{\V}{{\rm Var}}

\newcommand{\m}{\mf{m}}
\newcommand{\D}{\mf{D}}

\newcommand{\wpi}{\widehat \pi}

\newcommand{\Mb}{\pazocal{M}}

\newcommand{\Qb}{\pazocal{Q}}
\newcommand{\Pam}{\mathcal{R}}
\newcommand{\Pb}{\pazocal{P}}

\newcommand{\Ub}{\pazocal{U}}

\newcommand{\id}{\mathbbm{1}}

\newcommand{\mf}[1]{{\mathfrak #1}}

\newcommand{\dsp}{\displaystyle}

\title[Zero-range processes with slow sites] {Condensation, boundary conditions, and effects of slow sites in zero-range systems}

\author{Sunder Sethuraman}
\address{Department of Mathematics, University of Arizona,  Tucson, AZ 85721, USA}
\email{sethuram@math.arizona.edu}
\author{Jianfei Xue}
\address{Department of Mathematics, University of Missouri,  Columbia, MO 65211,USA}
\email{jxue@missouri.edu}

\begin{document}

\begin{abstract}
We consider the space-time scaling limit of the particle mass in zero-range particle systems on a $1$D discrete torus $\Z/N\Z$ with a finite number of defects.   
 We focus on two classes of increasing jump rates $g$, when $g(n)\sim n^\alpha$, for $0<\alpha\leq 1$, and when $g$ is a bounded function.   In such a model, a particle at a regular site $k$ jumps equally likely to a neighbor with rate $g(n)$, depending only on the number of particles $n$ at $k$. 
 At a defect site
 $k_{j,N}$, however, the jump rate is slowed down to $\lambda_j^{-1}N^{-\beta_j}g(n)$ when $g(n)\sim n^\alpha$, and to $\lambda_j^{-1}g(n)$ when $g$ is bounded.  Here, $N$ is a scaling parameter where the grid spacing is seen as $1/N$ and time is speeded up by $N^2$.

Starting from initial measures with $O(N)$ relative entropy with respect to an invariant measure,
we show the hydrodynamic limit and characterize boundary behaviors at the macroscopic defect sites $x_j = \lim_{N\uparrow \infty} k_{j, N}/N$, for all defect strengths.
For rates $g(n)\sim n^\alpha$,
at critical or super-critical slow sites ($\beta_j=\alpha$ or $\beta_j>\alpha$),  associated Dirichlet boundary conditions arise as a result of interactions with evolving atom masses or condensation at the defects.
Differently, when $g$ is bounded, at any slow site ($\lambda_j>1$), we find the hydrodynamic density must be bounded above by a threshold value reflecting the strength of the defect. Moreover, due to interactions with masses of atoms stored at the slow sites,
 the associated boundary conditions bounce between being periodic and Dirichlet.

\end{abstract}

\subjclass[2020]{60K35}

 \keywords{interacting particle system, zero-range, hydrodynamic, boundary condition, defect, inhomogeneity, condensation}

\maketitle

\maketitle

\section{Introduction}
The purpose of this article is to understand the macroscopic boundary conditions which arise in hydrodynamic scaling limits for the space-time `bulk' mass evolution of zero-range processes with a finite number of `defects', and also related effects of `condensation' at these defect locations.  Such an aim more broadly fits into the larger study of how macroscopic boundary conditions emerge from inhomogeneous microscopic interactions.

In this view, there has been much interesting work on
one-dimensional ($1$D) exclusion models where a site or bond is `slowed' down.
In \cite{Franco_aihp}, \cite{Franco_tams}, \cite{Franco_spa}, in computing the hydrodynamic limit in symmetric systems, different boundary conditions from Dirichlet to Neumann, and also Robin have been derived; see also \cite{DePaula}.   
 There are however only a few works with respect to different interactions, in particular zero-range systems, which do not limit the particle numbers at a location.
 Among these, \cite{JLT} studies the hydrodynamic limit for a system of $1$D symmetric independent particles moving in a (random) trap environment.  In \cite{lan_zrp}, with respect to a $1$D totally asymmetric zero-range process, with bounded, increasing rate function $g(\cdot)$, effects of a slow site and a slow particle are found when starting from a `flat' initial measure. 
  In \cite{Saada_hyd} (see also \cite{Saada}), with respect to a class of such $1$D totally asymmetric zero-range systems, however with a nontrivial density of site disorders, hydrodynamics is shown with respect to an effective flux function, constant at supercritical densities; see also the conjectures in Section 3.3 in \cite{Saada_hyd} about slow sites in asymmetric models.
   
We also mention that $1$D systems with `reservoir' boundaries have been studied; see \cite{Olla_reservoir} for a discussion of `hydrostatics', and related references, with respect to exclusion models.
  In zero-range processes, `static' reservoir effects are studied and dynamical conjectures are discussed in \cite{Frometa}.
 
In this context, the general goal of our work is to consider the effect of a finite number of `slow' sites in a class of $1$D symmetric zero-range systems in a torus $\Z/N\Z$. 
In such a model, particles would `condense' on a defect site if jump rates from it are `slow' enough.  One would expect that for the particle continuum mass, different boundary conditions would result depending on how `slow' the defect is.  
Our main results describe corresponding hydrodynamic limits, in terms of a nonlinear parabolic PDE and evolving point masses, with specified boundary conditions at defects reflecting different types of condensation.  To our knowledge, these are the first results for a general class of zero-range interactions, with bounded and unbounded rate functions $g$.  In particular, the behaviors found differ depending on the structure of the rate $g$.  The proof method, in the scheme of the `entropy' method, specifies local `replacements' which may be useful in other problems.

By specifying the locations of the `slow' defects in the system, we fix the macroscopically separated points where `condensation' can occur.  There seems to be little work on the dynamical structure in such systems.
This is in contrast to the well-developed study of  `condensation', which `spontaneously' forms at a random location  by introducing more particles in a zero-range system with a bounded decreasing rate function than is allowed to equilibriate.  
See \cite{Evans_Mukamel} for a discussion of both mechanisms with respect to canonical and grand canonical invariant measures of a bounded rate zero-range process with a single defect.

In passing, we mention, among the recent literature on the `spontaneous' formation of `condensation', \cite{Armendariz} considers, with respect to a thermodynamic limit in $1$D symmetric zero-range models,
the evolution of the random `condensate' in a certain time-scale.
In \cite{Landim_nonrev_condensate}, with respect to $1$D asymmetric dynamics in a set of $L$ fixed sites, motion of the `condensate' is described.   In \cite{Lan_mart}, a `martingale problem' approach for the condensate dynamics is developed.  In \cite{Loulakis}, some partial results on a hydrodynamic limit is given.  See also references therein in these papers for a more complete history of the subject. 
For related notions of `metastability', see books and surveys \cite{Bovier}, \cite{Lan_surv}, \cite{Vares}.

\subsection{Sketch of results} To describe our results, 
we consider zero-range processes on the torus 
${\mathbb T}_N := \{0,1,\ldots, N-1\}$ where the site $N$ (the right neighbor of $N-1$) is identified as $0$.
 The jump rate functions $g: \N_0 \mapsto [0,\infty)$ focused upon satisfy $\lim_{n\uparrow \infty} g(n)/n^\alpha = 1$ (to fix a time-scale) and $\alpha\in [0,1]$. 
Moreover, we classify the $g$'s into two families:  (1)  $g(n) \sim n^\alpha$ when $0< \alpha\leq 1$, and (2) $g$ is bounded.  
 In both settings, we will assume also that $g$ is a Lipschitz, increasing function, which rules out the case $\alpha>1$.  Such an assumption is convenient to bound the rate $g$ in estimates, and also allows `attractive' process couplings; see Section \ref{sec: coupling} 
for a discussion of the use of `attractive' coupling.

In this process, at a regular site $k\in \T_N$, if there are $n$ particles there, one of them leaves with rate $g(n)$ to a neighbor, jumping either to $k-1$ or $k+1$, with equal probability.
However, at a defect site $k$, the departure rate is altered as follows:  Let $\lambda>0$ and $\beta\in \R$.
If there are $n$ particles at the defect $k$, one of them leaves at rate $\tfrac{1}{\lambda N^{\beta}}g(n)$. 
We will say the rate is `slow' when $\beta>0$, or $\beta=0$ and $\lambda>1$.

The zero-range system tracks the evolution of the unlabeled particles on $\T_N$.  We denote by $\xi_t = \{\xi_t(k): k\in \T_N\}$ the configuration of the process, where $\xi_t(k)$ is the number of particles at site $k$ at time $t\geq 0$.  Given the symmetric transitions, it will be useful to define also the speeded-up process $\eta_t = \xi_{N^2t}$.    
For this Markov system, there is a family of product (reversible) invariant measures $\Pam^N_c$, indexed by an interval of `density' parameters $c$; 
see Section \ref{sec: invariant measures}.

For the system with rate $g(n)\sim n^\alpha$, we will start the process from measures $\mu^N$, associated to an initial macroscopic measure $\pi_0$ on the unit torus $\T=[0,1)$, with $O(N)$ relative entropy with respect to an invariant distribution $\Pam^N_{c_0}$, and stochastically bounded by another invariant distribution $\Pam^N_{c'}$.  With respect to bounded rates $g$, $\mu^N$ satisfies a similar but slightly different criteria, since when 
$\|g\|_{\infty}=g_\infty<\infty$, there will be a finite effective critical density above which the invariant measure is not defined;  see Condition \ref{condition on mu N}.

Here, in the setting $g(n)\sim n^\alpha$, the initial profile $\pi_0$ will be in form
$$\pi_0(dx) = \rho_0(x)dx + \sum_{j:\beta_j=\alpha} \m_{0,j}\delta_{x_j}(dx),$$
with a similar formulation in the $g$ bounded setting.
 Examples of suitable initial measures $\mu^N$ are given by local equilibrium product measures; see Section \ref{sec: local equilibria}.

\subsubsection{
Rates $g(n)\sim n^\alpha$}  To describe the main result in the setting $g(n)\sim n^\alpha$, it will be helpful to get a sense of the `condensation' of particles, under an invariant measure $\Pam^N_c$.
 Typically at a slow site, (1) the number of particles will be $O(N^{\beta/\alpha})$ when $\beta> \alpha$, referred to as `super-slow' or `super-critical', (2) order $O(N)$ when $\beta=\alpha$, called `critical', and (3) $o(N)$ when $\beta<\alpha$, named as `sub-critical'.  So, with a finite number of slow sites, if one of the $k_{j,N}$'s is `super-slow',
 there will be a superlinear 
$O(N^{\beta_j/\alpha})$ number of particles at that location.  Whereas, when all the $\beta_j\leq \alpha$, there will be $O(N)$ particles on $\T_N$. See Section \ref{sec: static limit} for precise statements.

Let $\D_{s,N}$ be the set of  super-slow sites $k_{j,N}$,
 and let $\D_s$ be the corresponding set of continuum points $x_j \sim k_{j,N}/N$.  Since, at these locations, the particle numbers are superlinear, we omit them in the definition of the empirical measure with respect to the diffusively scaled system:
$$\pi^N_t = \frac{1}{N}\sum_{k\in \T_N\setminus \D_{s,N}} \eta_{t}(k) \delta_{k/N}.$$
By a hydrodynamic limit, we mean, with respect to test functions $G\in C(\T)$, that
$\langle G, \pi^N_t\rangle$ converges in probability to $\langle G, \pi_t\rangle$, where $\pi_t$ is a measure-valued weak solution of a specified macroscopic evolution.

Our first result (Theorem \ref{main thm: k alpha}) is that three different behaviors and boundary conditions, at the macroscopic defect sites $x_j$, arise in the hydrodynamic limit depending on when $\beta_j>\alpha$, $\beta_j = \alpha$ and $\beta_j<\alpha$.  As might be suspected, when $\beta_j<\alpha$, the defect is not `slow' enough to be seen in the continuum limit.  However, defects where $\beta_j\geq \alpha$ on the other hand do register.  To describe the bulk mass macroscopic flow, segregate the unit torus $\T$ into intervals with endpoints $x_j$ corresponding to $\beta_j\geq \alpha$.  In the interior of each interval, the hydrodynamic limit is described by a nonlinear heat equation
\begin{equation}
\label{hyd_lim_intro}
\partial_t \rho = \partial_{xx}\Phi(\rho).
\end{equation}
Here, the `fugacity' function $\Phi$ is a continuum homogenization of the microscopic rate $g$ (cf. \eqref{Phi_eqn}).
We note this equation also arises in the context of the zero-range system without disorder. (cf. Chapter 4 \cite{KL}).

However, at a defect point $x_j$ where $\beta_j = \alpha$, an atom evolves with 
a mass given by $\mathfrak{m}_j(t) = \big\{\lambda_j \Phi(\rho(t,x_j))\big\}^{1/\alpha}$, in terms of $\rho$.  Such a statement is natural, given that under $\mu^N$, not far from $\Pam^N_{c_0}$, there are $O(N)$ particles at $k_{j,N}$.
Differently, when $\beta_j>\alpha$, we will observe initially a superlinear $O(N^{\beta_j/\alpha})$ number of particles at $k_{j,N}$ due to the relative entropy bound assumption.
This level of condensation will not evolve much for $t>0$ and, as a result, leads to a macroscopic boundary condition $\rho(t,x_j) = c_0$ where $c_0$ is the parameter of the reference measure 
$\Pam_{c_0}^N$, reflecting a level of initial condensation at $x_j$. The derivation of such Dirichlet boundary condition at super-slow sites ($\beta_j>\alpha$), to our knowledge, is novel.

Taken together, the hydrodynamic limit $\pi_t$ on $\T$ will be in form
$$\pi_t = \rho(t,x)dx + \sum_{j: \beta_j=\alpha} \mathfrak{m}_j(t)\delta_{x_j}(dx), \quad\text{for $t>0$.} $$
 At time $t=0$, $\pi_t$ reduces to $\pi_0$ mentioned earlier.
 In general, the PDE \eqref{hyd_lim_intro} for the bulk mass $\rho$ will not be closed without $\m_j(t)$ being specified.
In fact, the limit $\pi_t$ is characterized as the unique weak solution to a system in terms of $(\rho(t,x), \{\mathfrak{m}_j(t)\}_j)$:
\begin{equation} \label {eqn: hyro_eqn_pi_t}
\partial_t \pi_t = \partial_{xx} \Phi(\rho)  
\end{equation}
with boundary behaviors of $\rho(t,x)$ at slow sites $x_j$ given earlier;
see Definition \ref{def: weak sln} and uniqueness Theorem \ref{thm: uniqueness}.  See Example \ref{sec: single slow site}, for a specific discussion when there is only one defect in the system.
We mention in \cite{JLT}, when $g(n)\equiv n$
($\alpha=1$), that is for independent particles, a case of the above hydrodynamic limit may be inferred when $\beta_j \equiv \alpha =1$, with $\Phi(u)\equiv u$.

\subsubsection{Bounded rates $g$}
Our second result (Theorem \ref{main thm: bounded g}) concerns the process with 
a rate function $g(n)$, bounded and
increasing say to level $g_\infty = 1$ in the limit as $n\uparrow\infty$.
Slow sites $k_{j,N}$ that we consider will be those where the jump rate is $\tfrac{1}{\lambda_j N^{\beta_j}}g(n) = \tfrac{1}{\lambda_j}g(n)$, when $\beta_j=0$ and $\lambda_j>1$, 
when there are $n$ particles at $k$.  As before, with a finite number of slow sites $k_{j,N}$, there is a family of product invariant measures $\{\Pam^N_c\}$, but now with densities $c$ limited to $0\leq \max_j\{\lambda_jN^{\beta_j}\vee 1\}\Phi(c)\leq g_\infty$.   There is no non-trivial product invariant measure for $c$ above this level.  With this understanding, it is natural to focus on $\beta_j= 0$, as if we would slow down with 
$\beta_j>0$, there would be no such non-trivial invariant measure, as the fugacity $\lambda_j N^{\beta_j}\Phi(c)$ would exceed $g_\infty$ for a finite $N$.

Phenomenologically, the slow site effect is different here than when $g(n)\sim n^\alpha$ 
in that the system and in particular the slow sites under $\Pam^N_c$ have at most $O(N)$ particles on them.  As before, between macroscopic defects $\{x_j \sim k_{j,N}/N\}$, the hydrodynamic limit satisfies \eqref{hyd_lim_intro}.  An atom of mass $\mathfrak{m}_j(t)$ may also form though at macroscopic site $x_j \sim k_{j,N}/N$ where $\beta_j=0$ and $\lambda_j>1$.

However, unlike in the case 
$g(n)\sim n^\alpha$, this atomic mass may be evanescent:  We show that, at such a defect $x_j$, the hydrodynamic density $\rho(t,x_j)$ satisfies a bound
$\rho(t,x_j) \leq c_{j,\max}$ where 
$c_{j,\max}:=\Phi^{-1}(g_\infty/\lambda_j)$ for all $t>0$.
When this inequality is strict, it holds necessarily that $\mathfrak{m}_j(t)=0$. 
In particular, the boundary condition near $x_j$ may bounce between being periodic and Dirichlet. 

For an informal description, suppose that $\rho(t,x_j)$ is strictly less than $c_{j,\max}$ at $t=t_0$ and
$\m_j(t_0)=0$.
Here, the macroscopic flow will evolve as if the defect at $x_j$ were not present,
namely, we will observe a periodic boundary condition at $x=x_j$.
This will last until the first moment $t=t_1$ when the density $\rho(t,x_j)$ tends toward exceeding the threshold value $c_{j,\max}$.
The defect is then `set off' to keep $\rho(t,x_j)$ at the value $c_{j,\max}$, with any excess mass stored at the defect to form a delta mass.
In other words, we start to observe  $\m_j(t)>0$ as well as a Dirichlet boundary condition $\rho(t,x_j)=c_{j,\max}$ after $t=t_1$. Besides being a place of storage, the slow site $x_j$ also serves as a source of mass to maintain the mentioned Dirichlet boundary when needed as long as $\m_j(t)>0$. When the stored mass is used up (when $\m_j(t)$ becomes $0$) and the density $\rho(t,x_j)$ tends to drop below $c_{j,\max}$, say at $t=t_3$, the boundary condition switches to the periodic one again after $t=t_3$. Such a `bouncing' phenomenon between boundary conditions takes place at all slow sites.
 Finally, characterization of the limit 
$$\pi_t = \rho(t,x)dx + \sum_{j: \beta_j=0, \lambda_j>0} \mathfrak{m}_j(t)\delta_{x_j}(dx)$$
 is given through a weak formulation \eqref{eqn: hyro_eqn_pi_t} with boundary conditions described earlier, shown to have a unique solution; see Definition \ref{def: weak sln, bounded g} and uniqueness Theorem \ref{thm: uniqueness, bounded g}.

We remark that a sufficient condition to not feel the defects $\{x_j\}$ in the limit would be to start the process from initial (local equilibrium) measures $\mu^N$ associated to $\pi_0(dx)=\rho_0(x)dx$, where $\|\rho_0\|_\infty< \Phi^{-1}\big([\max_j\{\lambda_j\}]^{-1}\big)$, so that by say attractiveness $\rho(t,x)$ also satisfies this bound.  On the other hand, a sufficient condition so that atom masses form is that the initial density $\rho_0(x)> \Phi^{-1}\big([\max_j\{\lambda_j\}]^{-1}\big)$ on $\T$:  Indeed, if no atoms are formed, by the maximum principle, $\rho(t,x)>  \Phi^{-1}\big([\max_j\{\lambda_j\}]^{-1}\big)$ on $\T$, a contradiction.  We believe the results in Theorem \ref{main thm: bounded g} are the first for a general class of bounded symmetric zero-range processes with defects.

\subsection{
Proof ideas for Theorems \ref{main thm: k alpha} and \ref{main thm: bounded g}}  We now discuss ideas in the proofs.  We use the general scheme of the `entropy' method, discussed in \cite{KL}, although there are several departures since the dynamics is not translation invariant.  In particular, mixing in time estimates are used to make microscopic to macroscopic homogenizations in the `bulk'.  However, to capture the boundary effects and correspondences with macroscopic mass of atoms, we need to estimate the local behavior near defects, for which we give a more refined argument.

  Let $G$ be a test function on $\T$ when $g$ is bounded, and with compact support away from super-slow sites $x_j\in \D_s$ when $g\sim n^\alpha$.  Consider the stochastic differential
  $$d\langle G, \pi^N_t\rangle = \frac{1}{N}\sum_{k\in \T_N} \Delta_N G\big(k/N\big) g_k(\eta_t(k)) dt + dM^N(t).$$
  Here, $g_k = \tfrac{1}{\lambda_j N^{\beta_j}}g$ when $k=k_{j,N}$ and $g_k=g$ otherwise.
  Also, $M^N(t)$ is a martingale, which will vanish as $N\uparrow\infty$. 
  We state a `bulk' replacement in Lemma \ref{rplmt_global_lem}.  The proof that we give makes use of local replacements, as these will be useful to deduce boundary behaviors at the defects.   
  
In particular, we may replace the local function $g(\eta_t(k))$ at regular sites $k$ by $\Phi\big( \eta^{\theta N}_t\big)$ where $\eta^{\ell}$ is the $\ell$-window average $\frac{1}{2\ell+1} \sum_{|y-k|\leq \ell}\eta(y)$, by use of an entropy inequality and a `Rayleigh' estimate.  The `Rayleigh' estimate controls the difference between expectations of a local function under non-equilibrium and equilibrium measures in terms of a Dirichlet form and a term written in terms of the spectral gap of the process localized to an interval.  No special bound is required--this gap needs only to be positive, or equivalently that the localized dynamics is ergodic.  We remark, as a technical device, attractiveness is used to handle large densities, difficult to analyze otherwise, in the proofs of these local replacement, `1-block' Lemma \ref{lem: 1 block} and `2-block' Lemma \ref{lem: 2 blocks}.  As a consequence of these estimates, the hydrodynamic equation \eqref{hyd_lim_intro} may be derived in the `bulk'. We now discuss derivation of the boundary conditions in the two settings.

\subsubsection{Boundary derivations for $g(n)\sim n^\alpha$}  In the setting where 
$g(n)\sim n^\alpha$
at a defect $k_{j,N}$, when $0<\beta_j<\alpha$, under $\mu^N$, the number of  
particles $\eta_t(k_{j,N}) = O(N^{\beta_j /\alpha})$ is sublinear in $N$.  Hence, with respect to the empirical measure, the term 
$$
\langle G, N^{-1} \delta_{k_{j,N}/N} \rangle
=
N^{-1}\eta_t(k_{j,N}) G(k_{j,N}/N)
\leq \|G\|_\infty N^{-1}\eta_t(k_{j,N}) = O(N^{-1+\beta_j/\alpha})$$
vanishes as $N\uparrow\infty$.  In effect, we can ignore such defects in the continuum limit.

However, when $\beta_j\geq \alpha$, to deduce a non-trivial boundary condition, we consider the rate $g_{k_{j, N}}(\eta_t(k_{j,N}))=\tfrac{1}{\lambda_j N^{\beta_j}}g(\eta_t(k_{j,N})$ at the slow site $k_{j,N}$.  The following discussion leading to \eqref{boundary_intro} will hold actually for both $g$ of $n^\alpha$ type and $g$ bounded settings.  To fix ideas, let us say $k_{j,N}=0$, the origin.  Under $\mu^N$, not far from the reversible $\Pam^N_{c_0}$ in terms of relative entropy, we show that this rate is close to the neighboring rate $g(\eta_t(1))$ say; see Lemma \ref{replacement: slow site}.  By the local `1-block' estimate, we show that $g(\eta_t(1))$ is close to $\Phi(\eta^{\ell, +}_t(0))$, where $\eta_t^{\ell, +}(0)= \frac{1}{\ell}\sum_{y=1}^\ell \eta_t(y)$ is the average in the $\ell $-block to the right of the defect site $0$; see Lemma \ref{local replace, slow site}.  Now, by the local `2-block' bound, we have $\Phi(\eta^{\ell, +}_t(0))$ is close to the macroscopic quantity $\Phi((\eta^{\theta N, +}_t(0))$ for $\theta>0$ small.  Together, we conclude the fundamental relation,
\begin{equation}
\label{boundary_intro}
\Phi(\eta^{\theta N, +}_t(0)) \sim g_0(\eta_t(0)).
\end{equation}

After establishment of tightness of the empirical measure trajectories, and absolute continuity estimates in Lemmas \ref{lem: tightness} and \ref{lem: abs continuity}, we may consider 
a subsequential limit $\pi$ on $\T\setminus \D_s$ in the form $\pi = \rho(t,x)dx + \sum_{j: \beta_j=\alpha} \mathfrak{m}_j(t)$.  From \eqref{boundary_intro}, one has for this limit point that
$$\Phi(\rho(t, 0)) \sim g_0(\eta_t(0)).$$

When $\beta_j = \alpha$, we have 
$$\Phi(\rho(t, 0)) 
\sim g_0(\eta_t(0)) 
\sim \big(N^{-1}\eta_t(0)\big)^\alpha \lambda_j^{-1} 
\sim \big(\m_j(t)\big)^\alpha \lambda_j^{-1},$$
which is the boundary condition relating the atom strength to the local density in Theorem \ref{main thm: k alpha}; see Lemma \ref{lem: bdry at critical}.

When $\beta_j>\alpha$, we may take the same passage to obtain
$$\Phi(\rho(t, 0))
 \sim g_0(\eta_t(0)) 
\sim \big (N^{-\beta_j/\alpha} \eta_t(0)\big)^\alpha \lambda_j^{-1}.$$
But, in our time-scale, we show $N^{-\beta_j/\alpha}\eta_t(0)$ does not vary much from the time $t=0$ quantity $N^{-\beta_j/\alpha}\eta_0(0)$.  For instance, if $k_{j,N}=0$ were the only defect in the system, by mass conservation, 
$$|\eta_t(0)-\eta_0(0)| = \big|\sum_{k\neq 0} \big(\eta_t(k) - \eta_0(k)\big)\big|.$$
Starting under $\mu^N$, by use of the entropy inequality, $\sum_{k\neq 0}\eta_t(k) = O(N)$ for all $t\geq 0$.  Hence, 
$$N^{-\beta_j/\alpha}\big | \eta_t(0) - \eta_0(0)\big | = O(N^{-\beta_j/\alpha +1})$$
vanishes as $N\uparrow\infty$.  Moreover, by a calculation with respect to the initial measure $\mu^N$, we may show that 
$$\big(N^{-\beta_j/\alpha}\eta_0(0)\big)^\alpha \sim \lambda_j\Phi(c_0).$$
As a consequence, we have 
$$\Phi(\rho(t,0)) \sim \big(N^{-\beta_j/\alpha}\eta_0(0)\big)^\alpha \lambda_j^{-1} \sim \Phi(c_0),$$
yielding the boundary behavior in Theorem \ref{main thm: k alpha};
see Lemma \ref{lem: boundary cond, super} which addresses the multiple defect setting.

\subsubsection{Boundary derivations for bounded rate $g$} 
We now discuss the case of bounded rate function $g$: $\|g\|_{\infty}=g_\infty<\infty$.  There are subtleties here with respect to boundary estimates, different than in the $n^\alpha$ rate setting, although the bulk hydrodynamic limit between defects $x_j\sim k_{j,N}/N$ follows the same procedure as above to derive the equation \eqref{hyd_lim_intro}.

At a slow site $k_{j,N}$, say again at the origin, 
slowed down by $\lambda_j^{-1}$ (and $\beta_j=0$), we deduce
$$\Phi(\rho(t,0)) \sim \frac{1}{\lambda_j} g(\eta_t(0)) \leq \frac{g_\infty}{\lambda_j}.$$
Hence, the density near the slow site $x_j=0$ is restricted,
$$\rho(t,0) \leq \Phi^{-1}(g_\infty / \lambda_j).$$
When the restriction is strict, the intuition is that the particle mass is not high enough to impact much the 
flow of particle numbers through the slow site $0$.  However, when equality in the above relation holds, 
an atom at the slow site $0$ may form to store excess mass while maintaining the boundary condition.  More precisely, we have $\m_j(t)$, the atomic mass accumulated by $x_j=0$, satisfies $\m_j(t)=\m_j(t) \id_{ \rho(t,0) = \Phi^{-1}(g_\infty/\lambda_j)}$; this is established by showing a microscopic version of the following key relation
\[
\m_j(t) \big(\Phi^{-1}(g_\infty / \lambda_j) - \rho(t,0) \big) = 0;
\]
see Lemma \ref{lem: boundary, bounded}.   
These prescriptions give the boundary conditions specified in Theorem \ref{main thm: bounded g}.

\subsection{Paper outline} The plan of the paper is as follows:  In Section \ref{sec: model} and \ref{sec: initial measures}, we introduce carefully the zero-range model with defects, their invariant measures, and the initial measures considered.  In Section \ref{sec: results}, we state results; see Section \ref{sec: single slow site} for a discussion of the example when there is only one defect in the system.  In Section \ref{section: martingale}, we give the proof outline of the main results, Theorems \ref{main thm: k alpha} and \ref{main thm: bounded g}, referring to estimates in the sequel.  In Section \ref{sec: tightness}, we discuss tightness of the empirical measures.  In Section \ref{sec: properties of limit measures}, we discuss properties of limit points, including importantly their boundary behaviors near macroscopic defects.  In Section \ref{sec: rplmt lem}, we show local `1-block' and `2-block' replacements, and as a consequence `bulk' replacement.  In Section \ref{sec: rplmt lem at bdry}, we discuss replacements near the boundaries of defects needed to derive the macroscopic boundary conditions.  In Sections \ref{sec: energy section} and \ref{section: uniqueness}, we derive energy estimates and prove uniqueness theorems for the weak formulations.

\section{Model description}
\label{sec: model}

We will consider symmetric zero-range processes on the discrete torus $\T_N:= \Z/N\Z = \{0,1,\ldots,N-1\}$ with a finite number $n_0\geq 0$ of defects located in $\T_N$.
We will always assume that $N>n_0$ so that there is enough space in $\T_N$ to record the defects.

More carefully, the structure of the defects is the following:  
Let $J$ be the index set $\{1,2, \ldots, n_0\}$.
For each $j\in J$, fix $(x_j,\beta_j, \lambda_j)$ such that 
$x_j\in \T$, $\beta_j\in \R$ and $\lambda_j\in (0,\infty)$. 
Here, $\T$ stands for $[0,1)$ viewed as the unit torus.
We will assume that all $x_j$'s are different, that is macroscopically separated.
For each $j$, the point $x_j$ denotes the macroscopic location of a defect and $(\beta_j, \lambda_j)$ characterizes its strength.
Let $\D:= \{x_j\}_{j\in J}$ be the set of all macroscopic defect locations.
For each $j\in J$ and $N\in \N$,
we now define $k_{j,N} = \lfloor x_j N \rfloor$ be the integer part of $x_j N$.
Then, the set $\D_N := \{k_{j,N}\}_{j\in J}$
stands for the set of microscopic locations of the defects. 

Let now $\N_0 := \{0\}\cup \N$.
For each $N$, let $\Omega_N:=\N_0^{\T_N}$ be the space of all particle configurations on $\T_N$.
With respect to $\xi\in \Omega_N$,
at a normal site $k\in \T_N\setminus \D_N$, a particle jumps to neighboring sites $k\pm1$ equally likely at rate $g(\xi(k)) / \xi(k)$ where $\xi(k)$ is the number of particles at site $k$, and $g:\N_0 \mapsto  [0,\infty)$ is a jump rate function. At a defect site $k=k_{j,N}$, the jump rate is $(\lambda_j N^{\beta_j})^{-1} g(\xi(k)) / \xi(k)$. 
In particular, the site $k_{j,N}$ is a slow site if $\lambda_j N^{\beta_j}>1$ and a fast or normal site if $\lambda_j N^{\beta_j}\leq1$
for $N$ large.  Let
\[
g_{k,N} (\cdot)
=
\begin{cases}
g(\cdot)&k\in\T_N \setminus \D_N,\\
\dfrac {g(\cdot)} {\lambda_j N^{\beta_j}}
&k=k_{j,N}, \ j\in J.
\end{cases}
\]

The zero-range process is a Markov process $\xi_t$ with the generator
\begin{equation}
\label{eqn: generator L}
\begin{split}
L_N f(\xi) =
\sum_{k\in \T_N} 
\Big\{
g_{k,N}(\xi(k))
\big(f(\xi^{k,k+1}) - f(\xi) \big)
+
g_{k,N}(\xi(k)) 
\big(f(\xi^{k,k-1}) - f(\xi) \big)
\Big\} 
\end{split}
\end{equation}
where
\begin{equation} \label {def: xi xy}
\xi^{x,y} (k) 
=
\begin{cases}
\xi(x) -1& k=x,\\
\xi(y) +1& k=y,\\
\xi(k) & k\neq x,y.
\end{cases}
\end{equation}

So that the process is irreducible, we will assume that the jump rate function $g$ is such that $g(n)=0$ exactly when $n=0$.  We will also assume the following condition.
\begin{Cond}
\label{condition on g}
The jump rate function $g(\cdot)$ satisfies
\begin{enumerate}
\item Lipschitz: there exists $g^*>0$ such that $|g(n+1) - g(n)| \leq g^*$ for all $n\in \N_0$.
\item Power Interaction: there exists $\alpha \in [0,1]$ such that $g(n)\sim n^\alpha$, that is
\[
\lim_{n\to \infty} \dfrac{g(n)}{n^{\alpha}} = 1.
\]
\item Monotonicity: $g(n)\leq g(n+1)$ for all $n\in \N_0$.
\end{enumerate}
\end{Cond}

When the constant $\alpha=0$, in the power interaction condition, we note the function $g(n)$ increases to $g_\infty := 1$ as $n \to \infty$. We will refer to this case as 
`$g$ is of bounded type'
 or simply, $g$ is {\it bounded}.
On the other hand, when $\alpha\in (0,1]$, we will say 
`$g$ is of  $n^\alpha$ type'.

We remark that, when $g(n) \equiv n$, the dynamics is non-interactive, a superposition of independent random walks.  The Lipschitz condition, which rules out the case $\alpha>1$, is convenient to bound $g(n)$ by $g^*n$, and in using the entropy inequality in places.
Moreover, the assumption that $g$ is  increasing, an `attractive' dynamics assumption, allows use of the `basic coupling' in our proofs; for more discussion of this point; see Section \ref{sec: coupling}.

Also, as discussed in the introduction, we will suppose in the $g$ bounded setting that, at a defect site $k_{j,N}$, the parameter $\beta_j\leq 0$.  Such a condition ensures existence of a non-trivial family of invariant measures; see Section \ref{sec: invariant measures}.

\begin{Cond}
\label{no J_s in bounded}
In the $g$ bounded setting, we will assume $\beta_j\leq 0$ for a defect site $k=k_{j,N}$.
\end{Cond}

Our goal in this work is to study hydrodynamic limits of the dynamics generated by $L_N$. In particular, 
in the limit macroscopic flow, there will be different behaviors at the defect locations depending on the associated strength parameters. 
To prepare for these statements, it will be helpful to introduce a partition on the index set $J = J_s\cup J_c \cup J_b$. Here, the subscripts $s$, $c$, and $b$ stand for super-critical or `super-slow', critical, and sub-critical respectively.

There will be different partitions of $J$ depending on whether $g$ is of $n^\alpha$ type or bounded.
When $g$ is of $n^\alpha $ type, we let
\begin{itemize}
\item $J_s := \{j\in J: \beta_j >\alpha \}$,
\item $J_c := \{j\in J: \beta_j =\alpha \}$, and 
\item $J_b := \{j\in J: \beta_j <\alpha \}$.
\end{itemize}

When $g$ is bounded, we take 
\begin{itemize}
\item $J_s := \{j\in J: \beta_j >0 \}$ (= $\emptyset$), 
\item $J_c := \{j\in J: \beta_j =0, \ \lambda_j>1\}$, and
\item $J_b := \{j\in J: \beta_j <0 \text{ or } \beta_j=0 \text{ with }\lambda_j < 1 \}$.
\end{itemize}

In terms of the partition of $J$, the sets $\D$ and $\D_N$, the macroscopic and microscopic locations of the defects, are then divided into corresponding  subsets.
For example, we have $\D_s:= \{x_j\in \D: j\in J_s\}$ and 
$\D_{s,N}:= \{k_{j,N} \in \D_N: j\in J_s\}$. 
The sets $\D_{c}$, $\D_{b}$, $\D_{c,N}$, and $\D_{b,N}$ are defined in the same fashion.

 \subsection{Invariant measures}
 \label{sec: invariant measures}
 The construction of invariant measures under 
$L_N$ is based on  $\{\Pb_{\phi}\}$, a
family of Poisson-like distributions indexed by `fugacities' $\phi$.
In order to define $\Pb_\phi$, we first introduce the partition function:
\[
Z(\phi) := \sum_{n=0}^\infty  \dfrac{\phi^n} {g(n) !}.
\]
Let $r_g$ be the convergence radius of $Z(\cdot)$. 
It holds that $r_g = \lim_{n\to \infty} g(n)$.
In particular, when $g$ is of $n^\alpha $ type, we have $r_g = \infty$, namely, the ``FEM" condition (cf.\,p.\,69, \cite{KL}) is satisfied.
When $g$ is bounded, $r_g=g_{\infty} =1$.
In either case, it holds that $\lim_{\phi \uparrow r_g} Z(\phi) = \infty$.

For each $\phi\in [0,r_g)$, define $\Pb_{\phi}$ by
\begin{equation}
\label{Pb def}
\Pb_{\phi}(n) =
\dfrac{1}{Z(\phi)} \dfrac{\phi^n} {g(n) !},
\quad \text{for } n\geq 0.
\end{equation}
Here, $g(0)! :=1$ and $g(n)! :=\prod_{k=1}^n g(k)$ for $n\geq 1$.
Let $R(\phi) = E_{\Pb_\phi} [ X ]$, where $X(n) = n$, be the mean of
the distribution $\Pb_\phi$. A direct computation yields that
$R(0)=0$ and $\lim_{\phi\uparrow r_g} R(\phi) = \infty$.
Moreover, it holds $R'(\phi)= \sigma^2(\phi) / \phi$
and $\lim_{\phi\downarrow 0} R'(\phi) =1/ g(1)$
 where
$\sigma^2(\phi) = \V_{\Pb_{\phi}}[X]$ is the variance of $X$ under $\Pb_\phi$.
 
 Since
$R$ is strictly increasing, it has an inverse, denoted by
$\Phi: [0,\infty) \mapsto [0,r_g)$.  We may parametrize the family of
distributions $\Pb_{\phi}$ by their means: For $\rho\ge 0$, let
$\Qb_\rho = \Pb_{\Phi(\rho)}$, so that
$E_{\Qb_\rho} [ X ] = E_{\Pb_{\Phi(\rho)}} [ X ] = R(\Phi(\rho)) =
\rho$.
A straightforward computation yields that
$E_{\Pb_\phi} [ g(X) ] = \phi$ for $\phi\ge 0$. Thus,
\begin{equation}
\label{Phi_eqn}
\Phi(\rho) = E_{\Pb_{\Phi(\rho)}} [ g(X) ]
= E_{\Qb_{\rho}} [ g(X) ], \quad \rho\,\ge\, 0.
\end{equation}
As $g(n) \leq g^*n$, we have that
$\Phi(\rho) \leq g^*\rho$. 
As $\Phi(\cdot) = R^{-1}(\cdot)$, we have $\Phi'(\rho) = \Phi(\rho)/\sigma^2(\Phi(\rho))$.
Under our assumptions on $g$, in fact, it holds that $0\leq \Phi'(\rho)\leq g^*$
for all $\rho\ge 0$ (cf.\,p.~33, \cite{KL}). In particular,
$\Phi\in C^1[0,\infty)$ is an increasing function with a uniformly bounded derivative.
We note, in the case $g(n)\equiv n$, that $\Phi(\rho)\equiv \rho$ and $\Pb_\phi$ is a Poisson measure with mean $\phi$. 

We now introduce the invariant measures.
For each $N$, let 
$$q_N = \max \{1,\lambda_j N^{\beta_j} | j\in J\}.$$

For $c$ so that $\Phi(c) \in [0,r_g/q_N)$, denote by $\Pam_c^N$ the product measure on $\Omega_N$ whose
marginals are given by
\begin{equation}
\label {Pam c marginal}
\Pam_c^N (\xi(k) = n) = 
\begin{cases}
\Pb_{\Phi(c)}(n) & \text{ for } k\in\T_N \setminus \D_N  \text{ and } n\geq 0,\\
\Pb_{\lambda_j N^{\beta_j} \Phi(c)}(n) & \text{ for } k=k_{j,N}, \ j\in J \text{, and } n\geq 0.
\end{cases}
\end{equation}

Notice that the condition $\Phi(c) \in [0,r_g/q_N)$ is needed since the distributions $\{\Pb_\phi\}$ are defined for $\phi\in [0,r_g)$. 
When $g$ is of $n^\alpha $ type, as $r_g = \infty$, we have that $\{\Pam_c^N\}$ are defined for all  $c\in[0,\infty)$.
However, when $g$ is bounded, we have $r_g = g_\infty=1$, and thus the measures $\{\Pam_c^N\}$ are defined only for $c\in [0,R(1/ q_N))$.  Moreover, in the $g$ bounded setting, we have that $r_g/q_N$ is bounded away from $0$ for all large $N$ exactly when $J_s=\emptyset$, that is $\beta_j\leq 0$, the reason for the assumption $J_s=\emptyset$ in this case (cf.~Condition \ref{no J_s in bounded}).
 
With $\Pam_c^N$ defined, it is straightforward (cf.~\cite{A}, \cite{EH}) to check the following lemma.
\begin{Lem}
For $c$ so that $\Phi(c) \in [0,r_g/q_N)$, 
 $\Pam_c^N$ is invariant and reversible  with respect to the generator $L_N$ in \eqref{eqn: generator L}. 
 \end{Lem}

 \subsection{Static limit} 
 \label {sec: static limit}
 Before studying the hydrodynamic limits, it will be useful to understand the particle mass behavior under an invariant measure $\Pam^N_c$.
 For a configuration $\xi \in \Omega_N$, 
 define the associated scaled mass empirical measure:
\begin{equation}
\label{eqn: pi_N, hat}
\wpi^N(dx)
:=
\dfrac 1 N 
\sum_{k\in \T_N}
\xi (k) \delta_{k/N}(dx).
\end{equation}
In this formulation, each particle has mass $N^{-1}$.  Here and in the sequel, $\delta_z$ refers to a delta point mass at $z$.

For a test function $G\in C(\T)$, let 
\[
 \langle G, \wpi^N\rangle 
 :=
N^{-1} \sum_{k\in \T_N} \xi(k) G(k/N).
\]
More generally, we will use the notation $\langle G, \mu\rangle :=\int G d\mu$.
We now compute the limit of 
$\langle G, \wpi^N\rangle$ with respect to a sequence of  invariant measures $\Pam_c^N$ with $c$ fixed as $N\to \infty$.

  We assume first $g$ is of $n^\alpha$ type
  and $c> 0$.   
  Because of the product structure of $\Pam_c^N$, 
  $\{\xi(k) \}_{k\in\T_N}$ are independent
and have a commmon marginal $\Pb_{\Phi(c)}$ for all $k\neq \D_N$. As $\Pb_{\Phi(c)}$ has expectation $c$ and finite variance, we have 
$N^{-1} \sum_{k\notin \D_N} \xi(k) G(k/N) $
converges in probability to 
$\int_\T G(x) c \,dx$ as $N\to \infty$.

It remains to investigate the behavior of $N^{-1} \xi(k)$ for $k=k_{j,N}\in \D_N$.
As $\xi(k_{j,N})$ has distribution $\Pb_{\lambda_j N^{\beta_j} \Phi(c)}$, by the later Lemma \ref{lem: order of R}, for all $\beta_j>0$,
\[
 E_{\Pam_c^N }
[\xi(k_{j,N})] 
\sim
 (\lambda_j \Phi(c))^{1/\alpha}
 N^{\beta_j/\alpha},
\quad\text{and}\quad
 \V_{\Pam_c^N } 
[\xi(k_{j,N})] 
=
o(N^{2\beta_j/\alpha}).
 \]
 Then, according to the value of $\beta_j$, there are three different types of behaviors at $k\in \D_N$:
 \begin{enumerate}
\item
if $\beta_j < \alpha$,  
$N^{-1} \xi(k_{j,N})\to 0$ in probability;

\item
 if $\beta_j = \alpha$, 
  $N^{-1} \xi(k_{j,N})\to \left(\lambda_j \Phi(c)\right)^{1/\alpha}$ in probability;

\item 
 if $\beta_j > \alpha$,  
$N^{-\beta_j/\alpha}
 \xi(k_{j,N}) \to  \left(\lambda_j \Phi(c)\right)^{1/\alpha}$
  in probability.
\end{enumerate}
In other words, the defect site $k_{j,N}$ becomes macroscopically invisible when $\beta_j< \alpha$ as typically it contains $o(N)$ number of particles and each particle has mass $N^{-1}$.
In the case $\beta_j = \alpha$, typically the number of particles at $k_{j,N}$ is of order $N$ and a delta mass of magnitude $ \left(\lambda_j \Phi(c)\right)^{1/\alpha}$ emerges. When the site is super-slow, that is $\beta_j > \alpha$, the particle number at $k_{j,N}$ is of order $N^{\beta_j/\alpha}$, corresponding to an infinite macroscopic mass. Recall, the partition $J = J_b\cup J_c\cup J_s$ in Section \ref{sec: model} matches with this classification of the $\beta_j$'s.

As the macroscopical mass at $x_j\in \D_s$ explodes, to consider the remaining mass, we define microscopic empirical measures which exclude the super-critical defect set $\D_{s,N}$:
\begin{equation} \label{eqn: pi_N}
\pi^N(dx)
:=
\dfrac 1 N 
\sum_{k\in \T_N\setminus \D_{s,N}}
\xi (k) \delta_{k/N}(dx).
\end{equation}
Then, we may summarize the above discussion as follows.

\begin{Prop} \label {static limit of pi}
Assume $g$ is of $n^\alpha$ type. Then,
for any $G\in C(\T)$, $c\geq 0$, and $\delta>0$:
\[
\lim_{N\to \infty} \Pam_c^N
\big[ \, \big|
\langle G, \pi^N\rangle
-
\langle G, \pi\rangle
\big|
>\delta
\, \big]
 =0
\]
where $\pi(dx) = c\,dx + \sum_{j\in J_c} \left(\lambda_j \Phi(c)\right)^{1/\alpha} \delta_{x_j}(dx)$.
Moreover, for all $j\in \D_s$, we have
\[
\lim_{N\to \infty} \Pam_c^N
\big[ \, \big|
N^{-\beta_j/\alpha}  \xi(k_{j,N})
-
\left(\lambda_j \Phi(c)\right)^{1/\alpha}
\big|
>\delta
\, \big]
 =0.
\]

\end{Prop}

Now we turn to the case when $g$ is bounded. 
Recall, in this case, $\Pam_c^N$ is defined for $c$ such that $\Phi(c)\in [0,1/q_N)$ where
$q_N=\max \{1,\lambda_j N^{\beta_j} | j\in J\}$.  Recall also, in this setting, that $\beta_j\leq 0$, that is $J_s=\emptyset$.   
When $J_c\neq \emptyset$, we will define $\lambda_{\max} := \max_{j\in J_c} \{\lambda_j\}$.  Otherwise, when $J_c=\emptyset$, we will take $\lambda_{\max}=1$.  Hence, the domain for $c$ so that $\Pam_c^N$ is defined for all $N$ is $[0, R(1/\lambda_{\max})$.

For $c \in[0,R(1/\lambda_{\max}))$, we also have
$N^{-1} \sum_{k\notin \D_N} \xi(k) G(k/N) $
converges in probability to 
$\int_\T G(x) c \,dx$ as $N\to \infty$.
Given $\xi(k)$ has finite $\Pam_c^N$-expectation and variance for all $k\in \D_N$, the following is easily obtained.
\begin{Prop} 
\label {prop: static limit, bounded g}
Assume $g$ is bounded and $J_s=\emptyset$. Then,
for $c \in[0,R(1/\lambda_{\max}))$,
 for any $G\in C(\T)$ and $\delta>0$, we have
\[
\lim_{N\to \infty} \Pam_c^N
\big[ \, \big|
\langle G, \pi^N\rangle
-
\int_\T G(x)c\, dx\big|
>\delta
\, \big]
 =0.
\]
\end{Prop}

We finish this section with a technical lemma used in Proposition \ref{static limit of pi}.
 \begin{Lem} \label {lem: order of R}
 Assume $g(k)\sim k^\alpha$ for some $\alpha\in (0,1] $. 
 For each $\varphi>0$, let $X$ be a random variable with distribution $\Pb_{\varphi}$ (cf.\,\eqref{Pb def}).
 Then, for each $n\in \N$, $E[X^n]\sim\varphi^{n/\alpha}$ as $\varphi \to \infty$.
 As a result, $\ln Z(\varphi) \sim \alpha 
\varphi^{1/\alpha}$
 and $\V[X] =o\big( \varphi^{2/\alpha}\big)$ as $\varphi\to \infty$.
 \end{Lem}
 
 \begin{proof}
We first show $E[X^n]\sim\varphi^{n/\alpha}$ for all $n\in \N$. 
To this end, let us assume for now
\begin{equation} \label{mean: g by k^alpha}
\sum_{k=1}^\infty \dfrac{k^{n\alpha} \varphi^k}{g(k)!}
\sim
\sum_{j=n}^\infty  \dfrac{\varphi^k}{g(j-n)!}.
\end{equation}
Let $Y = X^\alpha$. Then,
 \[
 E[Y^n]
 =
 \dfrac{1}{Z(\varphi)}
\sum_{k=1}^\infty \dfrac{k^{n\alpha} \varphi^k}{g(k)!}
\sim
 \dfrac{1}{Z(\varphi)}
\sum_{k=n}^\infty  \dfrac{ \varphi^k}{g(k-n)!}
=\varphi^n.
 \]
 As $\alpha\in (0,1]$, we may find $p\in[0,1)$ and $l \in \N$ such that $\alpha^{-1} = p+(1-p)l$.
 Also, since $E[X^n] = E[Y^{n/\alpha} ]$, by Jensen's and H\"older's inequalities,
 \[
 E[Y^n]^{1/\alpha} \leq E[X^n] \leq E[Y^n]^p E[Y^{nl}]^{1-p}.
 \]
Since $E[Y^n]\sim  \varphi^n$ and $E[Y^{nl}]\sim\varphi^{nl}$, we obtain
$E[X^n ] \sim \varphi^{n/\alpha}$. 

For the limit behavior of $E[X^n]$, it remains to show the claim \eqref{mean: g by k^alpha}. 
As $g(k)\sim k^{\alpha}$, for any $A>0$, we may find $\lambda_1 = \lambda_1(A)$ and 
$\lambda_2 = \lambda_2(A)$, such that
$\lambda_1 k^\alpha\leq g(k) \leq \lambda_2 k^\alpha$ for all $k\geq A$ and $\lim_{A\to \infty}
\lambda_1 = \lim_{A\to \infty} \lambda_2 =1$. 
Then, for all $k\geq A+n$,
\[
 \lambda_2^{-n}
\leq
\dfrac{k^{n\alpha}}{\prod_{k-n<k'\leq k}g(k')}
\leq 
\Big(\dfrac{A+n}{A} \Big)^n \lambda_1^{-n}.
\]
Therefore,
\[
 \lambda_2^{-n}
\sum_{k=A+n}^\infty 
\dfrac{ \varphi^k}{g(k-n)!}
\leq
\sum_{k=A+n}^\infty \dfrac{k^{n\alpha} \varphi^k}{g(k)!}
\leq
\Big(\dfrac{A+n}{A} \Big)^n \lambda_1^{-n}
\sum_{k=A+n}^\infty 
\dfrac{ \varphi^k}{g(k-n)!}
\]
Notice that, if $\varphi$ is sent to infinity in the above display, we may replace $\sum_{k\geq A+n}$ by either $\sum_{k\geq 0}$ or $\sum_{k\geq n}$.
Then the claim \eqref{mean: g by k^alpha} follows from taking $\varphi\to \infty$ and then $A\to \infty$. 

 We have shown $E[X^n]\sim\varphi^{n/\alpha}$ for all $n\in \N$. Then, it follows that
 $\V[X] = o\big( \varphi^{2/\alpha}\big)$
 as  $E[X^2] \sim {E[X]}^2 \sim  \varphi^{2/\alpha}$.
To prove $\ln Z(\varphi) \sim \alpha 
\varphi^{1/\alpha}$, notice $
\dfrac{d}{d\varphi} \ln Z(\varphi)
=
\varphi^{-1} E[X]
\sim
\varphi^{1/\alpha - 1}
$
and then apply L'Hospital's rule, to finish the argument.
 \end{proof}


\section{Initial measures}
\label{sec: initial measures}
In this section, we specify the assumptions on the initial measures $\{\mu^N\}$ we use to start our dynamics.
Roughly speaking, $\{\mu^N\}$ should be associated with a macroscopic profile which gives the initial condition for the hydrodynamic limit.
We will also require $\mu^N$ to possess certain relative entropy estimates and to be stochastically bounded with respect to invariant measures. 

To specify these conditions, recall, for $\mu$, $\nu$, two probability measures on $\Omega_N$, we say that  $\mu\leq \nu$, that is $\mu$ is stochastically bounded by $\nu$, if for all $f: \Omega_N \mapsto \R$ coordinately increasing, we have $E_{\mu}(f) \leq E_{\nu} (f)$. 
Fix also $\pi$, a nonnegative measure on $\T$, such that 
\begin{equation} \label{eqn of pi}
\pi(dx) = \rho_0(x)dx + \sum_{j\in J_c} \m_{0,j}\, \delta_{x_j} (dx)
\end{equation}
where $\rho_0(x)\in L^1(\T)$ and $\m_{j,0}\geq 0$ for  $j\in J_c$. 
Recall also, the empirical measure $\pi^N$ defined in \eqref{eqn: pi_N}.

Throughout this work, we will assume the following on the sequence of initial measures $\{\mu^N\}_{N\in \N}$ on $\Omega_N$.
\begin{Cond}
\label {condition on mu N}
The following hold:
\begin{enumerate}
\item $\{\mu^N\}_{N\in \N}$ has macroscopic profile $\pi$ on $\T\setminus \D_s$, i.e.~for all $G(x)\in C ( \T)$ and $\delta>0$,
\[
\lim_{N\to \infty} \mu^N
\Big[ \, \Big|
 \langle G, \pi^{N} \rangle
-
 \langle G, \pi \rangle
\Big|
>\delta
\, \Big]
 =0.
\]  
\item
There exists $c_0> 0$ such that
the relative entropy of $\mu^N$ with respect to $\Pam_{c_0}^N$ is of order $N$:  Let $f_N := d\mu^N/d\Pam_{c_0}^N$.  Then, $H(\mu^N | \Pam_{c_0}^N):=  \int f_N \ln f_N d\Pam_{c_0}^N =O(N)$.
\item
When $g$ is of $n^\alpha$ type, 
there exists $c'$
such that $\mu^N$ is stochastically bounded by $\Pam_{c'}^N$ for all $N$.
\item[(3$'$)]
When $g$ is bounded, 
there exists $c'$ such that,
 when restricted to $\T_N\setminus \D_{c,N}$, $\mu^N$ is stochastically bounded by $\kappa_{c'}^N$, that is $E_{\mu^N}[f]\leq E_{\kappa_{c'}^N}[f]$ for coordinate increasing functions supported on $\{\eta(k): k\in \T_N\setminus\D_{c,N}\}$.
 Here,
 $$\kappa_{c'}^N =  \prod_{k\in \T_N\setminus \D_{b,N}} \Pb_{\Phi(c')} \times \prod_{k=k_{j,N}\in \D_{b,N}}\Pb_{\lambda_j N^{\beta_j}\Phi(c')}.$$
\end{enumerate}
\end{Cond}

We comment that
item (3) is sufficiently general in the $n^\alpha$ setting to allow $\{\mu^N\}$ to be associated with profiles of the form $\pi_0(dx) = \rho_0(x)dx + \sum_{j\in J_c}\mathfrak{m}_{0,j}\delta_{x_j}(dx)$, for densities $\rho_0$ and masses $\{\mathfrak{m}_{0,j}\}$, as demonstrated in Section \ref{sec: local equilibria}.  In the bounded $g$ case, however, if $\{\mu^N\}$ satisfied item (3), then $O(N)$ accumulations at points $\D_{c,N}$ would not be allowed.  In this case, the only profiles allowed would be of form $\pi_0(dx) = \rho_0(x)dx$, where $\|\Phi(\rho_0)\|_\infty \leq \min_{j\in J_c}\tfrac{1}{\lambda_j}$ (cf. Section \ref{sec: coupling}).  As mentioned in the introduction, hydrodynamic evolution from such profiles would not see any defects.

In this context, item (3$'$) is formulated so that, with respect to $\{\mu^N\}$, $O(N)$ accumulations are possible on $\D_{c,N}$ as well as later point masses on $\D_c$ in the hydrodynamic evolution.

In fact, conditions $(3)$ and $(3')$ can be made to accommodate a larger class of initial measures $\{\mu^N\}$.
For example, one may remove stochastic bounded assumptions on coordinates in $\D_{b,N}$. We have however chosen to state the $(3)$ and $(3')$ in the forms given to streamline arguments and avoid more piecemeal calculations.

Define now $\mu_t^N$ as the distribution of the process $\xi_t$, with initial measure $\mu^N$.
\medskip

In the rest of this section, we make remarks on 
the consequences of the
relative entropy bound in terms of $\mu^N_t$-particle numbers, discuss the use of attractiveness, and also 
provide a large class of examples of $\mu^N$ which satisfy Condition \ref{condition on mu N}.

\subsection{$\mu^N_t$-particle numbers and relative entropy}
\label{subsec: relative entropy}
We first comment on the $\mu_t^N$-particle numbers in the system.
Denote $H_N:= H(\mu^N | \Pam_{c_0}^N) = O(N)$.  Notice that the entropy does not increase in $t$, that is $H(\mu_t^N | \Pam_{c_0}^N) \leq H_N$ for 
$t\geq 0$
(cf.~pp~340, \cite{KL}).

Assume now that $g$ is of $n^\alpha$ type.
By the entropy inequality
$E_\mu[f] \leq H(\mu|\nu) + \ln E_\nu[e^f]$ (cf.\,p.338 \cite{KL}), 
\[
E_{\mu_t^N} \Big[
 \sum_{k\notin \D_{s,N}}
\xi(k)
\Big]
\leq
H(\mu_t^N | \Pam_{c_0}^N)
+
 \ln E_{\Pam_{c_0}^N} 
\big[
e^{\sum_{k\notin \D_{s,N}} \xi(k)}
\big].
\]
As $\Pam_{c_0}^N$ is product measure, the $\ln$ term is written as $ \sum_{k\notin \D_{s,N}}
 \ln E_{\Pam_{c_0}^N} [ e^{\xi(k)} ]$
  which is $O(N)$ by Lemma \ref{lem: order of R}.
Thus, the condition $H_N=O(N)$ specifies $O(N)$ particles on the sites $\T_N\setminus \D_{s,N}$ for times $t\geq 0$.

On the other hand, at a super-slow site $k_{j,N}\in \D_{s,N}$, the $\Pam_{c_0}^N$-particle number is typically of order $O(N^{\beta_j/\alpha})$, as discussed in Section \ref{sec: static limit}. Then, because $H_N = O(N)$, as consequence of the entropy inequality, we may conclude,
with respect to $\mu_t^N$, that
 $N^{-\beta_j/\alpha} \xi(k_{j,N})$  converges in probability to $\left(\lambda_j \Phi(c_0)\right)^{1/\alpha}$; see near \eqref{eqn: initial slow site esti} for a proof.

Therefore, 
the net exchange of particle numbers between super-slow sites $\D_{s,N}$ and the rest of the system is of order $N$ and 
the total particle number on $\T_N\setminus \D_{s,N}$ remains  $O(N)$ for all times $t\geq 0$.

\smallskip
We now turn to case when $g$ is bounded.
 Since in this setting $J_s=\emptyset$, that is $\D_s = \emptyset$ (Condition \ref{no J_s in bounded}), 
by the previous entropy inequality discussion, it follows that the total number of particles in $\T_N$ is of order $N$ for all times $t\geq 0$.

\subsection{Basic coupling} 
\label{sec: coupling}
We discuss the use of the stochastic boundedness assumption (3) and (3$'$) in Condition \ref{condition on mu N}.
 Since $g(n)$ is an increasing function in $n$, the dynamics generated by $L_N$ is `attractive':
if initially $\xi_0$ is distributed according to measures $\mu\leq \nu$, then we have $\mu_t \leq  \nu_t$ for all $t\geq 0$ where $\mu_t$ and $\nu_t$ are distributions of $\xi_t$
(cf.\,\cite{A}, Chapter II in \cite{L}).

In the case when $g$ of $n^\alpha$ type, as $\Pam_{c'}^N$ is invariant, the assumption $\mu^N \leq \Pam_{c'}^N$
 implies that $\mu_t^N \leq \Pam_{c'}^N$ for all $t\geq 0$.
However, when $g$ is bounded, 
the domain for $c$ in $\Pam_c^N$ is $c< R(1/\lambda_{\max})$ (cf.~Proposition \ref{prop: static limit, bounded g}).
Then, an assumption 
$\mu^N \leq \Pam_{c'}^N$
would imply that $\pi$, the macroscopic profile associated to $\mu^N$, is $\pi(dx) = \rho_0(x)dx$ with $\|\rho_0\|_\infty< R(1/\lambda_{\max})<\infty$.
To accommodate more initial profiles $\pi$ (and observe more involved limit evolutions), 
we have assumed $(3')$ in Condition \ref{condition on mu N} instead, that is, for all coordinately increasing  $f:\Omega_N \mapsto \R$ depending only on $\{\xi(k)\}_{k\notin \D_{c,N}}$,
we have $E_{\mu_t^N} [f(\xi)] \leq E_{\kappa_{c'}^N} [f(\xi)]$.

We now illustrate how we will apply attractiveness under assumption $(3')$. Instead of an evolution with respect to particle numbers in $\N_0$ and configurations in $\Omega_N^{\T_N}$, we may consider a dynamics corresponding to $\overline \N_0:=\N_0\cup \{\infty\}$ and $\overline \Omega_N := \overline \N_0^{\T_N}$. 
Recall the constant $c'$ in $(3')$.  Define $\overline g_{k,N}(\cdot)$  by $\overline g_{k,N}(n) = g_{k,N}(n)$ for $n\in \N_0$ and $\overline g_{k,N}(\infty) = \Phi(c')$. Consider now the following generator on $\overline\Omega_N$
\[
\overline L_N f(\xi) =
\sum_{k\in \T_N} 
\Big\{
\overline g_{k,N}(\xi(k))
\big(f(\xi^{k,k+1}) - f(\xi) \big)
+
\overline g_{k,N}(\xi(k)) 
\big(f(\xi^{k,k-1}) - f(\xi) \big)
\Big\} .
\]
Here, $\xi^{x,y}$ is defined as in \eqref{def: xi xy} by using the convention $\infty\pm1 = \infty$.
When starting with configurations in $\Omega_N$, $\overline L_N$ coincides with $L_N$. Once a site starts with $\infty$ particles, it will serve as a `reservoir' which pumps particles into its neighbors at the rate of $\Phi(c')$. 
Let $\delta_{\infty}$ be the Dirac measure on the extended number $\infty$.  Define $\overline \kappa_{c'}^N$ as the product  measure on $\overline \Omega_N$ that coincides with $\kappa_{c'}^N$ on $k\notin \D_{c,N}$ and has marginal $\delta_\infty$ for $k\in \D_{c,N}$.
It is straightforward to check that $\overline\kappa_{c'}^N$ is invariant under $\overline L_N$.  Also, we have that $\mu^N \leq \overline \kappa_{c'}^N$ by $(3')$ in Condition \ref{condition on mu N}.
As attractiveness and the basic coupling are still in effect for $\overline L_N$, we obtain that the later time $t$ distribution satisfies $\mu_t^N \leq \overline \kappa_{c'}^N$ for any $t\geq 0$.  In particular, for a coordinate-increasing function $f(\xi)$ which does not depend on $\{\xi(k)\}_{k\in \mathcal{D}_{c,N}}$, we have $E_{\mu^t_N}[f] \leq E_{\overline\kappa^N_{c'}}[f]=E_{\kappa^N_{c'}}[f]$.

Finally, as a general remark, we make technical use of `attractiveness' most essentially in the cutoff of large densities in the local replacements (Lemma  \ref{lem: 1 block} and Lemma \ref{lem: 2 blocks}),
and in proving the limit mass profile has $L^2$ absolutely continuous part (Lemma \ref{lem: abs continuity}).
The latter property is needed in the proof of uniqueness of weak solutions (Theorem \ref{thm: uniqueness} and  \ref{thm: uniqueness, bounded g}).

\subsection{Local equilibria}
\label{sec: local equilibria}
We now give explicit examples of initial measures that satisfy the Condition \ref{condition on mu N}.
These examples will be denoted by 
$\mu^N_{\text{le}}$ as they are related with the usual `local equilibrium' measures in setting without defects (cf. \cite{KL}).

Let $\pi$ be as in \eqref{eqn of pi}
with $\rho_0 (x)\in L^\infty(\T)$ and $\m_{0,j}\geq0$ for each $j\in J_c$.  
For each $k\in \T_N$, define
\[
 \rho_{k,N} = N\int_{(k-1)/N}^{k/N} \rho_0(x) dx.
\]
We will construct $\{\mu^N\}$ separately for $g$ of $n^\alpha$ and bounded types.
Consider first $g$ of $n^\alpha$ type.
Fix $c_0>0$ and define 
\[
\varphi_{k,N} =
\begin{cases}
\Phi(\rho_{k,N})&
 k\in\T_N \setminus \D_N\\
 0& k=k_{j,N}\in \D_{b,N}\\
 \left( N {\m_{0,j}} \right)^\alpha & k=k_{j,N}\in \D_{c,N}\\
\lambda_j N^{\beta_j} \Phi (c_0) & k=k_{j,N}\in \D_{s,N}.
\end{cases}
\]
For each $N\in \N$, let
$\mu^N_{\text{le}}$ be the product measure on
$\Omega_N$ with marginals given by
\[
\mu^N_{\text{le}} (\xi(k) = n)  \;=\;
\Pb_{\varphi_{k,N}}(n), \quad \text{for } k\in \T_N \,,\,  n\geq 0.
\]
\begin{Lem}
Suppose $g$ is of $n^\alpha$ type. Then, $\{\mu^N = \mu^N_{\text{le}}\}$
satisfies Condition \ref{condition on mu N}. 
\end{Lem}
\begin{proof}
Let $c'$ be such that $\Phi(c') = \max\{ \Phi\left(  \|\rho_0\|_\infty\right), ( \m_{0,j} )^\alpha, \Phi(c_0)  \} _{j\in J_c} $.
As $\Pb_{\phi_1} \leq \Pb_{\phi_2}$ if $\phi_1\leq \phi_2$ (cf. \cite{L}, pp.~32),
we have that the product measure $\mu^N_{\text{le}}$ is stochastically bounded by $\Pam_{c'}^N$ for all $N$.
As $G$ is uniformly continuous,
that $\mu^N_{\text{le}}$ is associated with the given macroscopic profile $\pi$
holds straightforwardly from Chebychev's inequality and Lemma \ref{lem: order of R}.

It remains to check the desired entropy bound
$H(\mu^N_{\text{lc}} | \Pam_{c_0}^N) = O(N)$.
Let $\varphi = \Phi(c_0)$. We compute
 \[
 H(\mu^N_{\text{le}} | \Pam_{c_0}^N)
 =
 \sum_{j\in J}H(\Pb_{\varphi_{k_{j,N},N}} | \Pb_{\lambda_j N^{\beta_j} \varphi})
 +
  \sum_{k\in \T_N\setminus \D_N}H(\Pb_{\varphi_{k,N}} | \Pb_{\varphi}).
 \]
Note that,  for any $\phi_1\geq 0$ and $\phi_2>0$,
 \[
 H(\Pb_{\phi_1} | \Pb_{\phi_2})
 =
 \ln\dfrac{\phi_1}{\phi_2}E_{\Pb_{\phi_1}} [X] 
 +
 \ln \dfrac{Z(\phi_2)}{Z(\phi_1)}
 \]
 where $X$ is defined as $X(n)=n$.
We also adopt $0\ln 0 = 0$ in the case $\phi_1=0$.
By Lemma \ref{lem: order of R}, we conclude the desired relative entropy bound $  H(\mu^N_{\text{le}} | \Pam_{c_0}^N) = O(N)$.
\end{proof}

We now consider the $g$ bounded case.  Note by assumption (Condition \ref{no J_s in bounded}) that $J_s=\emptyset$ in this setting.
Let
$\tilde \varphi_{k,N}=
\Phi(\rho_{k,N})$ for $k\in\T_N \setminus \D_N$  and $\tilde \varphi_{k,N}= 0$
 for $k\in \D_{b,N}$.  Also, let $\Pb'_{\lambda}$ denote the Poisson distribution with mean $\lambda$.
 
 We define $\tilde \mu_{\text{le}}^N$ as the product measure with marginals $\Pb_{\tilde \varphi_{k,N}}$ on sites $k\notin \D_{c,N}$ and $\Pb'_{\m_{0,j} N}$ for $k=k_{j,N} \in \D_{c,N}$.  
It is straightforward that $(1)$ and $(3')$ in Condition  \ref{condition on mu N} hold with $\mu^N = \tilde \mu_{\text{le}}^N$.  The choice of Poisson distributions at sites $\D_{c,N}$ allows for some explicit computation.

 Fix any $c_0\in (0,R(1/\lambda_{\max}))$,
 we now argue that $H(\tilde \mu_{\text{le}}^N | \Pam_{c_0}^N) =O(N)$.
 It suffices to check that $H(\Pb'_{aN} | \Pb_\varphi) = O(N)$ for any fixed $a\geq0$ and $\varphi\in(0,1)$, where $g_\infty =1$. 
 To see this, let $f_N(n) = \Pb'_{aN} (X=n)$ and $f(n) = \Pb_\varphi(X=n)$ where $X(n)=n$. 
 Then
 $H(\Pb'_{aN} | \Pb_\varphi) = E_{\Pb'_{aN} }[\ln f_N(X)] - E_{\Pb'_{aN} }[\ln f(X)]$.
 The term $E_{\Pb'_{aN} }[\ln f_N]$ is computed as $aN \ln(aN)  - aN - E_{\Pb'_{aN}} [\ln X!]$.
 By Stirling's formula, $n! \geq \sqrt{2\pi n}  e^{-n} n^n\geq e^{-n} n^n$, and Jensen's inequality, we have 
 $E_{\Pb'_{aN}} [\ln X!] \geq E_{\Pb'_{aN}} [X\ln X - X]\geq aN \ln(aN) - aN$. Therefore, we have $E_{\Pb'_{aN} }[\ln f_N]$ is $O(N)$.
 For the term $E_{\Pb'_{aN} }[\ln f]$, we may write it as $aN \ln \varphi - \ln Z(\varphi) - E_{\Pb'_{aN}}[\ln g(X)!]$. Note that $g(X) \leq 1$, so that $E_{\Pb'_{aN} }[\ln f] =O(N)$.  Hence, $H(\tilde \mu_{\text{le}}^N | \Pam_{c_0}^N) =O(N)$.

We now summarize the above calculations.
\begin{Lem}
\label{g bounded le lem}
Suppose $g$ is bounded. Then $\mu^N = \tilde \mu^N_{\text{le}}$
satisfies Condition \ref{condition on mu N}.
\end{Lem}


\section{Results}
\label{sec: results}
Suppose that $\{\mu^N\}$ satisfies Condition \ref{condition on mu N}.  Consequently, $\mu^N$
has macroscopic profile 
\[
\pi(dx) = \rho_0(x)dx + \sum_{j\in J_c} \m_{0,j}\, \delta_{x_j} (dx)
\]
 on $\T\setminus \D_s$ 
and we have  $H(\mu^N | \Pam_{c_0}^N) = O(N)$ for some $c_0\geq 0$. 

On $\T_N$, we will observe the evolution of the zero-range process
speeded up by $N^2$, in diffusive scale. 
Denote the process $\eta_t: = \xi_{N^2t}$,
generated by $N^2L_N$ (cf.~\eqref{eqn: generator L}), for times $0\leq t\leq T$.  Here, $T>0$ refers to
a fixed time horizon.  We will access the space-time structure of the
process through the scaled mass empirical measure:
\begin{equation} \label{pi_t w/o D_s}
\pi_t^N(dx)
:=
\dfrac 1 N 
\sum_{k\in \T_N \setminus \D_{s,N}}
\eta_t (k) \delta_{k/N}(dx).
\end{equation}

Throughout, we will view each $\pi^N_t$ as a member of $\Mb$, the space of finite nonnegative measures on $\T\setminus \D_s$.
 We will place a metric
$d(\cdot,\cdot)$ on $\Mb$ which realizes the vague convergence on $\T\setminus \D_s$,
(see Section \ref{sec: tightness} for a definitive choice).  Here, the
trajectories $\{\pi^N_t: 0\leq t\leq T\}$ are elements of the
Skorokhod space $D([0,T],\Mb )$, endowed with the associated Skorokhod topology.

In the following, the process
measure and associated expectation governing $\eta_\cdot$ starting
from $\mu$ will be denoted by $\P_\mu$ and $\E_\mu$.  When the process
starts from $\{\mu^N\}_{N\in \N}$, in the class satisfying Condition
\ref{condition on mu N}, we will denote by $\P_N:= \P_{\mu^N}$ and
$\E_N:= \E_{\mu^N}$, the associated process measure and expectation.

\subsection{Hydrodynamic limits}
We now state our main results. Recall the initial profile $\pi$ and the constant $c_0$ from the beginning of this section.
\begin{Thm}
\label{main thm: k alpha}
Asssume $g$ is of $n^\alpha$ type.
Then, for any $t > 0$, test function
$G\in C(\T)$, and $\delta>0$, we have
\begin{equation} \label{HDL: alpha g}
\lim_{N\to \infty} \P_N
\Big[ \Big|
\langle G, \pi^N_t \rangle
-
\langle G, \pi_t \rangle
\Big|
>\delta
 \Big]
 =0,
\end{equation}
where $\pi_t(dx)=\rho(t,x)dx + \sum_{j\in J_c} \m_{j}(t)\, \delta_{x_j} (dx)$ is the unique weak solution to 
\begin{equation} \label {pde: k alpha}
\begin{cases}
\partial_t \pi_t
=
\partial_{xx} \Phi(\rho(t,x)),
& x\in \T \setminus \D_s,\ t\in(0,T),\\
\pi_t \vert_{t=0} = \pi, \quad
\Phi(\rho(t,x_j))  = \Phi(c_0), &t\in(0,T),\ j\in J_s,\\
\m_{j}(t) = \left( \lambda_j \Phi(\rho(t,x_j)) \right)^{1/\alpha}, & t\in(0,T), \, j\in J_c.
\end{cases}
\end{equation}
\end{Thm}

\begin{Thm}
\label{main thm: bounded g}
Assume $g$ is bounded.
Then, for any $t > 0$, test function
$G\in C(\T)$, and $\delta>0$, we have
\begin{equation} \label{HDL: bounded g}
\lim_{N\to \infty} \P_N
\Big[ \Big|
\langle G, \pi^N_t \rangle
-
\langle G, \pi_t \rangle
\Big|
>\delta
 \Big]
 =0,
\end{equation}
where $\pi_t(dx)=\rho(t,x)dx + \sum_{j\in J_c} \m_{j}(t)\, \delta_{x_j} (dx)$ is the unique weak solution to 
\begin{equation} \label {pde: bounded g}
\begin{cases}
\partial_t \pi_t
=
\partial_{xx} \Phi(\rho(t,x)),
& x\in \T,\ t\in(0,T),\\
  \pi_t \vert_{t=0} = \pi,
\quad  \Phi(\rho(t,x_j)) \leq 1 / \lambda_j,
&  t\in(0,T), \, j\in J_c,\\
\m_j(t) = \m_j(t)\id_{\Phi(\rho(t,x_j)) = 1/ \lambda_j} ,
 & t\in(0,T), \, j\in J_c.
 \end{cases}
\end{equation}
\end{Thm}

We now define the weak solutions to the limit PDE \eqref{pde: k alpha} and \eqref{pde: bounded g}.
\begin{Def} \label{def: weak deriv}
Let $f(t,x)$ and $g(t,x)$ be in $L_{loc}^1([0,T]\times D)$ where $D$ is a domain of $x$. We say $f$ is weakly differentialble with respect to $x\in D$ if for all 
$G(t,x) \in C_c^{0,1} ( [0,T]\times D)$ that
\[
\int_0^T \int_\T 
\partial_x G(t,x) f(t,x) dx dt
=
-\int_0^T \int_\T 
G(t,x) g(t,x) dx dt;
\]
The weak derivative will be denoted by $\partial_x f(t,x)$ and $\partial_x f(t,x):=g(t,x)$.
\end{Def}
We comment that if $f(t,x)$ is weakly differentiable with respect to $x\in D$ as defined above,
 then for a.e.\ $t\in[0,T]$, $f(t,\cdot)$ is absolutely continuous and $f(t,b) - f(t,a)= \int_a^b \partial_x f(t,x)\, dx$ for all connected $a,b\in D$. In particular, the evaluation of
$f(t,x') = f(t,x)\vert_{x=x'}$ at a given value $x'\in D$ is well defined.

\begin{Def} \label{def: weak sln}
We say $\pi_t(dx)=\rho(t,x)dx + \sum_{j\in J_c} \m_{j}(t)\, \delta_{x_j} (dx)$
 is a weak solution to the system \eqref{pde: k alpha}
if
\begin{enumerate}
\item $\rho(t,x)$ is in $L^2([0,T]\times \T)$ and $\Phi(\rho(t,x))$ is weakly differentiable with respect to $x\in \T$ with 
$\partial_x \Phi(\rho(t,x))\in L^2( [0,T]\times \T)$;
\item $\Phi(\rho(t,x_j))  = \Phi(c_0)$, for almost all $t\in (0,T)$ and all $x_j\in \D_s$;
\item $\m_{j}(t) = \left( \lambda_j \Phi(\rho(t,x_j)) \right)^{1/\alpha}$ for almost all $t\in (0,T)$ and $j\in J_c$;
\item 
for all $G(t,x) \in C_c^\infty \left( [0,T)\times (\T\setminus \D_s)  \right)$
\begin{equation} \label {weak sln, def}
\begin{split}
\int_0^T \int_\T \partial_tG(t,x) 
\pi_t(dx) dt
+ \int_\T G(0,x) \pi(dx) 
+\int_0^T \int_\T 
\partial_{xx}G(t,x) \Phi(\rho(t,x)) dx dt
=0.
\end{split}
\end{equation}
\end{enumerate}
\end{Def}

\begin{Def} \label{def: weak sln, bounded g}
We say $\pi_t(dx)=\rho(t,x)dx + \sum_{j\in J_c} \m_{j}(t)\, \delta_{x_j} (dx)$
 is a weak solution to the system \eqref{pde: bounded g}
if
\begin{enumerate}
\item $\rho(t,x)$ is in $L^2([0,T]\times \T)$ and $\Phi(\rho(t,x))$ is weakly differentiable with respect to $x\in \T$ with 
$\partial_x \Phi(\rho(t,x))\in L^2( [0,T]\times \T)$;
\item 
$\Phi(\rho(t,x_j))  \leq 1 / \lambda_j $, for almost all $t\in (0,T)$ and all $x_j\in \D_c$;
\item 
$\m_j(t) = \m_j(t)\id_{\Phi(\rho(t,x_j)) = 1/ \lambda_j} $ for almost all $t\in (0,T)$ and $j\in J_c$;
\item 
the weak formulation \eqref{weak sln, def} holds 
for all $G(t,x) \in C_c^\infty \left( [0,T)\times \T   \right)$.
\end{enumerate}
\end{Def}

\begin{Rmk}
\label{rmk on w-deriv}
In Definitions \ref{def: weak sln} and \ref{def: weak sln, bounded g},
we have assumed 
 $\Phi(\rho(t,x))$ to be weakly differentialbe with respect to $x\in\T$. 
 Note that $\Phi(\cdot)$ is invertible 
 and the inverse function 
 $\Phi^{-1}(\cdot)$ is a continuous mapping from $ [0,r_g)$ to $[0,\infty)$
 where $r_g=\infty$ for $g$ of $n^\alpha$ type and $r_g=1$ for $g$ bounded (cf.~Section \ref{sec: invariant measures}).
 Then $\rho(t,x) = \Phi^{-1}\left( \Phi(\rho(t,x))\right)$ is continuous in $x$ for a.e.~$t\in[0,T]$ (viewed as $[0,\infty]$ valued function for $g$ bounded).
 In particular, the evaluation of $\Phi(\rho(t,x))$ at $x=x_j$ may be written in terms 
 $\rho(t,x_j)$ instead.

Moreover, notice that, under the settings of Theorems 
 \ref{main thm: k alpha} and \ref{main thm: bounded g}, it holds $\|\rho\|_\infty<\infty$; see Lemma \ref{lem: abs continuity}.
As $\Phi^{-1}(\cdot)$ is Lipschitz on 
$[0,\phi]$ where $\phi = \Phi(\|\rho\|_\infty)$ (cf.~Section \ref{sec: invariant measures}), we have that
$\rho(t,x)$ is also weakly differentiable with respect to $x\in \T$ and 
$\partial_x \rho(t,x)\in L^2( [0,T]\times \T)$.
\end{Rmk}

To illustrate the relation between boundary conditions and effects on defect sites, we consider the case with  only a single defect site. 

\begin{Exam}[Effects of a single slow site]
\label{sec: single slow site}
\rm
As an example to illustrate how defects affect the macroscopic bulk evolution, we look at systems with a single defect site.   
Without loss of generality, we may assume the defect location is at $0$ and the altered jump rate is $(\lambda N^\beta)^{-1} g(\xi(0))$.

Consider first $g$ of $n^\alpha$ type.
By Theorem \ref{main thm: k alpha}, the hydrodynamic limit $\pi_t$ is governed by the PDE \eqref{pde: k alpha}. 
As the defect site is at $x=0$, we have $\pi_t(dx) = \rho(t,x)dx$ when restricted to open interval $(0,1)$. 
Indeed, \eqref{pde: k alpha} is now reduced to a nonlinear heat equation regarding only $\rho(t,x)$ with different boundary conditions at $x=0$ and $x=1$ depending on the value of $\beta$.
Precisely, we have the following.
\begin{enumerate}
\item When $\beta<\alpha$, the defect site is invisible in the limit and \eqref{pde: k alpha} becomes
\begin{equation} \label {pde: regular}
\begin{cases}
\partial_t \rho(t,x)
=
\partial_{xx} \Phi(\rho(t,x)),
\quad x\in \T,\ t\in(0,T),\\
\rho(0,x) = \rho_0(x),
\end{cases}
\end{equation}
that is $\rho(t,x)$ satisfies periodic boundary conditions.
\item When $\beta=\alpha$, we have $\pi_t(dx) = \rho(t,x)dx + \m(t)\delta_0(dx)$ with 
$\m (t) = ( \lambda \Phi(\rho(t,0)) )   ^{1/\alpha}$.
 The atomic mass $\m(t)$ can also be expressed in terms of the its initial value $\m_0$ and net change of the bulk mass as the total mass is conserved:
\[
\m_0+\int_0^1 \rho_0(x)dx
=
\m(t)+\int_0^1 \rho(t,x)dx,
\quad
\text{for all }t>0.
\]
Therefore, we have 
$ \Phi(\rho(t,0)) =\lambda^{-1} \big[ \m_0 + \int_0^1 ( \rho_0 (x) -  \rho(t,x) ) dx  \big ]^{\alpha}$.
Noticing that $x=0$ and $x=1$ coincide on $\T$,
we obtain
\[
\begin{cases}
\partial_t \rho(t,x)
=
\partial_{xx} \Phi(\rho(t,x)) , \quad x\in (0,1),\ t\in(0,T), \\
\rho(0,x) = \rho_0(x), \\
 \Phi(\rho(t,0)) = \Phi(\rho(t,1))
=\lambda^{-1} \big [ \m_0 + \int_0^1 \left( \rho_0 (x) -  \rho(t,x) \right) dx \big ]^{\alpha}.
\end{cases}
\]
We comment that, in a system with more than one critical slow sites ($\beta=\alpha$), conservative of mass is not sufficient to determine the atomic masses individually. For a closed equation, one needs to stay with a form such as \eqref{pde: k alpha}.
\item When $\beta>\alpha$, as $\D=\D_s=\{0\}$, the PDE \eqref{pde: k alpha} is
\[
\begin{cases}
\partial_t \rho(t,x)
=
\partial_{xx} \Phi(\rho(t,x)),
\quad x\in (0,1),\ t\in(0,T),\\
\rho(0,x) = \rho_0(x),\\
\rho(t,0) = \rho(t,1) = c_0.
\end{cases}
\]
\end{enumerate}

\smallskip
We now turn to case when $g$ is bounded.
By Theorem \ref{main thm: bounded g}, the macroscopic flow of $\pi_t$ is described by \eqref{HDL: bounded g}.
Notice that here we have $\beta\leq 0$.
When the defect is a fast site, that is $\beta<0$ or $\beta=0$ with $\lambda<1$, the defect site is invisible macroscopically and the limit evolution is the usual nonlinear heat equation  \eqref{pde: regular} with periodic boundary condition. When the defect is a slow site, that is $\beta=0$ and $\lambda>1$,
we can write $\mathfrak{m}(t)$, the atomic mass at the slow site, in term of `mass conservation' as in part (2) above to form a closed equation of $\rho(t,x)$:
\[
\begin{cases}
\partial_t \rho(t,x)
=
\partial_{xx} \Phi(\rho(t,x)),
\quad x\in \T,\ t\in(0,T),\\
   \rho(0,x) = \rho_0(x),
  \quad
   \rho(t,0) =\rho(t,1) \leq \Phi^{-1}(1 / \lambda),\\
\m(t) = \m(t)\id_{\rho(t,0) = \Phi^{-1} (1/ \lambda)},\quad
\m(t)=\m_0 + \int_0^1 \left( \rho_0 (x) -  \rho(t,x) \right) dx.
\end{cases}
\]
Informally, during the evolution, we observe periodic boundary conditions when $\m(t)=0$ and  Dirichlet boundary conditions $\rho(t,0)=\rho(t,1)=\Phi^{-1}(1/\lambda)$ when $\m(t)>0$.
\end{Exam}

\subsection{Remarks on general defects}
In our model, we have assumed that the jumping rate at a defect site $k_{j,N}$ is slowed down by a factor of $\lambda_j N^{\beta_j}$. However, the form of $\lambda_j N^{\beta_j}$ is only a choice for convenience. In fact, 
Theorems \ref{main thm: k alpha} and \ref{main thm: bounded g}
hold for more general choices of slow down factors, as can be perused.

We may assume the jumping rate at a defect $k_{j,N}$ is $g(\cdot)/\beta_{j,N}$. (In terms of the current setting $\beta_{j,N} = \lambda_j N^{\beta_j}$.) Assume that $\beta_{j,N}>0$ and $\lim_{N\to \infty} \beta_{j,N} / g(N)$ exists (might be $\infty$ though). To extend Theorems \ref{main thm: k alpha} and \ref{main thm: bounded g} to the new setting, we need to clarify the sets of $J_s$, $J_c$ and $\lambda_j$ for $j\in J_c$.
When $g$ is of $n^\alpha$ type, $J_c=\{j\in J: \beta_{j,N}\sim N^\alpha \}$ and when $g$ is bounded, $J_c=\{j\in J: \lim_{N\to\infty}\beta_{j,N}>1\}$.
In either case of $g$, the super-critical index set $J_s$ consists of $j$'s such that 
$\beta_{j,N} \gg g(N)$ and $J_b = J\setminus \{J_s\cup J_c\}$.
The constant $\lambda_j$ for $j\in J_c$ can be defined as $\lambda_j :=\lim_{N\to\infty} \beta_{j,N} / g(N)$.

\section{Proof outline}
\label{section: martingale}

In this section, we prove  
Theorems \ref{main thm: k alpha} and \ref{main thm: bounded g}.
We begin with analyzing $\langle G, \pi^N_t\rangle$ by computing its stochastic
differential in terms of certain martingales.

\subsection{Stochastic differentials}
Let $G$ be a smooth function on $[0,T] \times \T$ and we write $G_t(x) := G(t,x)$.
For $t\geq 0$, consider $\wpi_t^N$, the emprical measure associated with $\eta_t$ on $\T_N$:
\begin{equation} \label{eqn: pi_t hat}
\wpi_t^N(dx)
:=
\dfrac 1 N 
\sum_{k\in \T_N}
\eta_t (k) \delta_{k/N}(dx).
\end{equation}
Note that if $G$ has compact support on $[0,T] \times \T\setminus \D_s$,
 it holds $\langle G_t, \wpi_t^N\rangle = \langle G_t, \pi_t^N\rangle$ for all $N$ large (cf.~\eqref{pi_t w/o D_s}). 

We have
\begin{equation*}
M^{N,G}_t
=
\left\langle G_t,   \wpi^N_t \right\rangle
-
\left\langle G_0,  \wpi^N_0 \right\rangle
-
\int_0^t 
\Big\{
 \left\langle  \partial_s G_s, \wpi_s^N\right\rangle + N^2L_N \left\langle G_s,   \wpi^N_s \right\rangle 
 \Big\}ds
\end{equation*}
is a mean zero martingale.
Denote the discrete Laplacian $\Delta_{N}$ by
\begin{equation*}
\Delta_N G\Big(\frac k N\Big)
:=
N^2 \Big(G\Big(\frac {k+1}N\Big)+G\Big(\frac {k-1}N\Big)-2G\Big(\frac k N\Big)\Big).
\end{equation*}
Then, for $N$ large, we compute 
\begin{equation} \label{gen_comp}
N^2 L _N \left\langle G,  \wpi^N_s \right\rangle 
= 
\dfrac1N \sum_{k\in \T_N}
 \Delta_{N} G_s\Big(\frac k N \Big) g_{k,N}(\eta_s(k)) .
\end{equation}

The quadratic variation of $M^{N,G}_t$ is given by
\begin{equation*}
\langle M^{N,G} \rangle_t
=
\int_0^t 
\left\{
N^2L_N \left(\left\langle G_s,  \wpi^N_s \right\rangle ^2 \right)
-
2 \left\langle G_s,  \wpi^N_s \right\rangle
 N^2 L_N \left\langle G_s,  \wpi^N_s\right\rangle 
\right\}
ds
\end{equation*}
which by standard calculation equals
\begin{equation*} \label {quadratic variation ln k}
\int_0^t 
\sum_{k\in \T_N}
g_{k,N}(\eta_s(k)) 
\left\{
\Big(G_s\big(\frac {k+1}N\big) - G_s\big(\frac kN\big)\Big)^2
+
\Big(G_s\big(\frac {k-1}N\big) - G_s\big(\frac kN\big)\Big)^2
\right\}
ds.
\end{equation*}

This variation may be bounded as follows.
\begin{Lem} \label{lem: martingale bounds} 
For any test function $G(x)\in C^\infty ( \T)$, there is a constant $C$ independent of $N$ such that, for all $N$ large,
 \[
 \sup_{0\leq t\leq T}\E_N \langle M^{N,G}\rangle_t \leq C N^{-1}.
 \]
\end{Lem}
\begin{proof}
Since $G$ is smooth, we obtain, for $N$ large
\[
\E_N\langle M^{N,G} \rangle_t
 \leq 
 2 (\| \partial_x G\|_\infty)^2 N^{-1}
\E_N
 \Big [
 \dsp\int_0^t 
\dfrac1N \sum_{k\in \T_N}
g_{k,N}(\eta_s (k))
 ds  \Big ] .
\]
Then, the result follows from
\eqref{expected_g_all} in
 the next Lemma \ref{lem: total nmbr bound}.
\end{proof}

It will be useful to bound local particle numbers and rates.  Although there are different ways to prove parts of the following statements, say using the entropy inequality and $g(n)/n^\alpha <C$ for $0\leq \alpha\leq1$, we will use attractiveness to facilitate a short proof.

\begin{Lem} \label{lem: total nmbr bound}
We have the following:
\begin{enumerate}
\item 
The expectation of total mass at all but super-critical sites is uniformly bounded:
\begin{equation} \label{expected_number}
\sup_{N\in \N} \sup_{t\geq 0} 
\E_N \Big [ \dfrac1N \sum_{k\in \T_N\setminus \D_{s,N}} \eta_t(k) \Big ]
< \infty;
\end{equation}
\item
The expected particle number at each regular site $k\notin \D_N$ is uniformly bounded:
\begin{equation}
 \label {nbhd expected nmbr}
\sup_{N\in \N} 
 \sup_{t\geq 0} 
\sup_{k\notin \D_N}
 \E_N \Big [ \eta_t(k) \Big ]
<\infty;
\end{equation}

\item
The expectation of weighted jump rate $N^{-1}g_{k,N}$ vanishes uniformly as $N\to \infty$:
\begin{equation} 
\label{expected g bound}
\lim_{N\to \infty} \sup_{t\geq 0} 
\sup_{k\in \T_N}
\E_N \Big [ \dfrac1N g_{k,N} \left(\eta_t(k) \right)\Big ]
=0;
\end{equation}
\item 
The expectation of total weighted jump rate   is uniformly bounded:
\begin{equation} \label {expected_g_all}
\sup_{N\in \N} \sup_{t\geq 0} 
\E_N \Big [ \dfrac1N \sum_{k\in \T_N} g_{k,N}(\eta_t(k)) \Big ]
< \infty.
\end{equation}
\end{enumerate}
\end{Lem}

\begin{proof} 
First, in both rate $g$ settings,
the total mass estimate \eqref{expected_number} follow directly from the initial bound of the entropy $H(\mu^N | \Pam_{c_0}^N)=O(N)$ and the entropy inequality; see Section \ref{subsec: relative entropy} for an explicit proof.

Moreover, \eqref{nbhd expected nmbr} and \eqref{expected g bound} follow by straightforwardly applying attractiveness, $\mu_t^N \leq \Pam_{c'}^N$ and
$\mu_t^N \leq \overline\kappa_{c'}^N$ for $g$ of $n^\alpha$ type and bounded type respectively
(see Section \ref{sec: coupling}).
  Indeed,  in the $g(n)\sim n^\alpha$ case,, we have, for all $k\not\in \D_N$, that $\E_N[\eta_t(k)] \leq E_{\Pam_{c'}^N}[\xi(k)] = c'$.  Also, as $g$ is increasing, 
$\E_N[g_{k,N}(\eta_t(k))] \leq E_{\Pam_{c'}^N}[g_{k,N}(\xi(k))] = \Phi(c')$ (cf. \eqref{Phi_eqn}).
When $g$ is bounded, we have $\E_N[\eta_t(k)] \leq E_{\overline\kappa_{c'}^N}[\xi(k)] = c'$ for regular sites $k\notin \D_N$. For the bound of $\E_N[g_{k,N}(\eta_t(k))]$, it holds trivially that $\E_N[g_{k,N}(\eta_t(k))] \leq g_\infty=1$ when $k\notin \D_N$ and $\E_N[g_{k,N}(\eta_t(k))] \leq1/ \lambda_j$ when $k=k_{j,N}$ with $\beta_j = 0$. When $k=k_{j,N}$ with $\beta_j<0$, we use attractiveness again $\E_N[g_{k,N}(\eta_t(k))] \leq E_{\overline\kappa_{c'}^N}[g_{k,N}(\xi(k))] = \Phi(c')$.

Finally, to show \eqref{expected_g_all},
 we separate the sum of $k$ into two sums consisting of the defect sites $\D_N$ and  the regular sites respectively:  The sum over $\D_N$ is bounded via \eqref{expected g bound} and  that $\D_N$ has finite cardinality $n_0$ independent of $N$.  
 Recall that $g_{k,N}(\cdot) = g(\cdot)$ for $k\notin \D_N$.
 Now, bounding $g(n)\leq g^*n$, the sum over the regular sites is bounded due to \eqref{expected_number}.
\end{proof}

\subsection{Proof outline of Theorems \ref{main thm: k alpha} and \ref{main thm: bounded g}}
\label{proof_section, critical}
We sketch the proofs of Theorems \ref{main thm: k alpha} and \ref{main thm: bounded g}.  Let $Q^N$ be the
probability measure on the trajectory space $D([0,T],\Mb )$ governing
$\pi_{\cdot}^N$ when the process starts from $\mu^N$.  By Lemma
\ref{lem: tightness}, the family of measures
$\left\{Q^N\right\}_{N\in \N}$ is tight with respect to the uniform
topology, stronger than the Skorokhod topology.
Let now $Q$ be any limit measure. 
 For both cases of $g$, we first show that $Q$ is supported on a class of weak solutions to the associated nonlinear PDE \eqref{pde: k alpha} or \eqref{pde: bounded g}. Then, by uniqueness results from Section  \ref{section: uniqueness}, we prove the desired results 
\eqref{HDL: alpha g} and \eqref{HDL: bounded g}.
\vskip .1cm

{\it Step 1.}
Let $G(t,x)$ be a smooth function with compact support in $[0,T) \times (\T\setminus \D_s)$.
Notice that $\langle G_t,\pi_t^N\rangle = \langle G_t,\wpi_t^N\rangle$ and
recall the martingale $M^{N,G}_t$ and its quadratic variation $\langle M^{N,G}\rangle_t$ in the last section.
By Lemma \ref{lem: martingale bounds}, we have $\E_N \big( M^{N,G}_T\big)^2 = \E_N \langle M^{N,G} \rangle_T$
vanishes as $N\to \infty$.
For each $\delta>0$, by Chebychev's inequality,
\begin{align*}
&&
\P_N
\Big[
\big| 
\big\langle G_T,  \pi^N_T \big\rangle
-
\big\langle G_0,  \pi^N_0 \big\rangle  
-
 \int_0^T \big( \big\langle \partial_s G_s,  \pi^N_s \big\rangle + N^2 L_N  \big\langle G_s,  \pi^N_s \big\rangle   \big)   ds \big|
>\delta
\Big]\\
&& 
=
\P_N
\Big[
\big| 
M_T^{N,G} \big|
>\delta
\Big]
\leq  \dfrac{1}{\delta^2}
 \E_N \big\langle M^{N,G}\big\rangle_T \ \rightarrow \ 0 \ \ {\rm as \ } N\to\infty.
\end{align*}
Note that $G_T(x)=0$ and recall the evaluation of $N^2L_N  \left\langle G_s,  \pi^N_s \right\rangle$ in \eqref{gen_comp}.
Then,
\begin{align} \label {eqn bf rplmt, discrete}
&\lim_{N\to \infty}
\P_N
\Big[
\Big| 
\left\langle G_0,  \pi^N_0 \right\rangle\\
&\ \ \ \ \ \ \ \ \ +
\int_0^T \Big\{ \left\langle \partial_s G_s,  \pi^N_s \right\rangle 
+
 \dfrac1N \sum_{k\in\T_N}
\Delta_{N} G_s\big(\frac k N \big) g_{k,N}(\eta_s(k))
   \Big\}   ds \Big|
>\delta
\Big]
=0. \nonumber
\end{align}
As $\E_N \big [ N^{-1} \sum_{k\in \T_N} g_{k,N}(\eta_t(k)) \big ]$ is uniformly bounded by \eqref{expected_g_all}, we may replace 
the discrete Laplacian $\Delta_N G_s(\cdot)$ in \eqref{eqn bf rplmt, discrete} by $\partial_{xx} G_s(\cdot)$.

Let $\D^{\e} =\cup_{j\in J} (x_j - \e, x_j+\e)$ 
and $F_\e(s,x) = \id_{\T\setminus \D^\e} (x)\partial_{ss} G(s,x)$.
To prepare for the `bulk' replacement in the next step, we will further replace  $\partial_{xx} G_s(\cdot)$ by its approximation
$F_\e(s,\cdot)$. It suffices to show
\[
\lim_{\e\to 0}
\sup_{N\in \N}
\sup_{t\geq 0}
\E_N \big [N^{-1} \sum_{k/N\in \D^{\e} } g_{k,N}(\eta_t(k)) \big] = 0.
\]
As the sum of $k$ is over a set of cardinality at most $(2\e N+1) n_0$, the desired limit follows from separating $k$'s into defects and regular ones, and applying \eqref{expected g bound} and \eqref{nbhd expected nmbr}. 
Therefore, we obtain
\begin{align}  \label {eqn before replace}
&\lim_{\e \to 0}
\lim_{N\to \infty}
\P_N
\Big[
\Big| 
\left\langle G_0,  \pi^N_0 \right\rangle\\
&\ \ \ \ \ \ \ \ \ +
\int_0^T \Big\{ \left\langle \partial_s G_s,  \pi^N_s \right\rangle 
+
 \dfrac1N \sum_{k\in\T_N}
F_\e \big(\frac k N \big) g_{k,N}(\eta_s(k))
   \Big\}   ds \Big|
>\delta
\Big]
=0. \nonumber
\end{align}

\vskip .1cm

{\it Step 2.}  We now replace the nonlinear term $g_{k,N}(\eta_s(k))$ by a
function of the empirical density of particles.  To be precise, let
$\eta^l (x) = \dfrac{1}{2l+1} \sum_{|y-x|\leq l} \eta(y)$, that is the
average density of particles in the box centered at $x$ with length
$2l+1$.
Therefore, using the Bulk Replacement Lemma (Lemma \ref{rplmt_global_lem}), we obtain from \eqref{eqn before replace},
\begin{equation}
\label{eqn after replace}
\begin{split}
&
\lim_{\e\to0}
\limsup_{\eps \to 0}
\limsup_{N\to \infty}
\P_N
\Big[
\Big| 
\left\langle G_0,  \pi^N_0 \right\rangle
+
\int_0^T \Big\{ \left\langle \partial_s G_s,  \pi^N_s \right\rangle 
\\
&\qquad\qquad\qquad
+
 \dfrac1N \sum_{k\in\T_N}
F_\e \big(s,\frac k N \big) \Phi\left(\eta_{t}^{\eps N}(k)\right)
   \Big\}   ds \Big|
>\delta
\Big]
=0.
\end{split}
\end{equation}

\vskip .1cm
{\it Step 3.}
For each $\eps>0$, take $\iota_\eps = (2\eps)^{-1} \id_{[-\eps,\eps]}$.
The average density $\eta_{t}^{\eps N}(k)$ is written as a function of the empirical measure $\pi^N_t$ 
\begin{equation*}
\eta_{t}^{\eps N}(k)
=
\dfrac{2\eps N}{2\theta N +1}
\langle \iota_{\eps}(\cdot -  k/N),  \pi^N_t \rangle.
\end{equation*}
Since $\Phi$ is Lipschitz continuous and the total number of particles has expectation of order $N$ on the bulk (cf.~Lemma \ref{lem: total nmbr bound}), we may replace $\eta_{t}^{\eps N}(k)$ by $\langle \iota_{\eps}(\cdot -  k/N),  \pi^N_t \rangle$.
Hence, we get from \eqref{eqn after replace} in terms of
the induced distribution $Q^N$ that
\begin{equation}
\label{eqn sum by integral}
\begin{split}
&
\lim_{\e\to0}
\limsup_{\eps \to 0}
\limsup_{N\to \infty}
Q^N
\Big[
\Big| 
\left\langle G_0,  \pi^N_0 \right\rangle
+
\int_0^T \Big\{ \left\langle \partial_s G_s,  \pi^N_s \right\rangle 
\\
&\qquad\qquad\qquad
+
\int_\T F_\e (s,x) 
\Phi \big(\langle \iota_{\eps}(\cdot -  x),  \pi^N_s\rangle \big) dx
   \Big\}   ds \Big|
>\delta
\Big]
=0.
\end{split}
\end{equation}
Notice that the discrete sum on $k$ is also replaced by the corresponding integral.

As the absolute value term in \eqref{eqn sum by integral} is continuous (with respect to the uniform topology) on $D([0,T],\Mb )$, the set of trajectories in \eqref{eqn sum by integral} is open, Taking $N\to \infty$, we obtain
\begin{equation}
\label{Q supports, with eps}
\begin{split}
&
\lim_{\e\to0}
\limsup_{\eps \to 0}
Q
\Big[
\Big| 
\left\langle G_0,  \pi_0 \right\rangle
+
\int_0^T \Big\{ \left\langle \partial_s G_s,  \pi_s \right\rangle 
\\
&\qquad\qquad\qquad
+
\int_\T F_\e (s,x) 
\Phi \big(\langle \iota_{\eps}(\cdot -  x),  \pi_s\rangle \big) dx
   \Big\}   ds \Big|
>\delta
\Big]
=0.
\end{split}
\end{equation}

\vskip .1cm
{\it Step 4.} We show in
Lemma \ref{lem: abs continuity} that $Q$ is supported on trajectories 
\[
\pi_s(dx) =
 \rho(s,x) dx + \sum_{j\in J_c} \m_j(s) \delta_{x_j}(dx).
 \]
Then, for $x\notin \D^\e$ and $\eps< \e$, 
 $\langle \iota_{\eps}(\cdot -  x),  \pi_s\rangle 
 = (2\eps)^{-1} \int_{x-\eps}^{x+\eps}  \rho(s,u)du$.
Note that $\Phi$ is Lipschitz and $\rho$ is integrable on $[0,T]\times \T$. Hence, for all $\e$ small, we have, $Q$-almost surely
\begin{equation*}
\lim_{\eps\to 0}
\int_0^T 
\int_\T
F_\e (s,x)
 \Phi(\langle \iota_{\eps}(\cdot - x),  \pi_s \rangle )
 dx  ds 
 =
 \int_0^T 
\int_\T
F_\e (s,x)
 \Phi(\rho(s,x))
 dx  ds .
\end{equation*}
Since almost sure convergence impiies convergence in probability, we obtain from  \eqref{Q supports, with eps} that, for all $\delta>0$
\begin{equation*}
\begin{split}
&
\lim_{\e\to0}
Q
\Big[
\Big|
\left\langle G_0, \pi_0 \right\rangle
+
\int_0^T \Big\{ \left\langle \partial_s G_s,  \pi_s \right\rangle
+
\int_\T
F_\e(s,x)
\Phi(\rho(s,x))
 dx  \Big\}  ds  \Big|
 >\delta \Big]
=0.
\end{split}
\end{equation*}
Taking $\e\to 0$ we may also replace $F_\e$ by $\partial_{xx} G$.
As $\delta$ is arbitrary, we have
\begin{equation*}
\begin{split}
Q\Big[
\left\langle G_0, \pi_0 \right\rangle
+
\int_0^T \Big\{ \left\langle \partial_s G_s,  \pi_s \right\rangle
+
\int_\T
\partial_{xx} G_s\left(x \right) 
\Phi(\rho(s,x))
 dx  \Big\}  ds =0
 \Big] = 1.
 \end{split}
\end{equation*}
By Condition \ref{condition on mu N}, the initial condition $\pi_0 = \pi$ holds.
Thus, we conclude the limit measure $Q$ is concentrated on trajectories $\pi_\cdot$ that satisfies the weak formulation  \eqref{weak sln, def}.

\vskip .1cm {\it Step 5.}
That $\rho(s,x) \in L^2([0,T]\times \T)$ follows from  Lemma \ref{lem: abs continuity}. 
The weak spatial differentiability of $\Phi(\rho(t,x))$ and the $L^2$ integrability of the weak derivative are addressed in Proposition \ref{prop: weak derivative of Phi} and Remark \ref{rmk on deriv L2}. 
When $g$ is of $n^\alpha$ type, 
by Lemma \ref{lem: boundary cond, super} and Lemma \ref{lem: bdry at critical}, we  obtain the boundary conditions $\m_{j}(t) = \left( \lambda_j \Phi(\rho(t,x_j)) \right)^{1/\alpha}$ for all $j\in J_c$ and $\Phi(\rho(t,x_j)) = \Phi(c_0)$ for all $j\in J_s$.
When $g$ is bounded, by Lemma \ref{lem: boundary, bounded}, it holds that 
$\Phi(\rho(t,x_j)) \leq 1/ \lambda_j$ and 
$\m_j(t) = \m_j(t)\id_{\Phi(\rho(t,x_j)) = 1/ \lambda_j}$ for all $j\in J_c$.
Therefore, $\pi_\cdot$ is a weak solution to \eqref{pde: k alpha} or \eqref{pde: bounded g} when $g$ is of $n^\alpha$ type or bounded respectively (cf.~Definitions \ref{def: weak sln} and \ref{def: weak sln, bounded g}).

In Section \ref{section: uniqueness}, we show that there is at most one weak solution
$\pi_\cdot$ to \eqref{pde: k alpha} or  \eqref{pde: bounded g}.
  We conclude then that the sequence of $Q^N$ converges weakly to the Dirac
measure on $\pi_\cdot$.  Finally, as $Q^N$ converges to $Q$ with
respect to the uniform topology, we have weak convergence  of $\pi^N_t$ for each $0\leq t\leq T$.
That is, for all $G$ in $C_c (\T\setminus \D_s)$,
$\langle G,\pi_t^N\rangle$ weakly converges to the constant
$\langle G,\pi_t\rangle$, and therefore convergence in probability as
stated in \eqref{HDL: alpha g} and \eqref{HDL: bounded g}.
To finish the proof, it remains to extend the test function from $G\in C_c (\T\setminus \D_s)$ to $G\in C  (\T )$. When $g$ is bounded, it holds that $C_c (\T\setminus \D_s) = C  (\T )$ as $\D_s=\emptyset$. 
When $g$ is of $n^\alpha$ type,
notice that $\rho(t,\cdot)$ is in $L^1(\T)$  for all $t\geq0$ (cf.~Lemma \ref{lem: abs continuity}). Then, the desired extension follows from standard approximations and \eqref{nbhd expected nmbr} of Lemma \ref{lem: total nmbr bound}.

\section{Tightness}
\label{sec: tightness}
In this section, we address the tightness of the probability measures associated with the trajectories of empirical measures $\pi_\cdot^N$. Recall that $\D_s$ is the macroscopic locations of the super-slow sites; in particular, $\D_s=\emptyset$ when $g$ is bounded. 
To unify treatmentss, in both rate $g$ settings, we let $\Mb$ be the space of locally finite nonnegative measures on $ \T \setminus \D_s$.
Let $C_c( \T \setminus \D_s)$ be the space of continuous functions with compact support on $ \T \setminus \D_s$. Let $\{f_k\}_{k\in \N}$ be a countable dense set in $C_c( \T \setminus \D_s)$ in the sense that for all $f\in C_c( \T \setminus \D_s)$ there exists a subsequence $\{f_{n_k}\}$ such that ${\rm supp\,} f_{n_k}\subset {\rm supp\,}f$ for all $k$ and $f_{n_k}$ converges uniformly to $f$. 
Equipped with the distance 
\[
d(\mu,\nu)
=
\sum_{k=1}^{\infty}
2^{-k}
\frac{\left|  \int_0^1 f_k (d\mu - d\nu)  \right|}{1+  \left| \int_0^1 f_k (d\mu - d\nu)  \right|}.
\]
the space $(\Mb, d(\cdot,\cdot))$ is a complete separable metric space.
The metric $d(\cdot,\cdot)$ realizes the vague topology, 
that is, $\lim_{n\to\infty}d(\mu_n,\mu)=0$ if and only if $\lim_{n\to \infty} \int f d\mu_n =  \int f d\mu$ for all $f\in C_c( \T \setminus \D_s)$.

Let $Q^N$ be the
probability measure on the trajectory space $D([0,T],\Mb )$ governing
$\pi_{\cdot}^N$ when the process starts from $\mu^N$.
We show that $\left\{Q^N\right\}$ is tight with respect to the
uniform topology, stronger than the Skorokhod topology on
$D([0,T], \Mb)$. 

\begin{Lem}
\label {lem: tightness}
$\left\{Q^N\right\}_{N\in \N}$ is relatively compact with respect to the uniform topology.
As a consequence, all limit points $Q$ are supported on trajectories $\{\pi_t\}_{t\in [0,T]}$ vaguely continuous on $\T\setminus \D_s$, that is, for all $G\in C_c(\T \setminus \D_s)$, the mapping $t\in [0,T] \mapsto \langle G, \pi_t\rangle$ is continuous.
\end{Lem}
\begin{proof}
To deduce that $\{Q^N\}$ is relatively compact with respect to uniform
topology, we show the following items (cf. Theorem 15.5 in
\cite{Billingsley}).
\begin{enumerate} 
\item For each $t\in[0,T]$, $\epp>0$, there exists a compact set $K_{t,\epp}\subset \Mb$ such that
\begin{equation}\label {eqn1 compactness}
\sup_{N} Q^{N} \left[\pi^N_\cdot: \pi^N_t\notin K_{t,\epp}\right ] \leq \epp.
\end{equation}
\item For every $\epp>0$, 
\begin{equation}\label {eqn2 compactness}
\limsup_{r \to 0} 
\limsup_{N\to \infty}
Q^N \Big [  \pi^N_\cdot : \sup_{|t-s|<r} 
d(\pi^N_t,\pi^N_s ) >\epp \Big]
=0.
\end{equation}
 \end{enumerate}

\vspace{0.1cm}
{\it Step 1.}
We first consider \eqref{eqn1 compactness}.
Notice that, for any $A>0$, the set $\left\{ \mu\in \Mb: \langle 1,\mu\rangle \leq A  \right\}$ is compact in $\Mb$.
Also, by \eqref{expected_number}, we have $\E_N \Big [N^{-1} \sum_{k\notin \D_{s,N}} \eta_t(k) \Big ]\leq C$ for some constant $C<\infty$ independent of $N$. As
$Q^N\big[\langle 1, \pi^N_t\rangle > A\big]
\leq
A^{-1} \E_N \big [ N^{-1} \sum_{k\notin \D_{s,N}} \eta_t(k) \big ]$, the first condition \eqref{eqn1 compactness} is checked by taking $A$ large. 

\vspace{0.1cm}
{\it Step 2.}
To verify the second condition \eqref {eqn2 compactness}, it is enough
to show a counterpart of the condition for the distributions of
$\langle G, \pi^N_{\cdot}\rangle$ where $G$ is any smooth test
function compactly supported on $ \T \setminus \D_s$ (cf.\,p.\,54, \cite{KL}).  In other words, we need to
show, for every $\epp>0$,
\begin{equation}
\label {eqn2 compactness with G}
\limsup_{r \to 0} 
\limsup_{N\to \infty}
Q^N \Big [  \pi^N_\cdot : \sup_{|t-s|<r} \Big| \langle G,  \pi^N_t \rangle  -\langle G, \pi^N_s \rangle\Big| >\epp \Big]
=0.
\end{equation}
Note that
$\left\langle G,  \pi^N_t \right\rangle
=
\left\langle G,  \pi^N_0 \right\rangle
+
\int_0^t N^2L_N \left\langle G,   \pi^N_s \right\rangle ds
+
M^{N,G}_t$.
Then, we only need to consider the oscillations of $\int_0^t N^2L_N \left\langle G,   \pi^N_s \right\rangle ds$
and $M^{N,G}_t$ respectively. 

\vspace{0.1cm}
{\it Step 3.}
For the oscillations of the martingale $M^{N,G}_t$, 
by $\big| M^{N,G}_t - M^{N,G}_s  \big| \leq  \big| M^{N,G}_t \big| + \big| M^{N,G}_s  \big|$, we have
$
\P_N
\big[
\sup_{|t-s|<r} 
\big| 
M^{N,G}_t - M^{N,G}_s  
\big|
>\epp
\big]
\leq
2\P_N
\big[
\sup_{0\leq t \leq T} 
\big| 
M^{N,G}_t  
\big|
>\epp /2
\big]
$.
Using Chebyshev and Doob's inequality, we further bound it by
\[
\dfrac{8}{\epp^2}
\E_N
\Big[
\Big(
\sup_{0\leq t \leq T} 
\big| 
M^{N,G}_t  
\big|
\Big)^2
\Big]
\leq
\dfrac{32}{\epp^2}
\E_N
\Big[
\big(
M^{N,G}_T
\big)^2
\Big]
= 
\dfrac{32}{\epp^2}\E_N
\langle
M^{N,G}
\rangle_T .
\]
By Lemma \ref{lem: martingale bounds}, $\E_N \langle M^{N,G}\rangle_T = O(N^{-1})$.
Then, we conclude
\[    
\lim_{r \to 0}
\lim_{N\to \infty}
\P_N
\Big[
\sup_{|t-s|<r} 
\left| 
M^{N,G}_t - M^{N,G}_s  
\right|
>\epp
\Big]
=0.
\]

\vspace{0.1cm}
{\it Step 4.}
To conclude \eqref{eqn2 compactness with G}, it suffices to show
\[
\limsup_{r \to 0} 
\limsup_{N\to \infty}
 Q^N \Big [ 
\sup_{|t-s|<r} 
 \int_s^t 
  \big|
 N^2L_N \left\langle G,  \pi^N_{\tau} \right\rangle 
 \big|
 d\tau
>\epp \Big]
=0.
\]
The absolute value term is bounded above by
$ \|\Delta G\|_\infty \dfrac1N \sum_{k\notin \D_{s,N}}  g_{k,N}(\eta_\tau(k))$ (cf.~\eqref{gen_comp} applied to $\pi^N_\cdot$).
Then, by Markov's inequality, it remains to show
\[
\limsup_{r \to 0} 
\limsup_{N\to \infty}
\E_N \Big [ 
\sup_{|t-s|<r} 
 \int_s^t \dfrac1N \sum_{k \in \T_N \setminus \D_{s,N}} 
  g_{k,N}(\eta_\tau(k)) d\tau  \Big]
=0.
\]
By Lemma \ref{lem: total nmbr bound}, 
$\E_N \big[ \int_0^T N^{-1} g_{k,N} (\eta_t(k)) dt \big]$
vanishes as $N\to\infty$ for defect sites $k\in \D_{c,N}$ or $\D_{b,N}$.
Therefore, we may restrict the summation term in the previous display over $k\in \T_N \setminus \D_N$.
Note that for each such $k$ we have $g_{k,N} (\cdot)= g(\cdot)$.
When $g$ is bounded, as
$N^{-1} \sum_{k \notin \D_N} 
  g(\eta_\tau(k)) \leq \|g\|_\infty$, the lemma is proved for this case.
  
The rest of this proof focuses on the case when $g$ is of $n^\alpha$ type. Since $g(\cdot)$ grows at most linearly,  we are left to show
\begin{equation} 
\label{oscillate, parti numb}
\lim_{r\to 0} 
\lim_{N\to \infty}
\E_N \Big [ 
\sup_{|t-s|<r} 
\Big |
 \int_s^t \dfrac1N \sum_{k\in \T_N\setminus \D_N}  
 \eta_\tau(k) d\tau \Big | \Big]
 =0.
 \end{equation}

\vspace{0.1cm}
{\it Step 5.}
To show \eqref{oscillate, parti numb}. we introduce a truncation $\id_{\eta(k)\leq A}$ with $A>0$.
Notice that
\[
\E_N \Big [ 
\sup_{|t-s|<r} 
\Big |
 \int_s^t \dfrac1N \sum_{k\notin \D_N}    
\eta_\tau(k) \id_{\eta_\tau(k)\leq A}
 d\tau \Big | \Big]
 \leq
 \E_N \Big [ 
\sup_{|t-s|<r} 
\Big |
 \int_s^t A
 d\tau \Big | \Big]
 \leq
 r A
\]
which vanishes when taking 
$N\to\infty$, $r\to 0$, and $A\to \infty$ in order.
It remains to show the error term with $\id_{\eta(k) >A}$ also vanishes in the limit.
Note that the error term is estimated by
\[
\begin{split}
\E_N \Big [ 
\sup_{|t-s|<r} 
\Big |
 \int_s^t \dfrac1N \sum_{k\notin \D_N}    
 \eta_\tau(k) \id_{\eta_\tau(k)> A}
 d\tau \Big | \Big]
 \leq\ &
 \E_N \Big [ 
  \int_0^T
 \dfrac1N \sum_{k\notin \D_N}    
 \eta_\tau(k) \id_{\eta_\tau(k)> A}
  d\tau
 \Big] \\
 =&
  \int_0^T
 \E_N \Big [ 
 \dfrac1N \sum_{k\notin \D_N}   
 \eta_\tau(k) \id_{\eta_\tau(k)> A}
 \Big]  d\tau.
 \end{split}
\]
Recall the initial entropy bound $H(\mu^N | \Pam_{c_0}^N) < CN$. By entropy inequality, for any $\tau\in[0,T]$ and $B>0$, the term $\E_N \Big [ 
 N^{-1} \sum    
 \eta_\tau(k) \id_{\eta_\tau(k)> A}
 \Big]$ is further bounded above by
\begin{equation} \label {tightness: t-s, >A}
\dfrac{C}{B} + 
\dfrac{1}{BN} \ln E_{\Pam_{c_0}^N}
\Big[
\exp\Big\{B \sum_{k\notin \D_N}    
 \eta(k) \id_{\eta(k)> A} \Big\}
 \Big].
\end{equation}
 Since, with respect to $\Pam_{c_0}^N$, the $\eta(k)$'s in \eqref{tightness: t-s, >A} are i.i.d.\,with common distribution $\Pb_{\varphi}$ for $\varphi = \Phi(c_0)$,
we have \eqref{tightness: t-s, >A} is equal to
\[
\dfrac{C}{B} + 
\dfrac{N-n_0}{BN} \ln 
E_{\Pb_{\varphi}} 
\left[
e^{BX \id_{X> A} }
 \right]
 \leq 
 \dfrac{C}{B} + 
B^{-1} \ln 
E_{\Pb_{\varphi}} 
\left[
e^{BX \id_{X> A} }
 \right].
\]
Here,  $n_0$ is the number of defect sites and the distribution of $X$ is $\Pb_{\varphi}$.  As $E_{\Pb_{\varphi}}[e^{B X }] <\infty$ for all $B>0$,
we have $B^{-1} \ln E_{\Pb_{\varphi}} 
\big [ e^{BX \id_{X> A} } \big] \to 0$ when taking 
$A\to \infty$ and then $B\to \infty$.
Hence, we obtain
\[
\lim_{A\to\infty} \lim_{r\to 0} \lim_{N\to \infty}
\E_N \Big [ 
\sup_{|t-s|<r} 
\Big |
 \int_s^t \dfrac1N \sum_{k\notin \D_N}  
 \eta_\tau(k) \id_{\eta_\tau(k)> A}
 d\tau \Big | \Big]
 =0
\]
which completes the proof.
\end{proof}

\section{Properties of limit measures}
\label{sec: properties of limit measures}

By Lemma \ref{lem: tightness}, the sequence $\left\{ Q^N\right\}$ is relatively compact  with respect to the uniform topology.
 Consider any convergent subsequence of $Q^N$ and relabel so that $Q^N \Rightarrow Q$.  
We now consider absolute continuity and  boundary behaviors at defect sites for trajectories under $Q$.

\subsection{Absolute continuity}
We now address absolute continuity.

\begin{Lem}
\label{lem: abs continuity}
 We have
$Q$ is supported on trajectories whose restriction on $\T\setminus  \D_c$ is absolutely continuous and the density is in $L^\infty$ uniformly in $t$: there exists $c\geq0$ such that
\begin{equation*} 
Q \Big[ \pi_\cdot :
\pi_t(dx) = \rho(t,x)dx 
+ \sum_{j\in J_c} \m_j(t) \delta_{x_j}(dx)
 \text{ with }
\|\rho(t,\cdot)\|_\infty<c
 \text{ for all } 0\leq t\leq T
\Big]
=
1.
\end{equation*}
\end{Lem}

\begin{proof}
Let $D:= \T \setminus \{\D_s\cup \D_c\}$
and $C_c^+(D)$ be the space of nonnegative continuous functions with compact support on $D$,
equipped with topology of uniform convergence on compact sets. 
Take $\left\{G_n\right\}_{n\in \N}$ be a dense sequence of $C_c^+(D)$.
The lemma will follow if we have, for some $c\geq 0$, 
\[
Q \Big[
\pi_\cdot:
\langle G_n, \pi_t \rangle
\leq 
\langle G_n, c \rangle
 \text{ for all } 0\leq t\leq T \text{ and } n\in \N
\Big] = 1.
\]
Here $\langle G, c\rangle$ is shorthand for $c\int_\T G(x)dx$.

To this end, recall $\kappa_{c'}^N$ defined in Condition \ref{condition on mu N}.  Let $\nu^N$ denote $\Pam_{c'}^N$ when $g$ is of $n^\alpha$ type and $\kappa_{c'}^N$ when $g$ is bounded. By the product structure of $\nu^N$ and Chebyshev inequality, for each $\delta>0$,  we have $\nu^N[ |\langle G_n,\pi^N\rangle - \langle G_n,c' \rangle| > \delta] \to 0$ as $N\to \infty$.

Fix $\e>0$ and $t\in [0,T]$. By attractiveness (cf.~Section \ref{sec: coupling}), 
$Q^N \big[
\langle G_n, \pi^N_t \rangle
\leq 
\langle G_n, c' \rangle
 + \e \big]$ is bounded from below by
$\nu^N \big[
\langle G_n, \pi^N \rangle
\leq 
\langle G_n, c' \rangle
 + \e \big].
 $
Then, we have, for all $t\geq0$ and $\e>0$
\[
\lim_{N\to \infty} Q^N \left[
\langle G_n, \pi_t^N \rangle
\leq 
\langle G_n, c' \rangle
 + \e
\right]
=
1.
\]
As compactness of $\{Q^N\}$ was shown in the uniform topology in Lemma \ref{lem: tightness}, the distribution of $\langle G_n,\pi^N_t\rangle$ under $Q^N$ converges weakly to $\langle G_n,\pi_t\rangle$ under $Q$. 
Hence, for any $\e>0$ and $t\in [0,T]$, we have 
\[
Q \Big[
\langle G_n, \pi_t \rangle
\leq 
\langle G_n, c' \rangle
 + \e
\Big]
\geq
\limsup_{N\to \infty}
Q^N \Big[
\langle G_n, \pi_t^N \rangle
\leq 
\langle G_n, c' \rangle
 + \e
\Big]=1.
\]
Since $Q$ is supported on vaguely continuous trajectories by Lemma \ref{lem: tightness}, we obtain for all $\e>0$ that
$
Q \Big[
\langle G_n, \pi_{t} \rangle
\leq 
\langle G_n, c' \rangle + \e
 \text{ for all } 0\leq t\leq T,\, n\in \N 
\Big]
=
1.$
Then, we conclude the lemma by taking $\e\to 0$ and choosing $c=c'$.
\end{proof}

\subsection{Boundary behavior}
\label{sec: boundary behavior}
We proceed by showing the behavior of the limit density $\rho(t,x)$ around the defects $x_j$ in $\D_s$ or $\D_c$.
Notice that, part of the results presented here are in terms of $\rho(t,x)$ as $x\to x_j+$. 
The corresponding results for $x\to x_j-$ follow from similar arguments.  In the following, we will make use of a `local' replacement lemma Lemma \ref{local replace, slow site} proved in Section \ref{sec: rplmt lem at bdry}.

For each $\eps\in (0,1)$, 
let $\hat \iota_\eps: (0,\eps)\mapsto [0,1]$
be a compactly supported continuous function such that $\hat \iota_\eps(x)=1$ for $x\in $
$(\eps^2, \eps-\eps^2)$.
Let $\| \hat\iota_\eps\|$ denote the $L^1$ norm $ \int_0^\eps \hat \iota_\eps(x) dx$ and $ \|\hat\iota_\eps\|^{-1}:= ( \|\hat\iota_\eps\|)^{-1}$. 
Notice that $\big| \|\hat\iota_\eps\|^{-1} - \eps^{-1} \big|$ is bounded from above for all $\eps$ small since 
$\big| \|\hat\iota_\eps\|^{-1} -\eps^{-1} \big| $
is less than $ \big| (\eps - 2\eps^2)^{-1} -\eps^{-1}\big  | $ which approaches $2$ as $\eps\to 0$.

We first describe behavior near `super-slow' sites in the $n^\alpha$ setting.
\begin{Lem} \label{lem: boundary cond, super}
Let $g$ be of $n^\alpha$ type.
Then, for any $j\in J_s$, $\delta>0$, and $G\in C[0,T]$,
we have
\begin{equation} \label {eqn: D_s bdry, with theta}
\lim_{\eps\to 0} Q
\Big[
\Big| \int_0^T G(s) 
\big( \Phi\big( \rho^{+,\theta}_{x_j}(s)\big)
-
\Phi(c_0)  \big)
  ds \Big| > \delta
\Big]
=0
\end{equation}
where $\rho^{+,\theta}_{x_j}(s) 
:=\|\hat \iota_\eps\| ^{-1}
\langle \hat\iota_\eps(\cdot-x_j) , \rho \rangle
= \|\hat\iota_\eps\|^{-1} \int_\T \hat\iota_\eps (x-x_j) \rho(s,x) dx$.
Consequently, 
\begin{equation} \label {eqn: D_s bdry}
Q\big[ 
\Phi(\rho(t,x_j) ) = \Phi(c_0) \text{ for a.e.~$t\in[0,T]$ and all $j \in J_s$}
\big] = 1.
\end{equation}  
\end{Lem}

\begin{proof}
 The proof of is split into steps. We show \eqref{eqn: D_s bdry, with theta} in the first four steps and the last step derives \eqref{eqn: D_s bdry} from \eqref{eqn: D_s bdry, with theta}. 

\vspace{0.1cm}
{\it Step 1.} 
Fix $j\in J_s$ and also a test function $G\in C[0,T]$ for later use.  We first show that
\begin{equation} \label {eqn: g eta simple, E}
\lim_{N\to \infty}
\sup_{s\geq0}
\E_N \big[
 \big| 
 N^{-\beta_j}g(\eta_s(k_{j,N})) 
-
 N^{-\beta_j} (\eta_s(k_{j,N}) )^\alpha 
  \big| \big]
  =0.
\end{equation}
Note that $g(n)\sim n^\alpha$.
For each $\epp>0$, let $A=A(\epp)$ be such that
$| n^\alpha/ g(n) -1 | < \epp$ for all $n> A$. 
Then, the expectation term in \eqref{eqn: g eta simple, E} is bounded above by
$E_1 + E_2$ where $E_1$ and $E_2$ are the same expectation as in \eqref{eqn: g eta simple, E} with the integrand multiplied by indicators 
$\id_{\eta_s(k_{j,N}) \leq A}$ and
$\id_{\eta_s(k_{j,N}) > A}$ respectively.
For each $A$, the term $E_1$ vanishes as $N\to \infty$.
By the definition of $A$ and Lemma \ref{lem: total nmbr bound}, the term $E_2$ is further bounded above by 
 $ \epp \sup_{s\geq0}  \E_N \big[ N^{-\beta_j} g(\eta_s(k_{j,N}) )  \big] = O(\epp)$.
Letting $\epp \to 0$, we conclude \eqref{eqn: g eta simple, E}.

\vspace{0.1cm}
{\it Step 2.} 
We now argue that $N^{-\beta_j} (\eta_s(k_{j,N}) )^\alpha$ may be replaced by $N^{-\beta_j} (\eta_0(k_{j,N}) )^\alpha$. 
Let $r_{N,s} = N^{-\beta_j/\alpha} ( \eta_s(k_{j,N}) -
\eta_0(k_{j,N}))$.
Let $\tau>0$ be a constant such that there are no other defects within a $\tau$-neighborhood of $x_j$.
Take a test function $F: \T\mapsto [0,1]$ such that $F$ has compact support in $(x_j - \tau, x_j+\tau)$ and $F(x)=1$ for $|x-x_j|\leq \tau/2$.
Recall $\wpi_t^N$ from \eqref{eqn: pi_t hat} and notice that
\[
| r_{N,s} |
\leq
N^{1-\beta_j/\alpha}  
\big| \langle F, \wpi_s^N\rangle 
-  \langle F, \wpi_0^N\rangle 
\big|
+
N^{-\beta_j/\alpha} 
{\sum_k}^\circ
 (\eta_s(k) + \eta_0(k))
 := I_1 + I_2 .
\]
where the $\sum_k^\circ$ term is summation over $k\in \T_N$ such that $|k/N - x_j| \leq \tau$ and $k\neq k_{j,N}$.

As $\beta_j >\alpha$, \eqref{expected_number} of Lemma  \ref{lem: total nmbr bound} asserts that
 $\sup_{s\geq 0}\E_N [I_2]$ vanishes as $N\to \infty$.
To show $\E_N[I_1]$ vanishes, notice that
$\left\langle F,  \wpi^N_s \right\rangle
-
\left\langle F,  \wpi^N_0 \right\rangle
=
\int_0^s N^2L_N \left\langle F,   \wpi^N_t \right\rangle dt
+
M^{N,F}_s$.
Using the generator computation 
\eqref{gen_comp} and the jump rate bound \eqref{expected_g_all}, we have 
\[
\sup_{N\in \N}\sup_{t\geq0}
\E_N \big[
N^2L_N \left\langle F,   \wpi^N_t \right\rangle \big]
\leq 
\sup_{N\in \N}\sup_{t\geq0}
\|\Delta_N F(\cdot)\|_\infty
\E_N \Big [ \dfrac1N \sum_{k\in \T_N} g_{k,N}(\eta_t(k)) \Big ] <\infty.
\]
Also, by Lemma \ref{lem: martingale bounds}, the martingale term vanishes as 
its variance $ \sup_{0\leq s\leq T}\E_N \langle M^{N,G}\rangle_s = o(1)$.
Thus, as $\beta_j>\alpha$, we have $\E_N[I_1]$ and, therefore, $\E_N [|r_{N,s}|]$
 vanishes as $N\to \infty$ uniformly for all $s\in [0,T]$.
 
As $\alpha \in (0,1]$, by the elementary inequality $(|x|+|y|)^\alpha \leq |x|^\alpha + |y|^\alpha$, we have 
\[
N^{-\beta_j} |(\eta_s(k_{j,N}))^\alpha - (\eta_0(k_{j,N}))^\alpha |
\leq |r_{N,s}|^\alpha.
\]
Therefore, we conclude
\begin{equation} \label {eqn: eta change, super}
\lim_{N\to \infty}
\sup_{0\leq s\leq T}
\E_N \Big[
 N^{-\beta_j}
\big|
(\eta_s(k_{j,N}) )^\alpha 
-
(\eta_0(k_{j,N}) )^\alpha
\big|
 \Big]
  =0.
\end{equation}

In particular,
it follows from \eqref{eqn: g eta simple, E}
and \eqref{eqn: eta change, super} that,
for any $\delta>0$,
\begin{equation} \label{bdry super, before rplmt}
\lim_{N\to \infty}
\P_N \Big[ \Big| 
\int_0^T G(s) 
\big(
 N^{-\beta_j}g(\eta_s(k_{j,N})) 
-
 N^{-\beta_j} (\eta_0(k_{j,N}) )^\alpha \big)
  ds\Big| >\delta \Big]
  =0.
\end{equation}

\vspace{0.1cm}
{\it Step 3.} 
In this step, we show that $ N^{-\beta_j} (\eta_0(k_{j,N}) )^\alpha$ may be replaced by 
$\varphi = \lambda_j \Phi(c_0)$.
Precisely, as $\eta_0$ has distribution $\mu^N$, we show that
\begin{equation} 
\label{eqn: initial slow site esti}
\lim_{N\to \infty}
\mu^N \big [  \big| N^{-\beta_j/\alpha}  \xi(k_{j,N}) -   \varphi^{1/\alpha} \big| > \delta \big] 
= 0.
\end{equation}
Indeed, let 
$ \Ub_N
:=
\big\{
\big| N^{-\beta_j/\alpha}  \xi(k_{j,N}) -  
 \varphi^{1/\alpha} \big| \geq \delta
\big\}$.
Fix $r_0\in (1,\beta_j/\alpha)$.
By the entropy inequality,
\[
\mu^N \left[\Ub_N\right]
=
E_{\mu^N} \left[ \id_{\Ub_N} \right]
\leq
N^{-r_0} H(\mu^N | \Pam_{c_0}^N) 
+
N^{-r_0}\ln E_{\Pam_{c_0}^N}
\big[
e^{N^{r_0} \id_{\Ub_N}}
\big].
\]
Since $H(\mu^N | \Pam_{c_0}^N)=O(N)$  (cf.~Condition \ref{condition on mu N}) and $E_{\Pam_{c_0}^N}
\big[
e^{N^{r_0} \id_{\Ub_N}}
\big] \leq 1 + 
e^{N^{r_0}} \Pam_{c_0}^N \left[\Ub_N\right]$, 
we have
\[
\lim_{N\to \infty} \mu^N \left[\Ub_N \right]
\leq
\max\big\{0,
 \lim_{N\to \infty}
N^{-r_0} \ln \big( e^{N^{r_0}} \Pam_{c_0}^N \left[\Ub_N\right] \big)  \big\}.
\]
By the following Lemma \ref{lem: LD super}, we have $\lim_{N\to \infty} N^{-r_0} \ln \Pam_{c_0}^N \left[\Ub_N\right] =-\infty$,
and therefore, \eqref{eqn: initial slow site esti} holds.

\vspace{0.1cm}
{\it Step 4.}
Putting together \eqref{bdry super, before rplmt} and \eqref{eqn: initial slow site esti}, we have 
\[
\lim_{N\to \infty}
\P_N \Big[ \Big| 
\int_0^T G(s) 
\big(
 N^{-\beta_j}g(\eta_s(k_{j,N})) 
-
 \lambda_j \Phi(c_9) \big)
  ds\Big| >\delta \Big]
  =0.
\]
For each $\eps>0$, let $\eta_s^{\eps N,+}(k)  = (\eps N)^{-1} \sum_{ 0< j - k \leq \eps N} \eta_s(k)$.
By Lemma \ref{local replace, slow site}, we may further replace $N^{-\beta_j}g(\eta_s(k_{j,N}))$ by $\lambda_j \Phi(\eta_s^{\eps N,+}(k_{j,N}))$ and obtain
\begin{equation} \label {bdry, indicator test fcn}
\lim_{\eps \to 0}
\limsup_{N\to \infty}
\P_N \Big[ \Big| 
\int_0^T G(s) 
\big(
 \Phi(\eta_s^{\eps N,+}(k_{j,N})) 
-
 \Phi(c_0) \big)
  ds\Big| >\delta \Big]
  =0.
\end{equation}
Notice that
$\big |\eta_s^{\eps N,+}(k_{j,N}) 
 - 
 \|\hat \iota_\eps\| ^{-1}
\langle \hat\iota_\eps(\cdot-x_j) , \pi_s^N \rangle \big|$
is bounded above by
\[
\big| \|\hat\iota_\eps\|^{-1} -\eps^{-1} \big|
 N^{-1} \sum_{0<k/N -x_j\leq \eps} 
 \eta_s(k)
 +
  (\eps N)^{-1} \sum_{(k/N -x_j) \in (0,\eps^2) \cup (\eps - \eps^2, \eps]} 
 \eta_s(k).
\]
Using the particle numbers bound \eqref{nbhd expected nmbr} and that $\big| \|\hat\iota_\eps\|^{-1} -\eps^{-1} \big|$ is bounded for all $\eps$ small, the $\E_N$ expectation of the previous display vanishes as $N\to\infty$.
Since $\Phi(\cdot)$ is Lipschitz, we may replace $\eta_s^{\eps N,+}(k_{j,N})$ in \eqref{bdry, indicator test fcn} by $\|\hat \iota_\eps\| ^{-1}
\langle \hat\iota_\eps(\cdot-x_j) , \pi_s^N \rangle$ to obtain
\[
\lim_{\eps\to0}
\limsup_{N\to \infty}
Q^N \Big[
\Big| \int_0^T G(s) 
\big( \Phi\big( 
\|\hat \iota_\eps\| ^{-1}
\langle \hat\iota_\eps(\cdot-x_j) , \pi_s^N \rangle
\big)
-
\Phi(c_0)  \big)
  ds \Big| > \delta
\Big] = 0.
\]
Since the absolute value term in the previous display is continuous (in the uniform topology) when viewed as a function on the trajectory space $D([0,T],\Mb )$, we obtain 
\[
\lim_{\eps\to0}
Q \Big[
\Big| \int_0^T G(s) 
\big( \Phi\big( 
\|\hat \iota_\eps\| ^{-1}
\langle \hat\iota_\eps(\cdot-x_j) , \pi_s \rangle
\big)
-
\Phi(c_0)  \big)
  ds \Big| > \delta
\Big] = 0.
\]
By Lemma \ref{lem: abs continuity}, $\|\hat \iota_\eps\| ^{-1}
\langle \hat\iota_\eps(\cdot-x_j) , \pi_s \rangle$ is recognized as $\rho^{+,\theta}_{x_j}(s)$; therefore \eqref{eqn: D_s bdry, with theta} follows.

\vspace{0.1cm}
{\it Step 5.} 
We now conclude the proof of the lemma by showing \eqref{eqn: D_s bdry} from \eqref {eqn: D_s bdry, with theta}.
By Lemma \ref{lem: abs continuity} and Proposition \ref {prop: weak derivative of Phi} (see also Remarks \ref{rmk on w-deriv} and \ref{rmk on deriv L2}), it holds almost surely under $Q$ that $\|\rho(\cdot,\cdot)\|_\infty<\infty$ and $\rho(t,\cdot)\in C(\T)$ for almost all $t\in [0,T]$. Then, we have, $Q$-almost surely,  $\int_0^T G(s) 
 \Phi\big( \rho^{+,\theta}_{x_j}(s)\big) ds$ converges to 
 $\int_0^T G(s) \Phi (\rho(s,x_j)) ds$.
As we have shown in \eqref {eqn: D_s bdry, with theta} that $\int_0^T G(s) 
 \Phi\big( \rho^{+,\theta}_{x_j}(s)\big) ds$ also converges to 
 $\int_0^T G(s) \Phi (c_0) ds$ in probability, we obtain,
 for all $G\in C[0,T]$ and $j\in J_s$
\[
Q
\Big[
 \int_0^T G(s) 
\big( \Phi\big( \rho(s,x_j)\big)
-
\Phi(c_0)  \big)
  ds  =0
\Big]
=1.
\]
As the continuous function space $C[0,T]$ is separable with respect to the unform norm and $J_s$ is a finite set, we further obtain
\[
Q
\Big[
 \int_0^T G(s) 
\big( \Phi\big( \rho(s,x_j)\big)
-
\Phi(c_0)  \big)
  ds  =0 \text{ for all $G\in C[0,T]$ and $j\in J_s$}
\Big]
=1
\]
which implies \eqref{eqn: D_s bdry}, finishing the proof. 
\end{proof}

We now show the estimate used in Step 3 of the above argument.

\begin{Lem} \label{lem: LD super}
Let $g$ be of $n^\alpha$ type. Let $X_N$ be distributed according to $\Pb_{N^\beta \varphi}$
for some  $\beta>\alpha$ and $\varphi\geq0$. 
Then, for any $r_0 < \beta/\alpha$
and $\delta>0$
\[
\lim_{N\to \infty}
N^{-r_0}
 \ln P\big[\big| N^{-\beta/\alpha} X_N -  \varphi^{1/\alpha} \big| \geq \delta \big]
= -\infty.
\]
\end{Lem}
\begin{proof}
Let $\phi = \varphi^{1/\alpha}$.
To prove the lemma, it suffices to show that
\[
\lim_{N\to \infty}
N^{-r_0}
 \ln P\big[X_N \geq ( \phi+\delta)N^{\beta/\alpha} \big]
 =
 \lim_{N\to \infty}
N^{-r_0}
 \ln P\big[X_N \leq ( \phi-\delta)N^{\beta/\alpha} \big]
=
-\infty.
\]
To this end, we notice that, for any $x>0$, 
\[
\begin{split}
\ln P\big[X_N \geq ( \phi+\delta)N^{\beta/\alpha} \big]
\leq&
\ln\big(
e^{-( \phi+\delta)N^{\beta/\alpha}x}
E[e^{xX_N} ] \big)\\
=&
- ( \phi+\delta)N^{\beta/\alpha} x 
+
\ln \dfrac{Z(e^x N^\beta \varphi) }
{ Z(N^\beta \varphi)}.
\end{split}
\]
By Lemma \ref{lem: order of R},
$\ln \dfrac{Z(e^x N^\beta \varphi) }
{ Z(N^\beta \varphi)}
\sim
\alpha \phi N^{\beta/\alpha} 
(e^{x/\alpha} - 1)$.
Then, for 
\[
K_{\delta,\phi,\alpha} 
:= 
\max_{x>0} 
\big\{ 
 ( \phi +\delta ) x 
 -
  \alpha \phi ( e^{x/\alpha} - 1) \big\}
\]
we have
\[
\limsup_{N\to \infty}
N^{-\beta/\alpha}
\ln P\big[X \geq ( \phi+\delta)N^{\beta/\alpha} \big]
\leq
-K_{\delta,\phi,\alpha} <0.
\]
Since $r_0 < \beta/\alpha$, we obtain
\begin{equation} \label {eqn: prob,> delta}
\lim_{N\to \infty}
N^{-r_0}
 \ln P\big[X_N  \geq ( \phi+\delta)N^{\beta/\alpha} \big]
 =
-\infty.
\end{equation}
On the other hand, for all $x>0$, we also have
\[
\ln P\big[X_N  \leq ( \phi-\delta)N^{\beta/\alpha} \big]
\leq
\ln\big(
e^{( \phi - \delta)N^{\beta/\alpha} x }
E[e^{-xX_N} ] \big).
\]
By a similar argument employed to prove \eqref{eqn: prob,> delta}, we have 
\[
 \lim_{N\to \infty}
N^{-r_0}
 \ln P\big[X_N  \leq ( \phi-\delta)N^{\beta/\alpha} \big]
=
-\infty.
\]
The lemma is now proved.
\end{proof}

We now consider the behavior near `critical' slow sites in the $n^\alpha$ setting used in Step 5.

\begin{Lem} \label{lem: bdry at critical}
Let $g$ be of $n^\alpha$ type.
Then, for any $j\in J_c$, $\delta>0$, and $G\in C[0,T]$, we have
\begin{equation}  \label {eqn: D_c bdry with theta}
\lim_{\eps\to 0} Q
\Big[
\Big| \int_0^T G(s) 
\Big(
 \lambda_j  \Phi\big( \rho^{+,\theta}_{x_j}(s)  \big) 
 -
 (\m_j(s) )^\alpha
  \Big)
  ds \Big| > \delta
\Big]
=0,
\end{equation}
where $\rho^{+,\theta}_{x_j}(s) 
=\|\hat \iota_\eps\| ^{-1}
\langle \hat\iota_\eps(\cdot-x_j) , \rho \rangle$.
Consequently, 
\begin{equation} \label {eqn: D_c bdry}
Q\big[ 
 \m_j(t)  =
 \big( \lambda_j  \Phi\big( \rho(t,x_j)  \big) 
 \big)^{1/\alpha}
\text{ for a.e.~$t\in[0,T]$ and all $j \in J_c$}
\big] = 1.
\end{equation}  
\end{Lem}

\begin{proof}
Fix $j\in J_c$ and $G\in C[0,T]$.
As $g(n) \sim n^\alpha$, arguments as in Step $1$ of the proof of 
Lemma \ref{lem: boundary cond, super}
show that $N^{-\alpha}g(\eta_s(k_{j,N})) $ may be replaced by $N^{-\alpha} (\eta_s(k_{j,N}) )^\alpha$.  Hence, we have
\begin{equation} \label {eqn: g eta simple, critical}
\lim_{N\to \infty}
\P_N \Big[ \Big| 
\int_0^T G(s) 
\Big(
 N^{-\alpha}g(\eta_s(k_{j,N})) 
-
  (N^{-1} \eta_s(k_{j,N}) )^\alpha \Big)
  ds\Big| >\delta \Big]
  =0.
\end{equation}

Let $\{F_\eps(x)\}$ be a sequence of nonnegative smooth functions such that $F_\eps$'s are supported on $(x_j-\eps,x_j+\eps)$, $\|F_\eps\|_\infty \leq 1$, and $F_\eps(x) =1$ for $|x-x_j|\leq \eps/2$.
Then, $$ | \langle F_\eps, \pi^N_s \rangle - N^{-1}\eta_s(k_{j,N}) | \leq B_{\eps,N} (\eta_s)$$ 
where 
$B_{\eps,N} (\eta) 
= N^{-1} \sum_{k\notin \D_N} 
\id_{(-\eps,\eps)} (k/N -x_j)\, \eta(k) $
for $\eps$ small and $N$ large.
As $0< \alpha \leq 1$, we have 
\[
\E_N \big[
\big|
\langle F_\eps, \pi^N_s \rangle ^\alpha
-
(N^{-1}\eta_s(k_{j,N} ))^\alpha \big|
 \big]
 \leq
 \E_N \big[
(B_{\eps,N})^\alpha
 \big]
 \leq
  \E_N \big[
B_{\eps,N} \big] ^\alpha, 
\]
which vanishes as $N\to \infty$ and $\eps\to 0$ by \eqref{nbhd expected nmbr} in Lemma \ref{lem: total nmbr bound}.
Therefore, we may replace 
 $N^{-1} \eta_s(k_{j,N})$ in \eqref{eqn: g eta simple, critical} to obtain
\[
\lim_{\eps\to0}
\limsup_{N\to \infty}
\P_N \Big[ \Big| 
\int_0^T G(s) 
\Big(
 N^{-\alpha}g(\eta_s(k_{j,N})) 
-
  \langle F_\eps, \pi^N_s \rangle ^\alpha \Big)
  ds\Big| >\delta \Big]
  =0.
\]

By Lemma \ref{local replace, slow site}, $N^{-\alpha}g(\eta_s(k_{j,N}))$ may be replaced by $\lambda_j \Phi(\eta_s^{\eps N,+}(k_{j,N}))$. 
Moreover, following Step $4$ in the proof of Lemma \ref{lem: boundary cond, super} to replace $\eta_s^{\eps N,+}(k_{j,N})$ with $ \rho^{+,\theta}_{x_j}(s)$, we obtain
\[
\lim_{\eps\to 0} Q
\Big[
\Big| \int_0^T G(s) 
\Big(
 \lambda_j  \Phi\big( \rho^{+,\theta}_{x_j}(s)  \big)
 -
   \langle F_\eps, \pi_s \rangle^\alpha  \Big)
  ds \Big| > \delta
\Big]
=0.
\]
Moreover, by Lemma \ref{lem: abs continuity}, $\int_0^T G(s) \langle F_\eps, \pi_s \rangle^\alpha  ds$ converges to 
\[
\int_0^T G(s) ( \pi_s(\{x_j\}) )^\alpha ds
=
\int_0^T G(s) (\m_j(s))^\alpha ds
\] 
almost surely with respect to $Q$.  Hence, with these observations,
\eqref  {eqn: D_c bdry with theta} in the lemma holds.
Following the argument in Step $5$ of Lemma \ref{lem: boundary cond, super}, we obtain 
\eqref{eqn: D_c bdry} from 
\eqref{eqn: D_c bdry with theta}, finishing the proof. 
\end{proof}

We now turn to the behavior near `critical' slow sites in the $g$ bounded setting.
\begin{Lem} \label {lem: boundary, bounded}
Let $g$ be bounded.
Then, for any $j\in J_c$, $\delta>0$, and $G \in C[0,T]$, we have
\begin{equation} \label{max_rho, g bded}
\lim_{\eps\to 0} Q
\Big[
 \int_0^T |G(s) |
\big(
 \Phi\big( \rho^{+,\theta}_{x_j}(s)\big)
 -
 \phi_j 
  \big)
  ds  > \delta
\Big]
=0
\end{equation}
and
\begin{equation} \label{m=0_if_rho<max}
\lim_{\eps\to 0} Q
\Big[
\Big|
 \int_0^T G(s) 
( \m_j(s) \wedge 1)
\big( \phi_j - 
 \Phi\big( \rho^{+,\theta}_{x_j}(s)\big)
  \big)
  ds  \Big | > \delta
\Big]
=0.
\end{equation}
where $\rho^{+,\theta}_{x_j}(s) 
:=\|\hat \iota_\eps\| ^{-1}
\langle \hat\iota_\eps(\cdot-x_j) , \rho \rangle$,
$\phi_j = \lambda^{-1}_j$,
and
$a \wedge b = \min \{ a,b \}$. 
 Consequently, 
\begin{align} \label {eqn: D_c bdry, bounded g}
&Q\big[ 
\Phi(\rho(t,x_j)) \leq \lambda_j^{-1},\,
 \m_j(t)  = 
 \m_j(t) \id_{\Phi( \rho(t,x_j)) = \lambda_j^{-1}}. \\
&\ \ \ \ \ \ \ \ \ \text{ for a.e.~$t\in[0,T]$ and all $j \in J_c$}
\big] = 1. \nonumber
\end{align}  
\end{Lem}

\begin{proof}
Fix $j\in J_c$ and $G\in C[0,T]$.
We first address \eqref{max_rho, g bded}.
As $g(\cdot)\leq 1$, it is trivial that, for all $N$
\[
\P_N \Big[ 
\int_0^T | G(s) |
\left(
 \lambda_j^{-1}g (\eta_s(k_{j,N}))
 -
 \phi_j
 \right)
  ds >\delta \Big]
  =0.
\]
We now use the local `replacement' Lemma \ref{local replace, slow site} to replace
$\lambda_j^{-1}g (\eta_s(k_{j,N}))$ by 
$\Phi(\eta_s^{\eps N,+}(k_{j,N}))$, to obtain
\[
\lim_{\eps\to 0} \limsup_{N\to \infty}
\P_N \Big[ 
\int_0^T | G(s) |
\left(
\Phi(\eta_s^{\eps N,+}(k_{j,N}))
 -
 \phi_j
 \right)
  ds >\delta \Big]
  =0.
\]
Then, \eqref{max_rho, g bded} follows by taking $N\to \infty$,
 cf.~Step $4$ in Lemma \ref{lem: boundary cond, super}.

To show \eqref{m=0_if_rho<max},
we take the same sequence $\{F_\eps(x)\}$ from the proof of Lemma \ref{lem: bdry at critical}.
We now observe a key microscopic boundary relation:
\begin{equation} \label{eqn: mj=0_if_rho}
\lim_{\eps\to 0}  
\limsup_{N\to \infty}
\E_N
\Big[
 \int_0^T \big |
\left( \langle F_\eps, \pi^N_s \rangle  \wedge 1  \right)
\big( \phi_j - 
 \lambda_j^{-1}g (\eta_s(k_{j,N}))
  \big) \big |
  ds  
\Big]
=0.
\end{equation}
Indeed, as $\lim_{n\to\infty}g(n) = 1$, let $A=A(\e)$ be such that $|g(n) - 1| < \e$ for all $n\geq A$. 
Then, in the above display,  on the one hand, the absolute value is bounded above by $\e$ when $\eta_s(k_{j,N}) \geq A$. On the other hand, $\langle F_\eps, \pi^N_s \rangle$ is less than $N^{-1}(A + \sum_k^{\circ,\eps} \eta_s(k) )$
when $\eta_s(k_{j,N}) < A$. 
 Here, the sum $\sum_k^{\circ,\eps}$ is over all $k$ such that $|k/N -x_j| \leq \eps$ and $k\neq k_{j,N}$. In considering these cases, when
 $\eta_s(k_{j,N}) < A$ and $\eta_s(k_{j,N})\geq A$ respectively,
the estimate \eqref{eqn: mj=0_if_rho} follows from \eqref{nbhd expected nmbr} of the `particle numbers' Lemma \ref{lem: total nmbr bound}.

Now, as $| \langle F_\eps, \pi^N_s \rangle  \wedge 1 | $ is bounded, we observe that the local `replacement' Lemma \ref{local replace, slow site}
remains effective when the test function is taken in form $G(s) 
\left( \langle F_\eps, \pi^N_s \rangle  \wedge 1  \right)$. 
Thus, we may replace $\lambda_j^{-1}g (\eta_s(k_{j,N}))$ by 
$\Phi(\eta_s^{\eps N,+}(k_{j,N}))$ and obtain
\[
\lim_{\eps\to 0}  
\limsup_{N\to \infty}
\P_N
\Big[
\Big|
 \int_0^T G(s) 
\left( \langle F_\eps, \pi^N_s \rangle  \wedge 1  \right)
\big( \phi_j - 
\Phi(\eta_s^{\eps N,+}(k_{j,N}))
  \big)
  ds  \Big | > \delta
\Big]
=0.
\]
Again, by following Step $4$ in Lemma \ref{lem: boundary cond, super}, we have
\[
\lim_{\eps\to 0}  
Q
\Big[
\Big|
 \int_0^T G(s) 
\left( \langle F_\eps, \pi_s \rangle  \wedge 1  \right)
\big( \phi_j - 
 \Phi\big( \rho^{+,\theta}_{x_j}(s)\big)
  \big)
  ds  \Big | > \delta
\Big]
=0. 
\]
Also, by Lemma \ref{lem: abs continuity},
 it holds that
$\lim_{\eps\to 0} \int_0^T 
\left| \left( \langle F_\eps, \pi_s \rangle  \wedge 1  \right) -
 \left( \m_j(s) \wedge 1  \right) \right| ds
 = 0$ almost surely with respect to $Q$. 
 As $\phi_j$ and $\Phi(\cdot)$ are bounded, 
 we conclude the proof of \eqref{m=0_if_rho<max}.

It remains to show \eqref{eqn: D_c bdry, bounded g}. Following Step $5$ in Lemma \ref{lem: boundary cond, super}, we obtain from \eqref{max_rho, g bded} that, $Q$-almost surely,
$\Phi( \rho(s,x_j) )\leq \phi_j$
for a.e.~$t\in [0,T]$ and all $j\in J_c$. 
To finish the proof, it suffices to show that $\m_j(s) = 0$ whenever $\Phi ( \rho(s,x_j)) < \phi_j$. In fact,
by \eqref{m=0_if_rho<max}, we have
$( \m_j(t) \wedge 1)
\big( \phi_j - 
 \Phi\big( \rho(s,x_j)\big)
  \big) = 0$. 
  As we have $ \Phi( \rho(s,x_j))\leq \phi_j$, we have that $\m_j(s) \wedge 1 = 0$ (therefore $\m_j(s) = 0$) when $\Phi ( \rho(s,x_j)) < \phi_j$, completing the proof. 
\end{proof}

\section{Local $1$ and $2$-block estimates of bulk sites}
\label{sec: rplmt lem}
In this section we address the $1$ and $2$-block estimates for non-defect sites $\T_N \setminus \D_N$.
These estimates are obtained through a Rayleigh-type estimation of a
variational eigenvalue expression derived from a Feynman-Kac bound.  These bounds will hold in either $g(n)\sim n^\alpha$ or $g$ bounded settings.

\subsection{Local $1$-block estimate}
We start from recalling the concept of spectral gap which is used to prove our local $1$-block estimate.
For $k\in\T_N$ and $l\geq 1$, define the set $\L_{k,l}=\left\{k-l,k-l+1,\ldots,k+l\right\}\subset \T_N$. 
Let $\Omega_{k,l} = \N_0^{\L_{k,l}}$ be the state space of
configurations restricted on sites $\L_{k,l}$.
Define the state space
of configurations with exactly $j$ particles on the sites $\L_{k,l}$:
$$\Omega_{k,l,j} = \Big\{\eta \in \Omega_{k,l}: \sum_{x\in \L_{k,l} }\eta(x) = j\Big \}.$$
Consider the generator $L_{k,l}$ on $\Omega_{k,l}$ given by
\begin{equation*}  \label {eqn: L_k l}
\begin{split}
L_{k,l}f(\eta)
=&
\sum_{x,x+1\in \L_{k,l}}
\Big\{
g(\eta(x))
\left [  f\left(\eta^{x,x+1} \right)  - f(\eta)   \right] 
+
g(\eta(x))
\left [  f\left(\eta^{x+1,x}  \right)   - f(\eta) \right ] 
\Big\}.
\end{split}
\end{equation*}
Recall the generator $L_N$ from \eqref{eqn: generator L}.
Notice that, for each $k,l$ such that $\L_{k,l}\cap \D_N=\emptyset$, the generator $L_{k,l}$ coincides with $L_N$ localized on $\L_{k,l}$.
 
For any $\rho>0$, let $\nu_\rho$ be the product measure on
$\Omega = \N_0^{\T_N}$ with common marginal
$\Pb_{\Phi(\rho)}$ on each site $k\in \T_N$,
and let $\nu^\rho_{k,l}$ be its restriction to $\Omega_{k,l}$.
Let $\nu_{k,l,j}$ be  $\nu_{k,l}^\rho$ conditioned on total number of particles on $\L_{k,l}$ being $j$. 
Notice that $\nu_{k,l,j}$ does not depend on $\rho$.  It is well-known that both $\nu^\rho_{k,l}$ and $\nu_{k,l,j}$ are  invariant measures with respect to the localized generator $L_{k,l}$ (cf.~\cite{A}).
For $\kappa = \nu_{k,l}^\rho$ or $\nu_{k,l,j}$,
the corresponding Dirichlet form is given by
\begin{equation} \label {d form symm}
\begin{split}
E_\kappa
\left[ f(-L_{k,l} f) \right]
=&
\sum_{x,x+1\in \L_{k,l}}
E_\kappa
 \left[ 
 g(\eta(x)) \left(  f\left(\eta^{x,x+1} \right)  - f(\eta)   \right)^2
 \right].
\end{split}
\end{equation}

For $j\geq 1$, let $b_{l,j}$ be the spectral gap of $-L_{k,l}$ on $\Omega_{k,l,j}$ (cf.\,p.~374, \cite{KL}):
\begin{equation} \label{def of b_lj}
b_{l,j}:=
\inf_{f} 
\dfrac
 {E_{\nu_{k,l,j}}[ f(-L_{k,l}f)]}
 {{\rm Var}_{\nu_{k,l,j}} (f)}
\end{equation}
where the infimum is taken over all $L^2(\nu_{k,l,j})$ functions $f$ from $\Omega_{k,l,j}$ to $\R$.
For all $l, j\geq 1$, as $\Omega_{k,l,j}$ is a finite space and the localized process is irreducible, we have $b_{l,j}>0$. 
As a consequence, we have the following Poincar\'e inequality: for all $f\in L^2(\nu_{k,l,j})$
\begin{equation} \label{Poincare for 1 block}
{\rm Var}_{\nu_{k,l,j}} (f)
\leq
 C_{l,j} 
 E_{\nu{k,l,j}}
 \left[ f(-L_{k,l}f) \right]
\end{equation}
where 
$
 C_{l,j}
:=
b_{l,j}^{-1} < \infty$ for $j\geq 1$ and $C_{l,0} = 0$.
We remark that even though, for a large class of $g(\cdot)$'s, sharp estimates of
 $b_{l,j}$ are available in the literature, 
 we will only need that $b_{l,j}$ is strictly positive for all $l,j\geq 1$.

\medskip
We now prove the local $1$-block estimate for regular sites:
\begin{Lem}[Local $1$-block estimate]
\label{lem: 1 block}
For any bounded function $G$ on $[0,T]\times \T$, we have
\[
\limsup_{l\to \infty} \limsup_{N\to \infty}
\sup_{k,k' }
\E_N
\Big[
\Big |
\int_0^T
G(s,k'/N) 
\left(
g(\eta_s(k')) -  \Phi(\eta_s^l(k))\right) ds
\Big |
\Big]
=
0
\]
where the $\sup$ is taken over all $k$ and $k'$ such that $k'\in \Lambda_{k,l}$ and $\Lambda_{k,l} \cap \D_N =\emptyset$.
\end{Lem}
\begin{proof}
We separate the argument into $5$ steps.

\vskip .1cm
{\it Step 1.}
We first introduce a cutoff of large densities.
Let 
\[
V_{k,k',l}(s,\eta)
:=
G(s,k'/N) 
\left(
g(\eta(k'))
- 
\Phi(\eta^l(k))\right).\]
As $g(n)\leq g^* n$ and $\Phi(x)\leq g^*x$, we have
\[
\E_N \Big[ \Big | \int_0^T V_{k,k',l}(s,\eta_s)  \id_{\eta_s^l(k) > A} ds\Big |\Big]
\leq
 g^* \|G\|_\infty  \int_0^T 
  \E_N \left[ \left(\eta_s(k')+ \eta_s^l(k) \right)  \id_{\eta_s^l(k)> A} \right] ds .
\]
By attractiveness (cf.~Section \ref{sec: coupling}), in both $g(n)\sim n^\alpha$ and bounded settings (as $\Lambda_{k,l}\cap \D_N=\emptyset$, the stochastic bound $\kappa^N_{c'}$ and $\Pam^N_{c'}$  agree with $\nu_{c'}$ on $\Lambda_{k,l}$), the last expectation is bounded above by
\[
\begin{split}
E_{\Pam_{c'}^N} \left[ \left(\eta(k')+\eta^l(k) \right)  \id_{\eta^l(k) > A} \right]
\leq&
A^{-1} 
E_{\nu_{c'}} 
\left[ \left(\eta(k') \eta^l(k) \right) 
+
\left(\eta^l(k) \right)^2 \right] \\
 \leq&
 A^{-1} 
E_{\nu_{c'}} 
\left[ \left(\eta(k') \right)^2 
+
2 \left( \eta^l(k) \right)^2  \right].
\end{split}
\]
Notice that $\left( \eta^l(k) \right)^2 \leq (2l+1)^{-1} \sum_{j\in \Lambda_{k,l}} \left(\eta(j) \right)^2$
and, under $\nu_{c'}$, $\left\{\eta(j)\right\}_{j\in \Lambda_{k,l}}$ has common distribution $\Pb_{\Phi(c')}$.
We obtain $\E_N \big[ \big | \int_0^T V_{k,k',l}(s,\eta_s)  \id_{\eta_s^l(k) > A} ds \big |\big] \to 0$ as $N$, $l$, and $A$ approach $\infty$ in this order. 

Therefore, to prove the lemma, it will be enough to show, for all $A>0$, that
\begin{equation*}
\limsup_{l\to \infty} \limsup_{N\to \infty}
\sup_{k,k'}
\E_N
\Big[
\Big |
\int_0^T
V_{k,k',l,A} (s,\eta_s) ds
\Big |
\Big]
=
0
\end{equation*}
where $V_{k,k',l,A} (s,\eta)
:=
V_{k,k',l}(s,\eta) \id_{\{\eta^l (k)\leq A\}}$.

\vskip .1cm
{\it Step 2.} 
As $H(\mu^N| \Pam_{c_0}^N)\leq C N$ for some $C<\infty$, it follows from the entropy inequality
\begin{equation*}
\begin{split}
\E_N
\Big[
\Big |
\int_0^T
V_{k,k',l,A} (s,\eta_s) ds
\Big |
\Big]
\leq 
\dfrac{C}{\gamma}
+
\dfrac{1}{\gamma N} \ln \E_{\Pam_{c_0}^N}
\Big[
\exp
 \Big\{
  \gamma N  
  \Big |
\int_0^T
V_{k,k',l,A} (s,\eta_s) ds
\Big |
  \Big\}
\Big].
\end{split}
\end{equation*}
The absolute value in the right hand side of last inequality can be dropped by using $e^{|x|} \leq e^x + e^{-x}$.
By Feynman-Kac formula (cf.\,p.336, \cite{KL}),
\begin{equation*}
\begin{split}
\dfrac{1}{\gamma N} \ln \E_{\Pam_{c_0}^N}
\Big[
\exp
 \Big\{
  \gamma N  
\int_0^T
V_{k,k',l,A} (s,\eta_s)
 ds
  \Big\}
\Big]
\leq
\dfrac 1 {\gamma N} 
\int_0^T \lambda_{N,l}(s) ds
\end{split}
\end{equation*}
where $\lambda_{N,l}(s)$ is the largest eigenvalue of 
$N^2 L_N + \gamma N V_{k,k',l,A}(s,\eta)$.

\vskip .1cm
{\it Step 3.}
Fix $s\in [0,T]$; we will omit the argument $s$ in
 $\lambda_{N,l}(s)$ to simplify notation, that is $ \lambda_{N,l} =  \lambda_{N,l}(s)$.
Note the variational formula for $\lambda_{N,l}$: 
\begin{equation*}
(\gamma N)^{-1} \lambda_{N,l}
=
\sup_f
\big\{
E_{\Pam_{c_0}^N}\big[  V_{k,k',l,A} f  \big]
-\gamma^{-1} N 
E_{\Pam_{c_0}^N}\big [
\sqrt f (-L_N \sqrt f)
\big]
\big\},
\end{equation*}
where the supremum is over all $f$ which are densities with respect to $\Pam_{c_0}^N$ (cf.\,\cite{KL}, p.\,377).

Let $f_{k,l}= E_{\Pam_{c_0}^N}\big[f|\Omega_{k,l}\big]$, be the conditional expectation of $f$ given the variables on $\L_{k,l}$. 
For $\mu=\Pam_{c_0}^N$ let $\mu_{k,l}$ be $\mu$ restricted on $\Lambda_{k,l}$. 
Clearly, $\mu_{k,l}=\nu_{k,l}^{c_0}$.
  Since the Dirichlet form $E_{\Pam_{c_0}^N}\left [
\sqrt f (-L_N \sqrt f)
\right]$ is convex, we have
\begin{equation*}
(\gamma N)^{-1} \lambda_{N,l}
\leq
\sup_{f_{k,l}}
\left\{
E_{\mu_{k,l}}\left [ V_{k,k',l,A} f_{k,l}  \right ]
-\gamma^{-1} N 
E_{\mu_{k,l}}\left [
\sqrt {f_{k,l}} (-L_{k,l} \sqrt {f_{k,l}})
\right]
\right\}.
\end{equation*}

\vskip .1cm
{\it Step 4.}
We now decompose $f_{k,l}  d\mu_{k,l}$ with respect to sets $\Omega_{k,l,j}$ of configurations with total particle number $j$ on $\L_{k,l}$:
\begin{equation} \label {beta 0, c klj}
E_{\mu_{k,l}}\left[  V_{k,k',l,A} f_{k,l}  \right]
=
\sum_{j\geq 0}
c_{k,l,j}(f) \int V_{k,k',l,A} f_{k,l,j} d\mu_{k,l,j},
\end{equation}
where
$ c_{k,l,j}(f) = \int_{\Omega_{k,l,j}} f_{k,l} d\mu_{k,l}$ and 
$ f_{k,l,j} = c_{k,l,j}(f)^{-1} \mu_{k,l} \left( \Omega_{k,l,j} \right) f_{k,l}$.  Here, $\sum_{j\geq 0} c_{k,l,j} = 1 $ and $f_{k,l,j}$ is a density with respect to $\mu_{k,l,j}$.

Straightforwardly, on $\Omega_{k,l,j}$, we have
\begin{equation*}
\begin{split}
\dfrac{L_{k,l} \sqrt {f_{k,l}}}{\sqrt {f_{k,l}}}
=
\dfrac{L_{k,l} \sqrt {f_{k,l,j}}}{\sqrt {f_{k,l,j}}}.
\end{split}
\end{equation*}
Using \eqref{beta 0, c klj}, we write
\begin{equation*}
E_{\mu_{k,l}}\Big[
\sqrt {f_{k,l}} (-L_{k,l} \sqrt {f_{k,l}})
\Big]
=
\sum_{j\geq 0}
c_{k,l,j}(f) 
E_{\mu_{k,l,j}}\Big[
\sqrt {f_{k,l,j}} (-L_{k,l} \sqrt {f_{k,l,j}})
\Big ].
\end{equation*}
Then, we get
\begin{equation*}
(\gamma N)^{-1} \lambda_{N,l}
\leq
\sup_{0\leq j\leq A(2l+1)}
\sup_f
\left\{
E_{\mu_{k,l,j}}\left[ V_{k,k',l,A} f  \right]
-\gamma^{-1} N 
E_{\mu_{k,l,j}}\left [
\sqrt f (-L_{k,l} \sqrt f)
\right]
\right\},
\end{equation*}
where the second supremum is on all densities $f$ with respect to $\mu_{k,l,j}$.

\vskip .1cm
{\it Step 5.}
Let
\begin{equation*}
\widehat V_{k,k',l,A}
=
V_{k,k',l,A}
-
E_{\mu_{k,l,j}}
\left[V_{k,k',l,A}   \right].
\end{equation*}

Let $C_{l,A,G}$ be such that $\| \widehat V_{k,k',l,A}\|_\infty \leq C_{l,A,G}$. Recall $C_{l,j}$ the inverse spectral gap of $L_{k,l}$ (cf.~\eqref{Poincare for 1 block}).
We now use the Rayleigh expansion (cf.\,\cite{KL}, pp.\,375--376, Appendix 3, Theorem 1.1)
\begin{equation} \label {eqn: rayleigh expansion}
\begin{split}
&E_{\mu_{k,l,j}}\left[   \widehat V_{k,k',l,A} f  \right]
-\gamma^{-1} N 
E_{\mu_{k,l,j}}\big [
\sqrt f (-L_{k,l} \sqrt f)
\big ]\\
\leq&
\dfrac{\gamma N^{-1}}{1-2C_{l,A,G} C_{l,j}\,\gamma N^{-1}}
E_{\mu_{k,l,j}}\left[
  \widehat V_{k,k',l,A}  (-L_{k,l} )^{-1}
  \widehat V_{k,l,A}
\right ].
\end{split}
\end{equation}

The spectral gap of $L_{k,l}$ also implies that $\|L_{k,l}^{-1}\|_2$, the $L^2(\mu_{k,l,j})$ norm of the operator $L_{k,l}^{-1}$ on 
mean zero functions, is less than or equal to $C_{l,j}$.
Now, by Cauchy-Schwarz and the estimate of $\|L_{k,l}^{-1}\|_2$, we have 
\begin{equation*} 
\begin{split}
E_{\mu_{k,l,j}}\left[
  \widehat V_{k,k',l,A}  (-L_{k,l} )^{-1}
  \widehat V_{k,k',l,A}
\right ]
\leq
C_{l,j}
E_{\mu_{k,l,j}}\left[
  \widehat V_{k,k',l,A} ^2
\right ]
\leq  C_{l,j}C^2_{l,A,G}.
\end{split}
\end{equation*}

Accordingly, retracing our steps, noting \eqref{eqn: rayleigh expansion}, we have that $\E_N
\big[
\big |
\int_0^T
 V_{k,k',l,A} (\eta_s) ds
\big |
\big]$ 
is less than or equal to
\begin{equation*}
\begin{split}
\dfrac{C_0}{\gamma}
+
\sup_{0\leq j\leq A(2l+1)}
\dfrac{T\gamma  N^{-1} C_{l,j} C_{l,A,G}^2 } {1-2C_{l,A,G} C_{l,j}\,\gamma N^{-1}}
+
T \sup_{0\leq j\leq A(2l+1)}
 E_{\mu_{k,l,j}}
\left[V_{k,k',l,A} \right].
\end{split}
\end{equation*}
Taking $N\to \infty$, first $\sup$ term vanishes. Notice that the expression $\sup E_{\mu_{k,l,j}}\left[ V_{k,k',l,A} \right]$  is independent of $N$ and vanishes as $l\to \infty$. 
In fact,  as $\mu_{k,l,j} = \nu_{k,l,j}$ is translation-invariant
\[   
\big| E_{\mu_{k,l,j}}\left[ V_{k,k',l,A} \right] \big|
\leq
\| G \|_\infty
\Big |
E_{\nu_{0,l,j}}
[g(\eta(0))]
-
E_{\nu_{j/(2l+1)}}[g(\eta(0))]
\Big |.
\]
By equivalence of ensembles (cf.\,p.355, \cite{KL}),  
the right hand side of the above disaplay vanishes as $l\to \infty$, uniformly for $\rho = j/(2l+1)\in [0,A]$. 
The lemma now is proved by letting $\gamma \to \infty$.
\end{proof}


\subsection{Local $2$-block estimate}
\label{subsec: 2 block}
 We now detail the local $2$-block estimate following the outline of the local $1$-block estimate.  
Recall the notation $\L_{k,l}$ from the $1$-block estimate and let  $\L_{k,k',l} = \L_{k,l} \cup \L_{k',l}$ for $|k-k'|>l$.
Define the generator $L_{k,k',l}$ on $\Omega_{k,k',l} = \N_0^{\L_{k,k',l}}$:
\begin{equation*}
\begin{split}
L_{k,k',l}f(\eta)
&= L_{k,l} f(\eta) + L_{k',l} f(\eta)  \\
&+
g(\eta(k+l))
\big [  f\big(\eta^{k+l,k'-l} \big)  - f(\eta)   \big] 
+
g(\eta(k'-l))
\big [  f\big(\eta^{k'-l,k+l}  \big)   - f(\eta) \big ] .
\end{split}
\end{equation*}
When $|k-k'|$ is large, the process governed by $L_{k,k',l}$ in effect treats the blocks $\L_{k,l}$ and $\L_{k',l}$ as adjacent, with a connecting bond.

Let $\Omega_{k,k',l,j}: = \{\eta\in \Omega_{k,k',l}: \sum_{x\in \L_{k,k',l}} \eta(x) = j\}$.
As before, the localized measure $\nu^\rho_{k,k',l}$ defined by $\nu_\rho$ limited to sites in $\L_{k,k',l}$, 
as well as $\nu_{k,k',l,j}$,  the canonical measure of $\nu^\rho_{k,k;,l}$ on $\Omega_{k,k',l,j}$, are both invariant and reversible with respect to $L_{k,k',l}$.

The corresponding Dirichlet form, with measure $\kappa$ given by $\mu_{k,k',l}$ or $\mu_{k,k',l,j}$, is given by
\begin{equation} \label {D form, 2 block}
\begin{split}
E_{\kappa} \left[ f(-L_{k,k',l} f) \right ]
=&
\sum_{x,x+1\in \L_{k,k',l}}
 E_\kappa \left[ 
g(\eta(x))
\left [  f\left(\eta^{x,x+1}  \right)   - f(\eta) \right ]^2
 \right ]\\
 &+
  E_\kappa \Big [
g(\eta(k+l))
\left [  f\left(\eta^{k+l,k'-l}  \right)   - f(\eta) \right ]^2
 \Big ].
\end{split}
\end{equation}

For $l,j\geq 1$, let $b_{l,l,j}$ be the spectral gap of $-L_{k,k',l}$ on $\Omega_{k,k',l,j}$ ( cf.~\eqref{def of b_lj}). As $b_{l,l,j}$ is strictly postive, we have the following
Poincar\'e inequality (cf.~\eqref{Poincare for 1 block}):
for all $f\in L^2 (\nu_{k,k',l,j})$
\begin{equation} \label{Poincare for 2 block}
{\rm Var}_{\nu_{k,k',l,j}} (f)
\leq
 C_{l,l,j} 
 E_{\nu_{k,k',l,j}}
 \left[ f(-L_{k,k',l}f) \right]
\end{equation}
where 
$
 C_{l,l,j} := b_{l,l,j}^{-1}$ for $j\geq 1$ and $C_{l,l,0} = 0$.

\medskip
We now state and show a local $2$-blocks estimate.  The scheme is similar to that of the local $1$-block estimate.

\begin{Lem}[Local $2$-block estimate] \label{lem: 2 blocks}
 We have
 \begin{equation} \label {eqn: local 2 block}
\limsup_{l\to \infty}
\limsup_{\eps \to 0} \limsup_{N\to \infty}
\sup_{k,k''}
\E_N
\Big[
\int_0^T
\left |
\Phi(\eta_s^l(k))
-
\Phi(\eta_s^{\eps N}(k''))
\right |
ds
\Big]
=
0
\end{equation}
where the $\sup$ is taken over all $k$ and $k''$ such that $ \Lambda_{k,l} \subset \Lambda_{k'',\eps N}$ and $\Lambda_{k'',\eps N} \cap \D_N =\emptyset$.
\end{Lem}

\begin{proof}
We separate the argument into steps.
\vskip .1cm

{\it Step 1.}
Since $\Phi(\cdot)$ is Lipschitz,
to prove the lemma, it suffices to show \eqref{eqn: local 2 block} with $\Phi(\eta_s^l(k)) - \Phi(\eta_s^{\eps N}(k''))$ replaced by $\eta_s^l(k) - \eta_s^{\eps N}(k'')$.
We may further replace $ \eta_s^{\eps N}(k'')$ by 
$ (2\eps N +1)^{-1}\sum_{k'\in \L_{k'',\eps N-l}} \eta^l (k')$. Indeed, the replacement error is
\[
\E_N
\Big[
\Big |
\eta_s^{\eps N}(k'')
-
\dfrac{1}{2\theta N+1} \sum_{k'\in \L_{k'',\eps N-l}}
\eta_s^{l}(k')
\Big |
\Big]
\leq
\dfrac{2l+1}{2\eps N +1} 
\E_N
\big[
\eta_s^l(k''-\eps N+l) + \eta_s^l(k''+\eps N-l)
\big] .
\]
By attractiveness (as explained in Step 1 of the proof of the local 1-block Lemma \ref{lem: 1 block}), for either case $g(n)\sim n^\alpha$ or $g$ bounded, the expectation term $\E_N$ is bounded uniformly in $t$, $l$, and $N$, the right-hand side of the above display vanishes as $N\uparrow\infty$ first.

Therefore, the lemma will follow if we show
\begin{equation} \label {eqn: two block, no cutoff}
\limsup_{l\to \infty}
\limsup_{\theta \to 0}
 \limsup_{N\to \infty}
 \sup_{k,k'}
\E_N
\Big[
\int_0^T 
\Big |
\eta_s^{l}(k)
-
\eta_s^l(k')
 \Big |
ds
\Big]
=
0
\end{equation}
where the $\sup$ is taken over all $k,k'$ such that $\Lambda_{k,k',l} \cap \D_N = \emptyset$ and $2l+1\leq k'-k \leq \theta N$.

\vskip .1cm

{\it Step 2.}
By a similar coupling argument as in the Step 1 of the proof of local $1$-block Lemma \ref{lem: 1 block},
we may apply a cutoff of large densities.
Therefore, to prove the lemma, it suffices to show
\begin{equation*}
\limsup_{l\to \infty}
\limsup_{\theta \to 0}
 \limsup_{N\to \infty}
 \sup_{k,k'}
\E_N
\Big[
\int_0^T 
\Big |
\eta_s^{l}(k)
-
\eta_s^l(k')
 \Big |
  \id_{\{\eta_s^{l} (k,k') \leq A \}}
ds
\Big]
=
0
\end{equation*}
where $\eta_s^{l} (k,k') =  \eta_s^l(k) + \eta_s^l(k')$
and the $\sup$ is over $k,k'$ as in \eqref{eqn: two block, no cutoff}.

Let $U_{k,k',l,A}(\eta)
:=
\left|
\eta^l (k)
- 
\eta^l(k')
\right|
\id_{\{\eta^l(k,k')\leq A\}}$.
Following the proof of Lemma \ref{lem: 1 block}, for fixed $l,\theta,N,k,k'$, in order to estimate
$\E_N
\big[
\int_0^T 
U_{k,k',l,A} (\eta_s)
ds
\big]$,
it suffices to bound
\begin{equation}
\label{eigen_2blocks}
(\gamma N)^{-1} \lambda_{N,l}
=
\sup_f
\left\{
E_{\Pam_{c_0}^N} \left [  U_{k,k',l,A} f  \right ]
-\gamma^{-1} N 
E_{\Pam_{c_0}^N} \left [
\sqrt f (-L_N \sqrt f)
\right ]
\right\}.
\end{equation}
where the supremum is over all $f$ which are densities with respect to $\Pam_{c_0}^N$.

\vskip .1cm
{\it Step 3.}
Recall the generator $L_{k,k',l}$ and its Dirichlet form defined in the beginning of this subsection. 
We now argue the following Dirichlet form inequality
\begin{equation}\label {eqn: d form kk'l LN}
E_{\Pam_{c_0}^N} \big [
\sqrt f (-L_{k,k',l}\sqrt f)
\big ]
\leq
 \theta N
E_{\Pam_{c_0}^N}\big [
\sqrt f (-L_N \sqrt f)
\big ].
\end{equation}
The Dirichlet form with respect to the full generator $L_N$ under $\Pam_{c_0}^N $ is given by
\begin{equation} \label {d form, full LN}
E_{\Pam_{c_0}^N } \left[ f(-L_N f)\right ]
= \sum_{k\in\T_N}  E_{\Pam_{c_0}^N} \Big[
g_{k,N}(\eta(k))
\big(f(\eta^{k,k+1}) - f(\eta) \big)^2
\Big].
\end{equation}

First, writing out the Dirichlet form in \eqref{D form, 2 block}, in terms of the product measure $\Pam_{c_0}^N$, we have
\begin{equation*}
\begin{split}
E_{\Pam_{c_0}^N} \left [ f(-L_{k,k',l} f) \right ]
=&
\sum_{x,x+1\in \L_{k,k',l}}
 E_{\Pam_{c_0}^N} \Big[
g(\eta(x))
\big(f(\eta^{x,x+1}) - f(\eta) \big)^2
\Big]\\
 &+
 E_{\Pam_{c_0}^N} \Big[
g(\eta(k+l))
\big(f(\eta^{k+l,k'-l}) - f(\eta) \big)^2
 \Big].
\end{split}
\end{equation*}

Next, by adding and subtracting at most $\theta N$ terms, we have
\begin{equation*}
\begin{split}
\left [  f\left(\eta^{k+l,k'-l}  \right)   - f(\eta) \right ]^2
\leq
(k'-k-2l) \sum_{q=0}^{k'-k-2l-1}
\left [f\left(\eta^{k+l,k+l+q+1}  \right)  -  f\left(\eta^{k+l,k+l+q}  \right)  \right ]^2
.
\end{split}
\end{equation*}
 By  the change of variables $\xi = \eta^{k+l,k+l+q} $, which takes away a particle at $k+l$ and adds one at $k+l+q$, we have 
$\Pam_{c_0}^N (\eta)
=
\dfrac{g(\eta(k+l+q)+1)}{g(\eta(k+l))}
\Pam_{c_0}^N (\xi)$.
Then
\begin{align*}
& E_{\Pam_{c_0}^N } \left [
g(\eta(k+l))
\left [f\left(\eta^{k+l,k+l+q+1}  \right)  -  f\left(\eta^{k+l,k+l+q}  \right)  \right ]^2 \right]\\
 =&
\sum_{\xi} \Pam_{c_0}^N (\eta) 
g(\eta(k+l))
 \left [f\left(\xi^{k+l+q,k+l+q+1}  \right)  -  f\left(\xi  \right)  \right ]^2\\
 =&
 E_{\Pam_{c_0}^N } \Big [ 
 g(\eta(k+l+q))
 \Big [  f\left(\eta^{k+l+q,k+l+q+1}  \right)   -f\left(\eta  \right)  \Big]^2
 \Big ].
\end{align*}
From these observations, \eqref {eqn: d form kk'l LN} follows.

\vskip .1cm
{\it Step 4.} 
Let $\mu_{k,k',l}$ be the restriction of $\mu=\Pam_{c_0}^N$ to $\L_{k,k',l}$. Clearly, $\mu_{k,k',l} = \nu_{k,k',l}$.
Inputting \eqref{eqn: d form kk'l LN} into \eqref{eigen_2blocks}, and considering the conditional expectation of $f$ with respect to $\Omega_{k,k',l}$ as in the $1$-block estimate proof, we have
\[
(\gamma N)^{-1} \lambda_{N,l}
\leq
\sup_{f_{k.k',l}}
\Big\{
E_{\mu_{k,k',l}} \left [ U_{k,k',l,A} f_{k,k',l}  \right ]
-\dfrac{1}{\theta \gamma} 
E_{\mu_{k,k',l}} \Big [
\sqrt {f_{k,k',l}} (-L_{k,k',l} \sqrt {f_{k,k',l}})
\Big ]
\Big\},
\]
where the supremum is over densities $f_{k,k',l}$ with respect to $\mu_{k,k',l}$.

Again, as in the proof of the $1$-block estimate, decomposing $f_{k,k',l}  d\mu_{k,k'.l}$ along configurations with common total number $j$, we need only to bound
\begin{equation*}
\sup_{0\leq j\leq A(2l+1)}
\sup_f
\left\{
E_{\nu_{k,k',l,j}} \left [ U_{k,k',l,A} f  \right ]
-\dfrac{1}{\theta \gamma} 
E_{\nu_{k,k',l,j}} \left [
\sqrt f (-L_{k,k',l} \sqrt f)\right]
\right\},
\end{equation*}
where the supremum is over densities $f$ with respect to $\nu_{k,k',l,j}$.
\vskip .1cm

{\it Step 5.}
Consider the centered object
\begin{equation*}
\widehat U_{k,k',l,A}
=
U_{k,k',l,A}
-
E_{\nu_{k,k',l,j}}
\left[U_{k,k',l,A}   \right].
\end{equation*}
Recall $C_{l,l,j}$, the inverse spectral gap of $L_{k,k',l}$ from \eqref{Poincare for 2 block} and note that $\| \widehat U_{k,k',l,A}\|_{\infty} \leq  A$. 
Using the Rayleigh expansion (cf. p.375, \cite{KL}), we have 
\begin{equation*}
\begin{split}
&E_{\nu_{k,k',l,j}} \big [  \widehat U_{k,k',l,A} f  \big ]
-(\theta \gamma)^{-1}
E_{\nu_{k,k',l,j}} \big [
\sqrt f (-L_{k,k',l} \sqrt f)
\big ]\\
\leq&
\dfrac{\theta \gamma}{1-2 A C_{l,l,j}\,\theta  \gamma}
E_{\nu_{k,k;,l,j}}\big [
 \widehat U_{k,k',l,A}  (-L_{k,k',l} )^{-1}
 \widehat U_{k,k',l,A}
\big ]\\
\leq&
\dfrac{\theta \gamma C_{l,l,j} }{1-2 A C_{l,l,j}\,\theta \gamma}
E_{\nu_{k,k',l,j}}\big [
 \widehat U_{k,k',l,A} ^2
\big ]
\to 0 \text{ as } \theta \to 0.
\end{split}
\end{equation*}

\vskip .1cm {\it Step 6.}  Recall the definition of $U_{k,k',l,A}$ in Step 2.  To finish, we still need to estimate the centering term
$E_{\nu_{k,k',l,j}}\left[ U_{k,k',l,A} \right ]$.
By adding and subtracting $j/(4l+2)$, we 
need only bound 
$E_{\nu_{k,k',l,j}}\left[\big |\eta^l(k)- j/(4l+2)\big|\right]$.   By exchangeability and
an equivalence of ensemble estimate (cf.~p.~355 \cite{KL}), the canonical variance
\begin{align*}
&E_{\nu_{k,k',l,j}}\left[\big |\eta^l(k)- j/(4l+2)\big|^2 \right] =
O(l^{-1})E_{\nu_{k,k',l,j}} \left[(\eta(k) - j/(4l+2))^2\right] \\
&\ \ \ \ \ \ \ \ \ \ + O(1) E_{\nu_{k,k',l,j}}\left[(\eta(k) - j/(4l+2) )(\eta(k+1)-j/(4l+2))\right]
\end{align*}
 and is further bounded by
$C(A){\rm Var}_{\nu^{j/(4l+2)}_{k,k',l} }\left( \eta^l(k)\right)$ for some constant $C(A)$ depending only on $A$.
 This variance is of order $O(l^{-1})$, since the single
site variance
${\rm Var}_{\nu^{j/(4l+2)}_{k,k',l} }\left(\eta(k)\right)$ is uniformly
bounded for $j/(4l+2)\leq A$. 
 Hence,
$\sup_{0\leq j \leq A(4l+2) }E_{\nu_{k,k',l,j}} \left [ V_{k,k',l,A} \right ]$ is of order
$O(l^{-1/2})$, vanishing as $l\uparrow \infty$.  This finishes the proof.
\end{proof}

\begin{Rmk}
\label{remark about +}
We comment that Lemmas \ref{lem: 1 block} and \ref{lem: 2 blocks} with the choice $k=k'=k''$ gives the replacement of $g(\eta_s(k))$ by $\Phi\left(\eta_s ^{\eps N}(k)\right)$ which will be used in the `bulk' replacement below.
Moreover, choosing $k'=x+1$, $k=x+l+1$, and $k''=x+\eps N+1$, 
we may replace $g(x+1)$ by 
$\Phi\left(\eta_s ^{\eps N,+}(x)\right)$.
Here $\eta^{l,+}(x) = \frac{1}{2l+1}\sum_{y=x+1}^{x+2l+1}(x)$ is the average number of particles over the $2l+1$ neighboring sites to the right of $x$.
Such an average will be used to treat `replacement' near the boundary of a defect site in Section \ref{sec: rplmt lem at bdry}.

\end{Rmk}

\subsection{Bulk Replacement Lemma}
Let $G(t,x)$ be a bounded function on $[0,T] \times \T$ with compact support on $[0,T] \times (\T\setminus \D)$.
As in Remark \ref{remark about +}, 
Lemma  \ref{lem: 1 block}  implies that
\begin{equation*} 
\limsup_{l\to \infty}
\limsup_{N\to \infty} 
 \E_N
\Big[
\dfrac1N \sum_{k\in \T_N }
\Big|
\int_0^T 
G(s,k/N) 
\left(
g(\eta_s(k))
- 
\Phi(\eta_s^l(k))\right)ds
\Big|
\Big]
=0
\end{equation*}
and by Lemma \ref{lem: 2 blocks},
 \begin{equation*}  
\limsup_{l\to \infty}
\limsup_{\eps\to 0} 
\limsup_{N\to \infty} 
 \E_N
\Big[
\dfrac1N \sum_{k\in \T_N}
\int_0^T 
G(s,k/N) 
\Big|
\Phi\left( \eta_{s}^{l}(k)\right)
-
\Phi\left( \eta_{s}^{\eps N}(k)\right)
\Big|
ds
\Big]
=0.
\end{equation*}
By Markov's inequality and triangle inequality, we obtain
\begin{Lem} [Bulk Replacement Lemma]
\label{rplmt_global_lem}

For each bounded function $G(t,x)$ on $[0,T] \times \T$ with compact support on $[0,T] \times (\T\setminus \D)$, and $\delta>0$, we have
\begin{equation*}
\begin{split}
\limsup_{\theta\to 0} \limsup_{N\to \infty} 
 \P_N
\Big[
\Big|
\dfrac1N \sum_{k\in \T_N }
\int_0^T 
G\big(s,\frac k N \big) 
\Big( g\left (\eta_s (k)\right )
-\Phi\left(\eta_s^{\eps N}(k)\right)\Big) ds
\Big|
 > \delta
\Big]
=0.
\end{split}
\end{equation*}
\end{Lem}
\begin{Rmk}
We comment, in the $g(n)\sim n^\alpha$ setting (where 'FEM' as stated in \cite{KL} holds), the ``attractiveness" assumption used in the local $1$ and $2$-block estimates to introduce cutoffs of large densities in a local region, may be dropped in the statement of the `global average' Lemma \ref{rplmt_global_lem}.

\end{Rmk}

\section{Replacement at the boundary}
\label{sec: rplmt lem at bdry}
In this section, we show a local replacement near the defect sites, used in Subsection \ref{sec: boundary behavior}, in two steps. In Lemma \ref{replacement: slow site}, we show that the jump rate $g_{k,N}(\eta_s(k))$ at any site $k\in \T_N$ is close to $g_{k+1,N}(\eta_s(k+1))$, that of its neighbor site $k+1$. 

In later use, when $k$ is a defect site in $\D_N$, this neighbor site will be a non-defect or regular site for $N$ large.
Then, we may apply local $1$ and $2$-blocks estimates from last section to obtain our local replacement Lemma  \ref{local replace, slow site} near the defect sites. 

 In both lemmas, the replacements will hold in the $g(n)\sim n^\alpha$ and $g$ bounded settings.

\begin{Lem} 
 \label{replacement: slow site}
 Let $G(\cdot): [0,T]\mapsto \R$ be bounded.
 Then, we have
\begin{equation*}
\limsup_{N\to \infty}
\sup_{k\in \T_N}
\E_N \Big[ \Big| 
\int_0^T G(s) 
\big(
 g_{k,N}(\eta_s(k)) 
-
 g_{k+1,N}(\eta_s(k+1))  \big)
  ds\Big |  \Big]
  =0.
\end{equation*} 
\end{Lem}

\begin{proof}
Let
\[
U_{s,k}(\eta)=\,
2 N^{-1}
 (G(s))^2
   \big(
    g_{k,N}(\eta(k)) 
   +
  g_{k+1,N}(\eta(k+1))   \big).
\]
By Lemma \ref{lem: total nmbr bound},
$ \lim_{N\to \infty} 
\sup_{k\in \T_N}
\E_N \big[\int_0^T U_{s,k}(\eta_s)ds \big] = 0$.
Let
\[
V_{s,k}(\eta)
=\,
G(s)  \big(
 g_{k,N}(\eta(k)) 
-
 g_{k+1,N}(\eta(k+1))  \big).
\]
Then, to prove the lemma, it suffices to show that
\begin{equation} 
\label {ean: kappa N limit}
\limsup_{\kappa \to \infty}
\limsup_{N\to \infty}
\sup_{k\in \T_N}
\E_N \Big[ \Big | 
\int_0^T V_{s,k}(\eta_s) ds\Big | 
-\kappa \int_0^T  U_{s,k}(\eta_s) ds
   \Big]
  =0.
\end{equation}
Recall the initial entropy bound $H(\mu^N| \Pam_{c_0}^N)\leq C N$. Then, for $c=c_0$ and $\kappa>0$, by the entropy inequality, the expectaton in the previous display is bounded from above by \begin{equation*}
\dfrac{C}{\kappa }
+
\dfrac{1}{\kappa N} \ln \E_{\Pam_c^N}
\Big[
\exp
 \Big\{
  \kappa N  
  \Big |
\int_0^T
V_{s,k} (\eta_s) ds \Big |
-
\kappa^2 N \int_0^T  U_{s,k} (\eta_s) ds
  \Big\}
\Big].
\end{equation*}
The absolute value in the right hand side of last inequality can be dropped by using $e^{|x|} \leq e^x + e^{-x}$.
By Feynman-Kac formula (cf.\,p.336, \cite{KL}), we have
\begin{equation*}
\begin{split}
\dfrac{1}{\kappa N} \ln \E_{\Pam_c^N}
\Big[
\exp
 \Big\{
  \kappa N  
\int_0^T
(V_{s,k} - \kappa U_{s,k}) ds
  \Big\}
\Big]
\leq
\dfrac 1 {\kappa N} 
\int_0^T \lambda_{N,k}(s) ds
\end{split}
\end{equation*}
where $\lambda_{N,k}(s)$ is the largest eigenvalue of 
$N^2 L_N + \kappa N (V_{s,k}(\eta) - \kappa U_{s,k}(\eta))$.
Fix $s\in [0,T]$ and note the variational formula for $\lambda_{N,k}$: 
\begin{equation*}
(\kappa N)^{-1} \lambda_{N,k}
=
\sup_f
\left\{
E_{\Pam_c^N}\left[ (V_{s,k}-\kappa U_{s,k})  f  \right]
-\kappa^{-1} N 
E_{\Pam_c^N}\big [
\sqrt f (-L_N \sqrt f) \big ]
\right\}
\end{equation*}
where the supremum is over all $f$ which are densities with respect to $\Pam_c^N$ (cf.\,\cite{KL}, p.\,377).
Thus, to prove \eqref{ean: kappa N limit}, it remains to show, for any density $f$,
\begin{equation}
\label {eqn: U V D_N bound}
E_{\Pam_c^N}[V_{s,k} f]
\leq
E_{\Pam_c^N}
[ \kappa U_{s,k} f ]
+
\kappa^{-1} N 
E_{\Pam_c^N}\big [
\sqrt f (-L_N \sqrt f) \big ].
\end{equation}

By the product structure of $\Pam_c^N$, we have
\[
g_{k,N}(\eta(k)) \Pam_c^N (\eta)
=
g_{k+1,N}(\eta(k+1)+1) \Pam_c^N (\eta^{k,k+1}).
\]
Thus, we compute that
\[
\begin{split}
E_{\Pam_c^N}\left[  V_{s,k} f  \right]
=
&E_{\Pam_c^N} \Big[
 G(s) 
\Big(
 g_{k,N}(\eta(k)) 
-
 g_{k+1,N}(\eta(k+1))  \Big)
 f(\eta) \Big]\\
  =&
E_{\Pam_c^N} \Big[
 G(s) 
 g_{k,N}(\eta(k)) 
\Big(
  f(\eta)
-
  f(\eta^{k,k+1})  \Big)
 \Big]\\
  =&
E_{\Pam_c^N} \Big[
 G(s) 
 g_{k,N}(\eta(k)) 
\Big( \sqrt{ f(\eta)} - \sqrt{f(\eta^{k,k+1})} \Big)
\Big( \sqrt{ f(\eta)} + \sqrt{f(\eta^{k,k+1})} \Big)
\Big].
\end{split}
\]
By Cauchy-Schwarz, for any $A>0$, the above display is estimated from above by 
\[
\begin{split}
&A E_{\Pam_c^N} \Big[
 g_{k,N}(\eta(k)) 
\Big( \sqrt{ f(\eta)} 
- \sqrt{f(\eta^{k,k+1})} \Big)^2 \Big]\\
&\qquad\qquad\qquad+
A^{-1} E_{\Pam_c^N} \Big[
 G(s)^2 
 g_{k,N}(\eta(k)) 
\Big( \sqrt{ f(\eta)} + \sqrt{f(\eta^{k,k+1})} \Big)^2.
 \end{split}
 \]
Notice that the first expectation in the above display is bounded by $E_{\Pam_c^N}\big [
\sqrt f (-L_N \sqrt f) \big]$ (cf.~\eqref{d form, full LN}).
Take $A = \kappa^{-1} N$. 
 The second summand is estimated from above by
 \[
 \begin{split}
&2 \kappa N^{-1} E_{\Pam_c^N} \Big[
 (G(s))^2
  g_{k,N}(\eta(k)) 
  \big(
  f(\eta)
   +
 f(\eta^{k,k+1})  \big)
 \Big]\\
 =\,&
 2 \kappa N^{-1} E_{\Pam_c^N} \Big[
 (G(s))^2
   \big(
    g_{k,N}(\eta(k)) 
   +
  g_{k+1,N}(\eta(k+1))   \big)
   f(\eta)
 \Big]\\
 =\,&
 E_{\Pam_c^N} \big[ \kappa U_{s,k} f \big] .
\end{split}
\]
Retracing the terms, we obtain \eqref{eqn: U V D_N bound}, finishing the proof. 
\end{proof}

We now finish this section with a local replacement lemma at defect sites:
\begin{Lem}[Local replacement at defect sites]
\label{local replace, slow site}
Let $G(\cdot): [0,T]\mapsto \R$ be bounded.
Then, for each defect site $k_{j,N} \in \D_N$,
we have
\begin{equation*}
\lim_{\eps \to 0}
\limsup_{N\to \infty}
\E_N \left[ \left| 
\int_0^T G(s) 
\left(
 g_{k_{j,N},N}(\eta_s(k_{j,N})) 
-
 \Phi(\eta_s^{\eps N,+}(k_{j,N})) \right)
  ds\right| \right]
  =0
\end{equation*} 
where $\eta^{l,+}(k)
:= (2l+1)^{-1} \sum_{k+1\leq x \leq k+2l+1} \eta(x)$.
\end{Lem}
\begin{proof}
Lemma \ref{replacement: slow site}
shows we may replace 
$g_{k_{j,N},N}(\eta_s(k_{j,N}))$
by $g_{k_{j,N}+1,N}(\eta_s(k_{j,N}+1))$.
To finish, notice $k_{j,N}+1 \in \T_N\setminus \D_N$ for $N$ large. Then, we may further replace  $g_{k_{j,N}+1,N}(\eta_s(k_{j,N}+1))$ by $ \Phi(\eta_s^{\eps N,+}(k_{j,N}))$ using Lemma  \ref{lem: 1 block} and Lemma \ref{lem: 2 blocks}; see Remark \ref{remark about +}.
The proof is now complete.
\end{proof}

\section{Energy estimate}
\label{sec: energy section}
By Lemma \ref{lem: abs continuity}, we know that $Q$ is supported on paths 
$\pi_t(dx)$ which can be decomposed into an absolute continuous part  $\rho(t,x)dx$ and atoms $\sum_{j\in J_c} \m_j(t) \delta_{x_j}(dx)$.
 In this section, we prove an energy estimate for $\rho(t,x)$. 
\begin{Prop} \label {prop: weak derivative of Phi}
$Q$ is supported on paths $\pi_t(dx)$ such that $\Phi(\rho(t,x))$ is weakly differentiable with respect to $x$ on $[0,T]\times \T$ and $\partial_x \Phi(\rho(t,x))$, the weakly derivative, satisfies 
\begin{equation} 
\label{eqn: energy of Phi_x}
\int_0^T \int_\T \dfrac{(\partial_x \Phi(\rho(t,x)))^2}{\Phi(\rho(t,x)} dx\, dt < \infty.
\end{equation}
\end{Prop}
\begin{Rmk} \label{rmk on deriv L2}
As $\rho(t,x)\in L^\infty ([0,T]\times \T)$ (see Lemma \ref{lem: abs continuity}) and $\Phi(\cdot)$ is Lipschitz, it follows from \eqref{eqn: energy of Phi_x} that $\partial_x \Phi(\rho(t,x)) \in L^2 ([0,T]\times \T)$.
\end{Rmk} 
\begin{proof}[Proof of Proposition \ref{prop: weak derivative of Phi}]
The proof presented here is based on the one of Theorem 7.1, p.~102, \cite{KL}. However, because of the presence of the defect sites and the difference in the underlying topology, many details are different in subtle ways.

Let $\{H_j \}_{j\in \N}$ be a dense sequence in $C_c^{0,1}([0,T]\times (\T\setminus \D) )$ under the norm $\|H\|_\infty + \|\partial_x H\|_\infty$.
We split the proof into steps.
In the first of two steps,
we show there is a constant $K_0$ such that, for all $\epsilon$ small,
\begin{equation} \label{tilde W limit}
\limsup_{N\to \infty} 
\E_N 
\Big[
\max_{1\leq j\leq m} 
\Big\{ \int_0^T  W_N(\epsilon,H_j(s,\cdot), \eta_s) ds  \Big\}
\Big]
\leq K_0
\end{equation}
where
\begin{equation*}
\begin{split}
 W_N(\epsilon, H(\cdot), \eta)
:=&
\sum_{ x \in \T_N} \dfrac{H(x/N)}{\epsilon N}
\left( g_{x,N}(\eta(x)) - g_{x+\epsilon N,N} (\eta(x+\epsilon N))   \right) \\
&-
\dfrac {2}{N} \sum_{x\in\T_N} \dfrac{H^2(x/N)}{\epsilon N} \sum_{0\leq k\leq \epsilon N} g_{x+k,N}(\eta(x+k)) .
\end{split}
\end{equation*}

\vspace{0.1cm}
{\it Step 1.}
Let $c=c_0$. By the entropy inequality, the expectation in \eqref{tilde W limit} is bounded from above by
\begin{equation*}
\begin{split}
\dfrac{H(\mu^N | \Pam_{c_0}^N)}{2N}
+
\dfrac1{2N} \ln \E_{\Pam_{c_0}^N}
\Big[
\exp
 \Big\{
\max_{1\leq j\leq m} 
\Big\{ 2N \int_0^T W_N (\epsilon,H_j(s,\cdot), \eta_s) ds  \Big\}
  \Big\}
\Big].
\end{split}
\end{equation*}
For convenience, we set $c=c_0$.
Using $H(\mu^N| \Pam_{c}^N)\leq CN$ and $e^{\max a_j} \leq \sum e^{a_j}$, the limsup in $N$ of the previous display is estimated from above by
\begin{equation*}
\begin{split}
\dfrac C2
+
\max_{1\leq j\leq m} 
\limsup_{N\to\infty}
\dfrac1{2N} \ln \E_{\Pam_c^N }
\Big[
\exp
\Big\{
 2N \int_0^T W_N(\epsilon,H_j(s,\cdot), \eta_s) ds
  \Big\}
\Big].
\end{split}
\end{equation*}
By Feynman-Kac formula, for any fixed index $j$, the limsup term in the previous display is less than or equal to
\[
\limsup_{N\to \infty}
 \int_0^T 
 \sup_f
 \left\{
 E_{\Pam_c^N } \left[ W_N (\epsilon, H_j(s,\cdot), \eta) f(\eta) \right]
- \dfrac N2 E_{\Pam_c^N } \Big[\sqrt f (-L_N\sqrt f)\Big] 
\right\}
ds
\]
where the supremum is over all $f$ which are densities with respect to
$\Pam_c^N$. 

\vspace{0.1cm}
{\it Step 2.}
To show \eqref{tilde W limit}, it now remains to show, for all $H$ in $C_c^1 (\T\setminus \D)$, that
\begin{equation} \label{eqn: log exp negative}
\begin{split}
 E_{\Pam_c^N} \left[ W_N(\epsilon, H(\cdot), \eta) f(\eta) \right]
- N E_{\Pam_c^N}  \Big[\sqrt f (-L_N \sqrt f)\Big] 
\leq 0.
\end{split}
\end{equation}

We first compute that $E_{\Pam_c^N } \left[ W_N (\epsilon, H(\cdot), \eta) f(\eta) \right]$ equals
\begin{equation} \label {eqn: H_s derivative estimate}
\begin{split}
&
E_{\Pam_c^N }
\Big[
\sum_{x\in\T_N} \dfrac{H(x/N)}{\epsilon N}
( g_{x,N}(\eta(x)) - g_{x+\epsilon N, N} (\eta(x+\epsilon N)) ) f(\eta)
\Big]\\
&-
\dfrac{2}{N} 
E_{\Pam_c^N }
\Big[ \sum_{x\in \T_N} \dfrac{H^2(x/N)}{\epsilon N} \sum_{0\leq k\leq \epsilon N} g_{x+k,N}(\eta(x+k))
f(\eta) \Big].
\end{split}
\end{equation}
Let $\delta_x$ be the configuration with the only particle at $x$ and  $\eta + \delta_x$ be the configuration obtaining from adding one particle at $x$ to $\eta$
By the definition of $\Pam_c^N $, we have, for each $x$, 
\begin{equation}\label {change of variable, Pam}
E_{\Pam_c^N } [g_{x,N}(\eta(x)) f(\eta)]
=
\varphi E_{\Pam_c^N } [ f(\eta+\delta_x)]
\end{equation}
where $\varphi=\Phi(c)$. Then, the first expectation in 
  \eqref{eqn: H_s derivative estimate} is written as
\begin{equation} \label {eqn: H_s derivative estimate, rewritten}
\begin{split}
\sum_{x\in \T_N} \dfrac{ \varphi H(x/N)}{\epsilon N}
E_{\Pam_c^N }
\Big[
f(\eta+\delta_x) - f(\eta+\delta_{x+\epsilon N})
\Big].
\end{split}
\end{equation}
which is rewritten as
\begin{equation} \label{eqn: H_s derivative estimate, 2nd term, rewritten}
\begin{split}
& E_{\Pam_c^N }
\Big[
\sum_{x\in \T_N} 
\sum_{0\leq k \leq \epsilon N -1}
\dfrac{\varphi H(x/N)}{\epsilon N}
\big(\sqrt{f(\eta+\delta_{x+k})} + \sqrt{f(\eta+\delta_{x+k+1})}\big)\\
&\qquad\qquad\qquad\qquad\qquad\qquad \qquad\times
\big(\sqrt{f(\eta+\delta_{x+k})} - \sqrt{f(\eta+\delta_{x+k+1})}\big)
\Big].
\end{split}
\end{equation}
Using $2ab\leq a^2 + b^2$, for any $A>0$, \eqref{eqn: H_s derivative estimate, 2nd term, rewritten} is bounded from above by
\begin{equation} \label{eqn: H_s estim, separated}
\begin{split}
& E_{\Pam_c^N }
\Big[
\sum_{x\in \T_N} 
\sum_{0\leq k \leq \epsilon N -1}
\dfrac{\varphi H^2(x/N)}{2\epsilon N A}
\big(\sqrt{f(\eta+\delta_{x+k})} + \sqrt{f(\eta+\delta_{x+k+1})}\big)^2
\Big]\\
&\qquad  +
 E_{\Pam_c^N }
\Big[
\sum_{x\in \T_N} 
\sum_{0\leq k \leq \epsilon N -1}
\dfrac{\varphi A}{2\epsilon N}
\big(\sqrt{f(\eta+\delta_{x+k})} - \sqrt{f(\eta+\delta_{x+k+1})}\big)^2
\Big] :=I_1 + I_2.
\end{split}
\end{equation}

The $I_2$ term in \eqref{eqn: H_s estim, separated} is rewritten as
$
\sum_{x\in \T_N}
\dfrac{\varphi A}{2} 
E_{\Pam_c^N } \big[
\big(\sqrt{f(\eta+\delta_{x})} - \sqrt{f(\eta+\delta_{x+1})}\big)^2
\big]$ which, by the change of variable formula \eqref{change of variable, Pam}, is further recognized as the Dirichelet form 
$ \dfrac A2 E_{\Pam_c^N } \big[\sqrt f (-L_N\sqrt f)\big]$, cf.~\eqref{d form, full LN}. 
For the first expectation $I_1$ in \eqref{eqn: H_s estim, separated}, using first $(\sqrt a + \sqrt b)^2 \leq 2(a+b)$ and then \eqref{change of variable, Pam}, it is bounded from above by
\begin{equation} \label {eqn: H_s derivative estimate, ii}
\sum_{x\in \T_N}
\dfrac{2 H^2(x/N)}{\epsilon N A}
\sum_{0\leq k \leq \epsilon N}
E_{\Pam_c^N }
\Big[ g_{x+k,N}(x+k) f(\eta) \Big].
\end{equation}
Notice that the summation of $k$ is ranging from $0\leq k \leq \epsilon N$ instead of $0\leq k \leq \epsilon N-1$.
Now, we set $A = N$.
Putting together \eqref {eqn: H_s derivative estimate} and \eqref{eqn: H_s derivative estimate, ii}, 
we obtain  \eqref{eqn: log exp negative} with $K_0 = C/2$.

\vspace{0.1cm}
{\it Step 3.}
Recall that $H_j$'s have compact support in $[0,T]\times( \T\setminus \D)$. 
Let $\iota_\delta = (2\delta)^{-1} \id_{[-\delta, \delta]}$.
Applying the Bulk Replacement Lemma (Lemma \ref{rplmt_global_lem}) to \eqref{tilde W limit} and taking $N\to\infty$, we  obtain 
\begin{equation*}
\begin{split}
&\limsup_{\delta\to0}E_Q 
\Big[
\max_{1\leq j\leq m} 
\Big\{  \int_0^T  \int_\T 
 H_j(s,x) 
 \epsilon^{-1}
\big(  \Phi(  \iota_\delta*\rho_s(x))
 -
  \Phi(  \iota_\delta*\rho_s(x+\epsilon))
    \big)
 dx  ds\\
&\qquad \qquad\qquad\qquad-
 2\int_0^T  \int_\T H_j^2(s,x) 
 \epsilon^{-1} 
 \int_{x}^{x+\epsilon}
  \Phi(  \iota_\delta*\rho_s(u)) du
  dx  ds
\Big\}
\Big]
\leq K_0,
\end{split}
\end{equation*}
Here $\iota_\delta*\rho_s(\cdot) = \int_\T \iota_{\delta}(x-\cdot) \rho(s,x)dx$
and $\rho(s,x)\in L^1([0,T]\times \T)$ is the absolute continuous part of each limit path $\pi_s$ (cf.~Lemma \ref{lem: abs continuity}).

Sending $\delta\to 0$, applying a discrete integration by parts, and then taking $\epsilon\to 0$, we have
\begin{equation*}
\begin{split}
&E_Q 
\Big[
\max_{1\leq j\leq m} 
\Big\{  \int_0^T  \int_\T \partial_x H_j(s,x) \Phi(\rho(s,x)) dx  ds\\
&\qquad \qquad\qquad\qquad\qquad\qquad-
 2\int_0^T  \int_\T H_j^2(s,x)  \Phi(\rho(s,x))    dx  ds
\Big\}
\Big]
\leq K_0,
\end{split}
\end{equation*}
By monotone convergence, the $\max_{1\leq j\leq m} 
$ above can be replaced by $\max_{1\leq j< \infty}$.
Furthermore, as $H_j$ is dense in $C_c^{0,1}([0,T]\times ( \T\setminus \D))$ with respect to the norm $\|H\|_\infty + \|\partial_x H\|_\infty$, we conclude 
\begin{equation} \label {esti before Rieze rep}
E_Q 
\Big[
\sup_H
\Big\{  \int_0^T  \int_\T \partial_x H(s,x) \Phi(\rho(s,x)) dx  ds
-
 2\int_0^T  \int_\T H^2(s,x)  \Phi(\rho(s,x))    dx  ds
\Big\}
\Big]
\leq K_0,
\end{equation}
where the $\sup$ is over $H\in C_c^{0,1}([0,T]\times ( \T\setminus \D))$.

\vspace{0.1cm}
{\it Step 4.}
As a result of \eqref{esti before Rieze rep}, for $Q$-a.e.\,path $\pi_t(dx)$, there exists $B=B(\pi_t)$ such that, for all $H\in C_c^{0,1}([0,T]\times ( \T\setminus \D))$,
\begin{equation*}
\begin{split}
\int_0^T  \int_\T \partial_x H(s,x) \Phi(\rho(s,x)) dx  ds
-
2 \int_0^T  \int_\T H^2(s,x)  \Phi(\rho(s,x))    dx  ds
\leq
B.
\end{split}
\end{equation*}
Define on $C_c^{0,1}([0,T]\times (\T\setminus \D))$ a linear functional $l(H) := \int_0^T  \int_\T \partial_x H(s,x) \Phi(\rho(s,x)) dx  ds$.
Also define $\|H\|_{2,\rho} := \left( \int_0^T  \int_\T H^2(s,x)  \Phi(\rho(s,x))  dx  ds\right)^{1/2}$.
 Then we have, for all $a\in \R$
\[
a\,l(H) -
2a^2\left( \|H\|_{2,\rho}\right)^2
\leq B.
\]
Maximizing the left hand side over $a\in \R$, we obtain $l(\cdot)$ is a bounded linear functional under the norm $\|\cdot\|_{2,\rho}$. By Riesz representation theorem, there exists $F$ such that $\| F\|_{2,\rho}<\infty$ and it holds that $l(H) = \langle H, F\rangle_\rho
:=\int_0^T  \int_\T H(s,x) F(s,x)  \Phi(\rho(s,x))  dx  ds$.
Define $\partial_x \Phi(\rho(t,x)) := -F(t,x) \Phi(\rho(t,x))$. Then, we have shown that $\Phi(\rho(t,x))$ is weakly differentialble with respect to $x$ on $[0,T]\times ( \T\setminus \D)$ and the weak derivative $\partial_x \Phi(\rho(t,x))$ satisfies \eqref{eqn: energy of Phi_x}. 

To finish, we argue that the weak differentiability  of $x$ may be extended from $\T\setminus \D$ to $\T$ with the same weak derivative $\partial_x \Phi(\rho(t,x))$.
The weak differentialbility on $[0,T]\times (\T\setminus \D)$ implies that, for almost all $t\in [0,T]$, $\Phi(\rho(t,\cdot))$ is absolutely continuous on $\T\setminus \D$.
It suffices to show that, for almost all $t\in [0,T]$, $\Phi(\rho(t,\cdot))$  is continuous on the defects $\D$.
Notice that, as $(\partial_x \Phi(\rho) )^2/  \Phi(\rho)$ and $\Phi(\rho)$ are both in $L^1([0,T]\times \T)$, by Cauchy-Schwarz, we have $\partial_x \Phi(\rho)$ is in $L^1([0,T]\times \T)$ as well.
Then, for almost all $t\in[0,T]$, we have  that the limits
 $\lim_{x\to x_j\pm} \Phi(\rho(t,x))$
exist and are finite for all $x_j\in \D$
(cf.~Remark \ref{rmk on w-deriv}). 
Moreover, by Lemma \ref{local replace, slow site}, the left and right limits match.
Therefore, we conlude the continuity of 
$\Phi(\rho(t,\cdot))$ on $\D$, 
finishing the proof.
\end{proof}

\section{Uniqueness} \label{section: uniqueness}
In this section, we present the uniqueness of the weak solutions to equations \eqref{pde: k alpha} and \eqref{pde: bounded g}.
The proof is based on an energy argument (cf.~\cite{JLT}). 
Recall Definitions  \ref{def: weak sln} and \ref{def: weak sln, bounded g}.

\begin{Thm} \label {thm: uniqueness}
There exists at most one weak solution to 
\eqref{pde: k alpha}.
\end{Thm}
\begin{proof}
Let $\pi^{(1)}_t$ and $\pi^{(2)}_t$ be two weak solutions of \eqref{pde: k alpha} such that
$\pi^{(i)}_t = \rho_i(t,x)dx + \sum_{j\in J_c} \m^{(i)}_j(t) \delta_{x_j}(dx)$ for $i=1,2$.
As $\m^{(i)}_j(t) =\big[ \lambda_j
 \Phi(\rho_i (t,x_j))\big]^{1/\alpha}$,
to prove the theorem, it suffices to show $\rho_1 = \rho_2$ for almost all $t$.

\vspace{0.1cm}
{\it Step 1.}
Define 
 $\overline \Phi(t,x) := \Phi(\rho_1(t,x)) -  \Phi(\rho_2(t,x))$.
By Definition \ref{def: weak sln},
we have that $\overline \Phi$ is weakly differentiable with respect to $x\in \T$ and 
$\overline \partial_x \Phi$  (and therefore 
$\overline\Phi$) is in $L^2([0,T]\times \T)$.
As $\Phi(\rho_i(t,x_j)) = \Phi(c_0)$ for each $i=1,2$ and $j\in J_s$, we have $\overline \Phi(t,x_j) = 0$ for for all $j\in J_s$.
We first show that $\overline\Phi$ can be approximated `well' by a sequence $\{\overline \Phi_\e\}_{\e>0}$ in the sense that (1) $\overline \Phi_\e$ is smooth and compactly supported on  $[0,T]\times (\T\setminus \D_s)$; 
(2) $\overline \Phi_\e \to \overline \Phi$ and $\partial_x\overline \Phi_\e \to \partial_x \overline \Phi$ in $L^2 ([0,T]\times \T)$;
and (3) 
$\overline \Phi_\e (t,x_j) \to \overline \Phi (t,x_j)$ in $L^2[0,T]$ for all $j\in J_c$ as $\e\to 0$.

Indeed, to verify this approximation, for $\delta>0$, let $\D_{s}^\delta = \cup_{j\in J_s} (x_j-\delta,x_j+\delta)$.
We define $F_\delta$ be such that 
$F_\delta (t,x)= 0$ on $[0,T]\times \D_s^{\delta}$ and
$F_\delta (t,x)=\overline \Phi(t,x)$ on $[0,T]\times (\T \setminus \D_s^{2\delta})$. 
For $(t,x)\in \D_s^{2\delta} \setminus \D_s^{\delta}$, define
\[
F_\delta(t,x)=
\begin{cases}
\overline \Phi(t,2x-x_j-2\delta) 
&{\rm on \ } [0,T]\times [x_j+\delta,x_j+2\delta),\\
\overline \Phi(t,2x-x_j+2\delta) 
& {\rm on \ } [0,T]\times (x_j-2\delta,x_j-\delta].\\
\end{cases}
\]
We note that $F_\delta$ has compact support in $[0,T]\times (\T\setminus \D_s)$.
It is easy to check that $F_\delta$ and $\partial_x F_\delta$ approximate $\overline \Phi$ and $\partial_x\overline \Phi$ in $L^2([0,T]\times \T)$ respectively.
Since $F_\delta$ and $\overline \Phi$ match at $x=x_j$ for all $j\in J_c$,
to find a desired sequence $\{\overline \Phi_\e\}$,
it suffices to show that, for each small $\delta$, there exists $\{F_{\delta,\e}\}$ that approximates $F_\delta$ `well'.
  
To this end, let $\tau_\e(x)$ be a standard mollifier supported on $[-\e,\e]$.
With $F_\delta (t,x)$ extended to be $0$ for $t\notin [0,T]$, we define
 \[
F_{\delta,\e}(t,x) 
:=
\int_\R \int_\T 
F_\delta (t-s,x-u) 
\tau_\e(s) \tau_\e(u) \, du ds .
\]
Notice that $F_{\delta,\e} \in C_c^\infty (\R\times 
(\T\setminus \D_s))$  for $\e<\delta$.
When restricted on $[0,T]\times \T$,
it is standard that $\overline F_{\delta,\e}$ and  
$\partial_x F_{\delta,\e}$ approximates
$ F_{\delta}$ and $\partial_x F_{\delta}$ respectively in $L^2 ([0,T]\times \T)$ as $\e\to 0$.
It remains to show
that $F_{\delta,\e} (t,x_j)$
approximates $F_{\delta} (t,x_j)$
 in $L^2[0,T]$.
Write that
\begin{align*}
&\int_0^T  
\left( F_{\delta,\e} (t,x_j)
-
F_{\delta} (t,x_j) \right)^2
dt\\
&\ \ \ \ \ \ \ =
\int_0^T  
\Big[\int_\R \int_\T
\left( F_{\delta} (t-s,x_j-u) 
-
F_{\delta} (t,x_j)\right) \tau_\e(s) \tau_\e(u) duds
\Big]^2
dt.
\end{align*}
By adding and subtracting 
$F_{\delta} (t-s,x_j)$, the above is bounded above by $I_1 + I_2$ where
\[
\begin{split}
I_1 &:= 
2 \int_0^T  
\Big[\int_\R \int_\T
\left(F_{\delta} (t-s,x_j-u)
-
F_{\delta} (t-s,x_j)\right) \tau_\e(s) \tau_\e(u) duds
\Big]^2 dt,\\
I_2 &:=
2 \int_0^T  
\Big[ \int_\R
\left(
F_{\delta} (t-s,x_j)
-
F_{\delta} (t,x_j) \right)
 \tau_\e(s)ds
\Big]^2
dt
\end{split}
\]
As $\int_\R
F_{\delta} (t-s,x_j) \tau_\e(s)ds$ approximates $F_{\delta} (t,x_j)$ in $L^2[0,T]$, the term $I_2$ vanishes as $\e\to 0$.
For the term $I_1$, using 
$F_{\delta} (t-s,x_j-u) -
F_{\delta} (t-s,x_j) = \int_{x_j}^{x_j-u} \partial_x F_{\delta} (t-s,x)dx$, we have
\[
I_1
\leq
2 \int_0^T  
\int_\R \int_\T
\Big( \int_{x_j}^{x_j-u} \partial_x F_{\delta} (t-s,x) dx \Big)^2 \tau_\e(s) \tau_\e(u) duds
 dt
 \]
The square term above is further bounded by $u^2 \big| \int_{x_j}^{x_j-u} (\partial_x F_\delta(t,x) )^2 dx\big| $.
As $\tau_\e(u)=0$ for $u\notin [-\e,\e]$,  we have that $I_1\leq
2\e^2 \int_0^T \int_\T (\partial_x F_\delta(t,x))^2 dx dt$ which vanishes as $\e\to0$.

\vspace{0.1cm}
{\it Step 2.}
We now proceed to the uniqueness of weak solutions. 
Let $\overline \rho := \rho_1 - \rho_2$ and
 $\overline \m_j (t) := \m_j^{(1)}(t) - \m_j^{(2)}(t)$ for each $j\in J_c$.
As $\pi^{(1)}_t$ and $\pi^{(2)}_t$ both satisfy \eqref{weak sln, def}, we have,
for all $G(t,x) \in C_c^\infty \left( [0,T)\times (\T \setminus \D_s )\right)$,
\begin{equation} \label {eqn: rho1-rho2}
\begin{split}
&\int_0^T \int_\T \partial_tG(t,x) \overline \rho(t,x) dx dt
+
\sum_{j\in J_c}
\int_0^T \partial_t G(t,x_j)
\overline \m_j (t)
  dt
=\int_0^T \int_\T 
\partial_x G(t,x)
 \partial_x \overline\Phi(t,x) 
 dx dt.
\end{split}
\end{equation}
Let $\overline \Phi$ be  approximated `well' by some $\{\overline\Phi_\e\}$ as in Step $1$.
Taking $G(t,x) =-\int_t^T \overline \Phi_\e(s,x)ds$ and then leting $\e\to 0$, we obtain
\begin{equation}
\label{energy eqn, critical}
\begin{split}
&\int_0^T \int_\T \overline \Phi (t,x) \overline \rho(t,x) dx dt
+
\sum_{j\in J_c}
 \int_0^T\overline \Phi(t,x_j)
\overline \m_j (t)
  dt\\
 &\qquad\qquad\qquad\qquad
 = - \int_\T  \int_0^T 
 \Big[\int_t^T
\partial_x \overline\Phi (s,x)ds \Big]
 \partial_x \overline\Phi(t,x) dt dx.
 \end{split}
\end{equation}
The right hand side of the above is computed as $-\dfrac 12 \int_\T 
 \big[ \int_0^T
 \partial_x \overline\Phi(t,x) dt 
 \big]^2 dx \leq 0$.
However, for the left hand side, we have $\overline \Phi (t,x) \overline \rho(t,x)\geq 0$ and $\overline \Phi(t,x_j)
\overline \m_j (t)\geq 0$ for all $t$, $x$, and $j$.
Then, we deduce that $ \int_0^T \int_\T \overline \Phi (t,x) \overline \rho(t,x) dx dt =0$
which implies $\overline\Phi(t,x) \overline\rho(t,x) = 0$ a.e.\ 
and, therefore, $\overline \rho(t,x) = 0$
a.e.  Since $\overline \rho$ is continuous in $x$ for almost all $t$ (cf. Remark \ref{rmk on w-deriv}),
the theorem is proved.
\end{proof}

\begin{Thm} \label {thm: uniqueness, bounded g}
There exists at most one weak solution to 
\eqref{pde: bounded g}.
\end{Thm}
\begin{proof}
For $i=1, 2$, let $\pi^{(i)}_t = \rho_i(t,x)dx + \sum_{j\in J_c} \m^{(i)}_j(t) \delta_{x_j}(dx)$
be two weak solutions to \eqref{pde: bounded g}. 
Notice that $\m^{(i)}_j(t) = \m^{(i)}_j(t) \id_{\rho(t,x_j)=c_j }$ 
and $\rho(t,x_j) \leq c_j$ implies $\overline \Phi(t,x_j)
\overline \m_j (t)\geq 0$ where 
$\overline \Phi:= \Phi(\rho_1) -  \Phi(\rho_2)$ and $\overline \m_j := \m_j^{(1)} - \m_j^{(2)}$.
Following the proof of Theorem \ref {thm: uniqueness}, we have $\overline\rho:= \rho_1- \rho_2=0$ for almost all $t$.
To conclude, it remains to show $ \overline \m_j = 0$ for each $j\in J_c$ for almost all $t$.
To this end, we fix any $j\in J_c$ and take $F\in C^\infty (\T)$ such that $F(x_j) =1$ and $\text{supp}\, F \cap J_c = x_j$. Letting $G(t,x) = \int_t^T h(s) F(x)ds$ in \eqref{eqn: rho1-rho2}
for any $h(t)\in C_c^{\infty} (0,T)$ 
and using $\overline \rho = \overline\Phi=0$,
 we obtain 
$\int_0^T h(t) \overline \m_j (t) dt = 0$, and therefore, $\overline \m_j (t) = 0$ for almost all $t$, finishing the proof.
\end{proof}

\medskip
\noindent{\bf Acknowledgements.} 
SS was partly supported by
ARO-W911NF-18-1-0311.


\end{document}